\documentclass{amsart}

\usepackage[utf8]{inputenc}

\usepackage[style=ext-numeric,
           maxbibnames=99,  
           maxcitenames=99,  
           giveninits=true, 
           sortcites=true,  
           backend=biber]{biblatex}
\DeclareFieldFormat
  [article,inbook,incollection,inproceedings,patent,thesis,unpublished]
  {titlecase:title}{\MakeSentenceCase*{#1}}

\renewrobustcmd*{\bibinitdelim}{\,}

\usepackage{amssymb, amsmath, amsthm}
\usepackage{bbold}
\usepackage{graphicx}
\usepackage[shortlabels]{enumitem}
\usepackage{amsfonts}
\usepackage{mathtools}
\usepackage{orcidlink}

\usepackage{xcolor}
\definecolor{lightgreen}{RGB}{200,200,34}
\definecolor{mediumgreen}{RGB}{120,160,34}
\definecolor{darkgreen}{RGB}{59, 122, 87}
\definecolor{orangered}{RGB}{255,66,14}
\definecolor{coordinateblue}{RGB}{0,100,200}
\definecolor{coordinatepink}{RGB}{250,50,0}
\definecolor{neonpink}{RGB}{255,85,255}
\definecolor{lightblue}{RGB}{20,140,190}
\definecolor{mediumblue}{RGB}{72,108,180}
\definecolor{darkblue}{RGB}{16,24,143}
\definecolor{birsanblue}{RGB}{200,100,125}
\definecolor{3dcolor}{RGB}{135,104,150}
\usepackage{mathtools}
\usepackage{colonequals}
\usepackage{lmodern}
\usepackage{mathrsfs}
\usepackage{setspace} 
\usepackage{pgfplots}
\usepgfplotslibrary{groupplots}
    
\usepackage{microtype}
\usepackage{todonotes}

\usepackage{esdiff}
\usepackage{float}

\usepackage{tikz}
\usetikzlibrary{arrows}

\usepackage{multirow}
\usepackage{changes}
\usepackage{hyperref}

\usepackage{csvsimple}

\newcommand{\R}{\mathbb R}

\newcommand{\Ee}{{\boldsymbol{E}^e}}
\newcommand{\Eek}{{\boldsymbol{E}^e_k}}
\newcommand{\Ke}{{\boldsymbol{K}^e}}
\newcommand{\Kek}{{\boldsymbol{K}^e_k}}
\newcommand{\Qe}{\boldsymbol{Q}_e}
\newcommand{\Qeh}{\boldsymbol{Q}_{e,h}}

\newcommand{\hatEe}{\widehat{\boldsymbol{E}}{\mathstrut}^e}
\newcommand{\hatKe}{\widehat{\boldsymbol{K}}{\mathstrut}^e}

\newcommand{\bQ}{\boldsymbol{Q}}
\newcommand{\bR}{\boldsymbol{R}}

\newcommand{\ba}{{\boldsymbol{a}}}
\newcommand{\bb}{{\boldsymbol{b}}}
\newcommand{\bc}{{\boldsymbol{c}}}
\newcommand{\bd}{{\boldsymbol{d}}}
\newcommand{\be}{\boldsymbol{e}}
\newcommand{\boldf}{\boldsymbol{f}}
\newcommand{\bm}{\boldsymbol{m}}
\newcommand{\bn}{\boldsymbol{n}}
\newcommand{\bt}{\boldsymbol{t}}
\newcommand{\bv}{\boldsymbol{v}}
\newcommand{\bw}{\boldsymbol{w}}

\newcommand{\ovr}{\overline{R}}

\newcommand{\axl}{\operatorname{axl}}
\renewcommand{\skew}{\operatorname{skew}}
\newcommand{\sym}{\operatorname{sym}}
\newcommand{\tr}{\operatorname{tr}}
\newcommand{\dev}{\operatorname{dev}}
\newcommand{\I}{{\operatorname{I}}}
\newcommand{\II}{{\operatorname{I\!I}}}

\newcommand{\identity}{\mathbb{1}}

\newcommand{\abs}[1]{\lvert #1 \rvert}
\newcommand{\norm}[1]{\lVert #1 \rVert}
\newcommand{\Tref}{{T_\textnormal{ref}}}

\DeclareMathOperator{\dist}{dist}
\DeclareMathOperator{\SOdrei}{SO(3)}
\DeclareMathOperator*{\argmin}{arg\,min}

\newtheorem{definition}{Definition}
\newtheorem{theorem}[definition]{Theorem}

\newtheorem{lemma}[definition]{Lemma}
\newtheorem{remark}[definition]{Remark}
\newtheorem{problem}[definition]{Problem}

\theoremstyle{remark}

\graphicspath{{gfx/}}

\title[Finite elements for non-orientable Cosserat shells]
      {A geometrically nonlinear Cosserat shell model for orientable and non-orientable surfaces: Discretization with geometric finite elements}

\author[Nebel]{Lisa Julia Nebel}
\address{Lisa Julia Nebel\\
Technische Universität Dresden\\
Institut für Numerische Mathematik\\
Zellescher Weg 12--14\\
01069 Dresden\\
Germany \\
\orcidlink{0000-0002-7200-0312}~\href{https://orcid.org/0000-0002-7200-0312}{0000-0002-7200-0312}}
\email{lisa\_julia.nebel@tu-dresden.de}

\author[Sander]{Oliver Sander}
\address{Oliver Sander\\
Technische Universität Dresden\\
Institut für Numerische Mathematik\\
Zellescher Weg 12--14\\
01069 Dresden\\
Germany \\
\orcidlink{0000-0003-1093-6374}~\href{https://orcid.org/0000-0003-1093-6374}{0000-0003-1093-6374}}
\email{oliver.sander@tu-dresden.de}

\author[Bîrsan]{Mircea Bîrsan}
\address{Mircea Bîrsan\\
Universität Duisburg--Essen\\
Lehrstuhl für Nichtlineare Analysis und Modellierung \\
Fakultät für Mathematik,
Thea-Leymann Str.\,9 \\
45127 Essen \\
Germany \\
and Department of Mathematics, University ``A.I. Cuza'' of Iaşi, 700506 Iaşi, Romania \\
\orcidlink{0000-0002-1360-4044}~\href{https://orcid.org/0000-0002-1360-4044}{0000-0002-1360-4044}}
\email{mircea.birsan@uni-due.de}

\author[Neff]{Patrizio Neff}
\address{Patrizio Neff\\
Universität Duisburg--Essen\\
Lehrstuhl für Nichtlineare Analysis und Modellierung \\
Fakultät für Mathematik \\
Thea-Leymann Str.\,9 \\
45127 Essen \\
Germany \\
\orcidlink{0000-0002-1615-8879}~\href{https://orcid.org/0000-0002-1615-8879}{0000-0002-1615-8879}}
\email{patrizio.neff@uni-due.de}

\thanks{This research has been funded by the Deutsche Forschungsgemeinschaft (DFG, German Research Foundation)
–- Project: SA~2130/6-1 (L.\,Nebel and O.\,Sander) and Project no. 415894848, NE~902/8-1 (P.\,Neff) and BI~1965/2-1 (M.\,Bîrsan).}

\begin{document}

\begin{abstract}
 We investigate discretizations of a geometrically nonlinear elastic Cosserat shell
 with nonplanar reference configuration originally introduced by \citeauthor{neff_dim_reduction:2019} in~\citeyear{neff_dim_reduction:2019}.
 The shell model includes curvature terms
 up to order~5 in the shell thickness, which are crucial to reliably simulate
 high-curvature deformations such as near-folds or creases. The original model
 is generalized to shells that are not homeomorphic to a subset of $\R^2$.
 For this, we replace the originally planar parameter domain by an abstract two-dimensional manifold,
 and verify that the hyperelastic shell energy and three-dimensional reconstruction are invariant
 under changes of the local coordinate systems. This general approach allows to
 determine the elastic response for even non-orientable surfaces like the Möbius strip and the Klein bottle.
 We discretize the model with a geometric finite element method and,
 using that geometric finite elements are $H^1$-conforming, prove that the discrete shell model has a
 solution.  Numerical tests then show the general performance and versatility
 of the model and discretization method.
\end{abstract}



\maketitle

\tableofcontents

\section{Introduction}

In~\cite{neff_dim_reduction:2019,neff_derivation:2020}, the authors introduced a physically linear but geometrically nonlinear elastic Cosserat shell model for shells with a curved reference configuration.
The model was derived by dimensional reduction of a three-dimensional Cosserat continuum model,
and it is a direct generalization of the flat Cosserat
shell model of~\textcite{neff:2004}.
In later papers, different variants of the model were discussed, and justified by
derivation~\cite{birsan:2020,birsan:2021} or $\Gamma$-convergence arguments~\cite{saem_ghiba_neff:2022}.

Configurations of this
Cosserat shell consist of the total deformation $\bm : \omega \to \R^3$, which maps a two-dimensional parameter domain (the ``fictitious domain'' in the parlance of \cite{neff_dim_reduction:2019}) into three-dimensional Euclidean space,
and an independent field of microrotations $\Qe : \omega \to \SOdrei$
(where $\SOdrei$ is the special orthogonal group, i.e., the group of orthogonal
$3\times 3$ matrices with determinant $1$), which describes transverse shear
and local drilling of the shell.
The material behavior is given as a hyperelastic energy functional
\begin{equation*}
 I(\bm, \Qe)
 \colonequals
 \int_\omega \big[W_\text{memb}(\Ee, \Ke) + W_\text{bend}(\Ke) \big] \,d\omega +
\textnormal{external loads},
\end{equation*}
which features terms up to order~5 in the shell thickness.
The matrix fields $\Ee$ and $\Ke$ are strain and curvature measures of the shell, depending
on the deformation and microrotation, as well as on the geometry of an assumed stress-free reference
configuration $\bm_0$. The energy density neatly
separates the material coefficients of the original three-dimensional model
from geometric properties of the stress-free configuration.
The model is geometrically nonlinear but physically linear, which means that it is frame-indifferent and
allows for large rotations but only small elastic strains.
Nevertheless, existence of minimizers
in the space $H^1(\omega, \R^3) \times H^1(\omega, \SOdrei)$ has been shown~\cite{neff_existence:2020}.
This sets the model apart from other geometrically nonlinear shell models that
combine membrane and bending effects, but lack an existence proof.
The provable existence of minimizers is an important justification for the use of Cosserat shell models.

In the original shell model in~\cite{neff_dim_reduction:2019,neff_derivation:2020} the parameter domain $\omega$ was an open set in $\R^2$. The model was therefore restricted to shells that are homeomorphic to such sets,
which excluded even simple geometries such as spheres. In this paper we reformulate
the model for more general topologies.  Following ideas from~\cite{marsden_hughes:1983},
we replace the flat parameter domain $\omega$ of~\cite{neff_dim_reduction:2019,neff_derivation:2020} by an abstract two-dimensional manifold, again called $\omega$.
To address points in this manifold we introduce local coordinate systems, which map open sets of $\omega$
homeomorphically to subsets of $\R^2$, as is the standard construction in differential geometry.
The single flat parameter (``fictitious'') domain of~\cite{neff_dim_reduction:2019,neff_derivation:2020} is hence replaced by a set of subsets of $\R^2$, which serve as local coordinates. The situation of~\cite{neff_dim_reduction:2019,neff_derivation:2020} is recovered when the parameter domain $\omega$ can be covered by a single coordinate chart.

Configurations of the two-dimensional shell surface are given as immersions $\bm$ of $\omega$ into $\R^3$ with associated microrotation fields $\Qe:\omega \to \SOdrei$.  The immersion $\bm : \omega \to \R^3$ represents the shape
of the shell surface, and the microrotation $\Qe$ is a local rotation of the shell with respect to its
configuration in the natural, i.e., stress-free, reference state. This stress-free state is given by a second
immersion $\bm_0 : \omega \to \R^3$, which is part of the problem formulation. Unlike in models
of non-Euclidean elasticity~\cite{efrati_sharon_kupferman:2009,lewicka_mahadevan:2022,lewicka:2023} we explicitly assume the existence
of such a stress-free configuration, but note that a large part of this paper would carry over
to the more general non-Euclidean case, too.

The approach of using a general manifold~$\omega$ as the parameter domain may appear
unnecessarily abstract at first sight. However, it seems conceptually cleaner to us
than using the reference surface $\bm_0(\omega)$ (a subset of $\R^3$) for parametrization.
Indeed, unlike in the original derivation in~\cite{neff_dim_reduction:2019,neff_derivation:2020},
the map $\bm_0(\omega) \to \bm(\omega)$ (which corresponds to the restriction
of the map $\varphi_\xi : \Omega_\xi \to \Omega_c$ to the two-dimensional set
$\omega_\xi \subset \Omega_\xi$ there) never plays a role.
Also, note that we explicitly
allow the stress-free configuration $\bm_0(\omega)$ to have self-intersections,
which also makes it unsuitable as a parameter domain.
For physically meaningful results, of course, the shell surface may not self-intersect,
i.e., the configuration map $\bm$ has to be an embedding rather than an immersion.
However, when investigating non-orientable parameter manifolds~$\omega$ it turned out
that our construction can be effortlessly formulated for the more general case of immersions as well.
We therefore write everything in terms of immersions, which then makes the model cover even non-physical but interesting objects like the Klein bottle.

Models of shells with curved reference configurations exist in the literature,
but like \cite{neff_dim_reduction:2019,neff_derivation:2020} they
never explicitly discuss
objects with non-trivial topology~\cite{pimenta_campello:2009,duong_roohbakhshan_sauer:2017}.
An alternative general shell model able to handle curved reference configurations
and complicated topology is the 6-parameter shell. An account of this approach
has been presented, e.g.,  in the book of \textcite{libai_simmonds:1998}.
Other shell models
with curved reference configurations and large rotations are analyzed in the papers
\cite{burzynski_chroscielewski_daszkiewicz_witkowski:2016,burzynski_chroscielewski_witkowski:2016,chroscielewski_kreja_sabik_witkowski:2011,chroscielewski_witkowski:2011},
and in the book~\cite{wisniewski:2010}.
From a kinematical point of view,
the 6-parameter shells are equivalent to Cosserat shells, since both models involve
the deformation and a microrotation as independent variables. The difference to our approach
consists in the constitutive assumptions and the relation
to the three-dimensional theory. In contrast to our work, the papers
on 6-parameter shells assume the parent three-dimensional model to be a Cauchy continuum
(i.e., without microrotations). Also, their assumed constitutive relations are relatively simple,
since constitutive coefficients do not depend on the curvature of the reference surface.
Further details on geometrically nonlinear shells and their derivation
can be found in the monograph~\cite{steigmann_birsan_milani:2023}.

{Nontrivial topologies are covered by the Cosserat model of \textcite{sansour_bednarczyk:1995}.
Unlike our model, though, which is derived by a consistent dimensional reduction
from a three-dimensional Cosserat model, \citeauthor{sansour_bednarczyk:1995}
use the direct approach and postulate directly that the shell is a two-dimensional Cosserat continuum.
This approach is simpler than ours, but the relationship between the shell model
and the actual three-dimensional model remains obscure. In particular, with the derivation approach
we obtain the constitutive coefficients for shells and the expression of the shell strain energy density
in terms of quantities of the three-dimensional model, which are much easier to obtain in practice.
Also, by our derivation approach we are able to determine higher order terms
in the shell thickness, which are important to improve the accuracy of numerical solutions.}

\medskip

{While the first part of this manuscript is independent of any discretization, the} 
second part deals with finite element discretizations of the presented Cosserat model.
For such discretizations, we equip the shell parameter surface~$\omega$ with a triangulation,
which allows to express all integrals as sums of integrals over a reference triangle.
This corresponds to the usual practice when constructing finite element models
of curved shells, but is hardly ever spelled out in any detail.
Finite element spaces are defined with respect to the triangulation of~$\omega$.
In particular, we represent configurations of the shell by Lagrange finite element
functions $\bm_h : \omega \to \R^3$. The given stress-free configuration $\bm_0 : \omega \to \R^3$
can in principle be represented in any manner, but will most frequently also be represented
by a finite element function. While piecewise linear finite elements are possible,
they neutralize some of the advantages of the shell model presented here, because
all terms involving the curvature of $\bm_0(\omega)$ would then vanish.
We will therefore mostly use Lagrange finite elements of second order.

The discretization of the microrotation field $\Qe$ requires additional attention.
Problems with directional or orientational degrees of freedom such as the field of
microrotations $\Qe:\omega \to \SOdrei$ are difficult to treat numerically,
because spaces of functions mapping into a nonlinear set such as $\SOdrei$ or the unit sphere $S^2$
cannot form vector spaces. Consequently, approximations by vector spaces such as spaces of piecewise polynomials
are not conforming, in the sense that the image of the approximating map is not contained
in $\SOdrei$.
Seen from another direction, the problem is that there are no nontrivial $\R^{3 \times 3}$-valued polynomials with values in $\SOdrei$.
Various ad hoc approaches for discretizations of
$\SOdrei$- or $S^2$-valued fields exist in the literature, each with its own
strengths and weaknesses.
Overviews can be found, e.g., in~\cite{kuo-mo_yeh-ren:1989,mueller_bischoff:2022,romero:2004,wisniewski:2010}.

To take an example, discretizations based on interpolating Euler angles
such as~\cite{wriggers_gruttmann:1993,gruttmann_wagner_meyer_wriggers:1993}
are straightforward
to construct, but they are plagued by
coordinate singularities near certain configurations, and they are therefore only usable for
moderate rotations.  Also, the discrete models do not inherit the frame indifference of
the continuum model.
Similar problems exist for methods that interpolate between values by lifting them onto
a fixed tangent space and interpolate there~\cite{mueller:2009,muench:2007,pimenta_campello:2009,sansour_bednarczyk:1995}.
To avoid large distortions, some methods such as~\cite{muench:2007} switch between several tangent spaces.

Methods that repeatedly average between pairs of orientations can be interpreted
as generalizations of spline functions~\cite{areias_rabczuk_dias-da-costa:2013,hardering_sander:2020,absil_gousenbourger_striewski_wirth:2016}.
They lead to objective and path-independent formulations, but suffer from
spurious dependencies of the simulation results on the node ordering.
A more difficult approach, originally proposed by \citeauthor{simo_fox:1989} \cite{simo_fox:1989,simo_fox_rifai:1990}, interpolates only
the corrections of the Newton method used to solve the shell equilibrium equations.
As elements of a tangent space, these corrections can be approximated by piecewise polynomials.
For this to work, the values of the nonlinear variables have to be stored as history variables at the quadrature points.
The method mixes discretization and solver algorithm, which makes it
difficult to analyze. Unfortunately, it also leads to a dependence of the discrete solution
on the load path. This was originally shown by \textcite{crisfield_jelenic:1999} for the rod
model of \textcite{simo_vu-quoc:1986}. In the context of isogeometric analysis,
\cite{dornisch_klinkel_simeon:2013,dornisch_klinkel:2014} have used a similar approach interpolating
with NURBS functions.
{Recently, Magisano et al.\ \cite{magisano_leonetti_madeo_garcea:2020} have proposed a method
that combines interpolating the Newton increments
with a corotational approach, which does indeed lead to a scheme that has all desirable properties.}

One further possibility is the use of nonconforming discretizations.  These use
standard $\R^{3 \times 3}$-valued piecewise polynomials for the approximation of the
microrotation field, and enforce the restriction to $\SOdrei$ only
at the Lagrange points.  For Cosserat beams this is mentioned
in~\cite{romero:2004,betsch_steinmann:2002}.
For Reissner--Mindlin shells (i.e., shells with only one director),
the corresponding technique is used, e.g., in \cite{hughes_liu:1981} and follow-up work.
Such a discretization is simple, singularity-free, and preserves frame indifference.
However, the models need to be modified (explicitly or implicitly) to account for the non-orthogonality
of the microrotation field away from the Lagrange points.
Outside of shell theory, nonconforming discretizations have been analyzed mathematically
for maps into the unit sphere $S^2$~\cite{bartels_prohl:2007,alouges_jaisson:2006}.
While the original works considered only first-order finite elements
and showed only weak convergence results,
quasioptimal convergence could be shown recently for harmonic maps~\cite{bartels_palus_wang:2022}
and harmonic map heat flow~\cite{bartels_kovacs_wang:2022}, even for finite elements
of approximation order larger than~1.
A numerical study also testing such higher-order discretizations is given in~\cite{bartels_boehnlein_palus_sander:2022}.

Recently, geometric finite elements (GFEs) have emerged as an elegant and robust way
to discretize the nonlinear vector and orientation fields appearing in geometrically nonlinear
director shell models \cite{sander:2010,sander:2012,sander:2013,grohs:2013}.
They are based on generalizations of polynomial interpolation formulas to data in
non-Euclidean spaces. Indeed, unlike most previous discretization approaches,
GFE methods define actual (nonlinear) spaces of finite element functions.
This makes their construction and behavior more transparent, and it allows for rigorous
analytical investigations similar to the classical Euclidean finite element theory.
Indeed, optimal $L^2$ and $H^1$ interpolation error bounds have been shown in~\cite{grohs_hardering_sander_sprecher:2019,grohs_hardering_sander:2015,hardering:2018Arxiv,hardering:2018}
for finite element functions of any order, along with discretization error bounds for harmonic maps.

Various ways to generalize polynomial interpolation to manifold-valued data have been proposed
in the literature~\cite{hardering_sander:2020}. In this paper we use geodesic~\cite{sander:2013}
and projection-based finite elements~\cite{grohs_hardering_sander_sprecher:2019}.
Both allow for approximation functions of arbitrary order, and preserve the frame-indifference of the continuous models.
Also, both types of finite element
functions are first-order Sobolev functions. This makes analytical investigations
much easier than for competing approaches. A case in point for this claim is Chapter~\ref{sec:existence}
of this manuscript, where we give a proof for the existence of finite element solutions for the nonlinear
shell problem, reusing considerable parts of the existence proof of~\cite{neff_existence:2020}
for the continuous model.

Geometric finite elements have already
been used successfully for planar Cosserat shells~\cite{sander_neff_birsan:2016},
and for Cosserat rods with initial curvature~\cite{sander:2010}.
\textcite{mueller_bischoff:2022,grohs:2013} extended the concept to spline approximation functions,
and the latter used it to discretize the 1-director shell model of \textcite{simo_fox:1989}
(i.e., the geometrically nonlinear Reissner--Mindlin model).
\textcite{romero:2004} lists projecting onto $\SOdrei$, i.e., projection-based finite elements,
as one way to discretize Cosserat beams.
The discretizations have also appeared in experimental studies of the wrinkling behavior
of coated substrates~\cite{knapp_et_al:2021,ghosh_et_al:2021}.

The algebraic formulation of {finding stable configurations of} the discrete shell is a minimization problem on
the product space $\R^{3N_1} \times \SOdrei^{N_2}$, where $N_1$ and $N_2$ are the numbers of Lagrange nodes
used for discretizing the deformation and microrotation, respectively. This space is
a $3N_1 + 3N_2$-dimensional Riemannian manifold. We use a Riemannian trust-region
algorithm to solve this minimization problem~\cite{absil_mahony_sepulchre:2008}. This is a globalized
Newton method replacing each Newton step with a quadratic minimization problem
subject to a convex inequality constraint. {It converges for any initial iterate,
while retaining the fast local convergence of traditional Newton methods.}
Standard trust-region methods work only for energies defined on
Euclidean spaces, however \cite{absil_mahony_sepulchre:2008} presents a generalization
to energies on Riemannian manifolds. The correction problems of this generalization
are quadratic minimization problems on the (linear) tangent spaces of $\R^{3N_1} \times \SOdrei^{N_2}$,
again with a convex constraint.
To make these large constrained problems feasible, we choose a formulation where the convex constraint consists of separate bound constraints for the individual degrees of freedom. The resulting quadratic minimization problems can then be solved with a monotone multigrid method,
as explained in~\cite{sander:2012}. The challenging computations of the
tangent matrices of the energy are done using {the reverse mode of} the automatic differentiation
software ADOL-C~\cite{walther_griewank:2012},
{which in our experience outperforms automatic differentiation approaches
based on alternative number types.}

When constructing discretizations of shells and plates, locking is always an issue.
For objects with a planar stress-free configuration, (shear) locking
is determined by the finite elements used for the deformation and the microrotation fields.
As it turns out, if the stress-free geometry is curved, then the approximation of its geometry comes into
play as well. Unfortunately, the rigorous understanding and treatment of locking is still out of reach
for discretizations of geometrically nonlinear shell models. In this manuscript
we therefore only perform numerical tests.  We observe that the proposed discretization
does not exhibit shear locking if the geometry discretization is at least of second order.
This is consistent with, and generalizes, our previous results for shells with a planar
stress-free configuration~\cite{sander_neff_birsan:2016}.

\bigskip

This article presents the model and the discretization, and shows a set of numerical tests. Chapter~\ref{sec:continuous_model} presents the generalized shell model. We do not derive it from a three-dimensional model, but we show in detail how it relates to its parent model~\cite{neff_dim_reduction:2019,neff_derivation:2020} for simple topologies.
Chapter~\ref{sec:discretization} then recalls the geometric finite element method for the
approximation of microrotation fields.  Chapter~\ref{sec:discete_and_algebraic_problems} presents the
discretized shell problem, and proves rigorously that solutions (possibly non-unique)
do exist for both geodesic and projection-based finite elements, and for any approximation order.

The article then shows five numerical examples.  With the first one,
we systematically investigate the model response and locking behavior as they depend on the approximation order for the deformation field, the microrotations, and the reference surface geometry.
In a further sequence of tests, we then compare the simulation with an actual three-dimensional shell.
We also do a comparison with a variant of the shell model recently proposed by \textcite{birsan:2021}.
Further tests show the behavior of shells with a complex topology undergoing large rotations,
and for shells that buckle.
Finally, to show that the model can properly handle non-orientable shell surfaces, we compute
equilibrium configurations of a Möbius strip and a Klein bottle subject to a volume load.

\section{Cosserat shell model with general topology}
\label{sec:continuous_model}

In this first section we introduce the Cosserat shell model.
It is a generalization of a model originally derived
in~\cite{neff_dim_reduction:2019, neff_derivation:2020, birsan:2021}.
While the original model only allowed
for shells that are diffeomorphic to a domain in $\R^2$,
the new one covers more general topologies.

The original model was introduced twice (in \cite{neff_dim_reduction:2019} and \cite{neff_derivation:2020}) using different notations. We mainly follow the tensor notation of~\cite{neff_dim_reduction:2019}, but clarify the connection to the matrix notation in \cite{neff_derivation:2020} when appropriate.
Boldface letters are used for vectors and tensors. When Greek letters are used as indices,
they always range over the set $\{1, 2\}$, and Einstein summation is used.

\subsection{The shell surface and its extrinsic geometry}
\label{sec:shell surface_geometry}

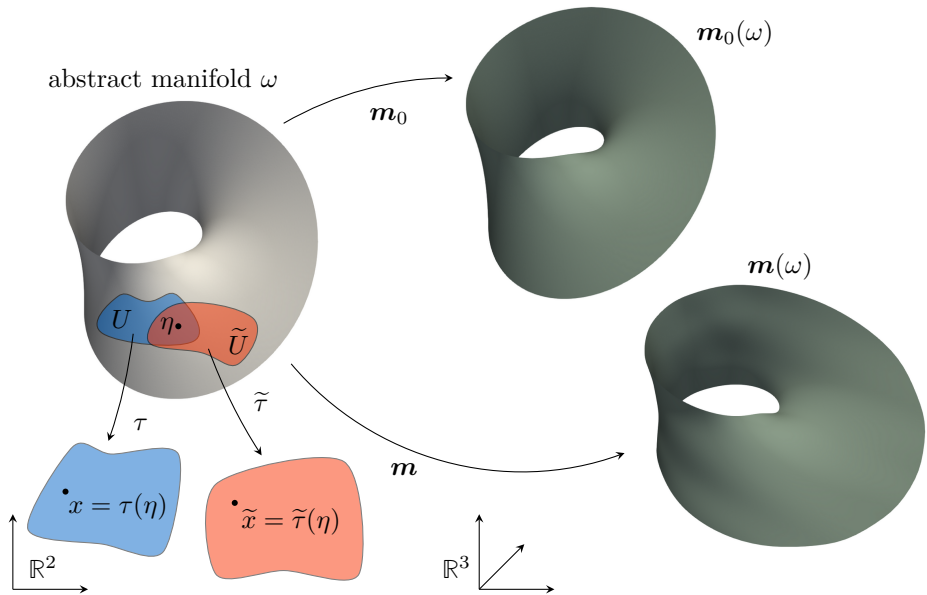
\begin{figure}
	\begin{center}
		\begin{tikzpicture} 

			\node at (1.8,4){abstract manifold $\omega$};
			\node at (2.2,1.7) {\includegraphics[height=4cm]{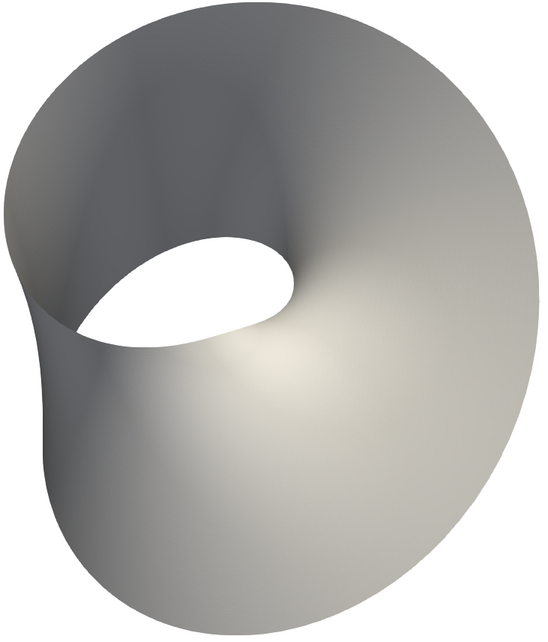}};

			\draw [black, fill=coordinateblue, opacity=0.5] plot [smooth cycle, tension=0.8] coordinates {(1,0.6) (2.2,0.5) (2.0,1.1) (1.6,1.03) (1.2,1.1)};
			\draw [black, fill=coordinatepink, opacity=0.5] plot [smooth cycle, tension=0.7] coordinates {(1.6,0.5) (2.3, 0.35) (2.8,0.2) (3,0.8) (2,1) };

			\fill[black] (2,0.7) circle(.04);
			\node at (1.85,0.7) {$\eta$};

			\node at (1.25, 0.8) {$U$};
			\draw[-stealth] (1.4,0.6) arc (353:341:7cm);
			\node at (1.5,-0.6){$\tau$};
			\draw [black, fill=coordinateblue, opacity=0.5] plot [smooth cycle] coordinates {(0,-2) (1,-2.2) (1.8,-2.3) (2.0,-1) (1.1, -1) (0.6,-0.9)};
			\fill[black] (0.5,-1.5) circle(.04);
			\node at (1.2,-1.7) {$x = \tau(\eta)$};

			\node at (2.8, 0.5) {$\widetilde U$};
			\draw[-stealth] (2.4,0.4) arc (200:213:7cm);
			\node at (3.1,-0.3){$\widetilde{\tau}$};
			\draw [black, fill=coordinatepink, opacity=0.5] plot [smooth cycle] coordinates {(2.5,-2.6) (3.4,-2.55) (4.4,-2.6) (4.2,-1.1) (2.5,-1.4) };
			\fill[black] (2.75,-1.65) circle(.04);
			\node at (3.5,-1.9) {$\widetilde x = \widetilde \tau(\eta)$};

			\draw[-stealth] (-0.2,-2.8) -- ++(1,0);
			\draw[-stealth] (-0.2,-2.8) -- ++(0,1);
			\node at (0.2,-2.5){$\R^2$};

			\node at (7.5,3) {\includegraphics[height=4cm]{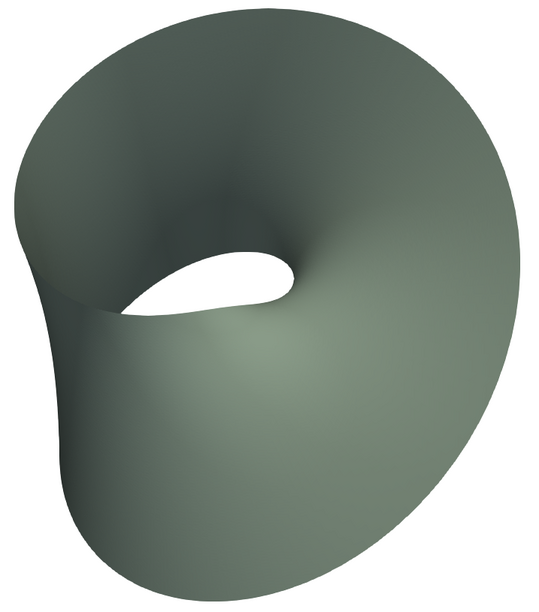}};
			\node at (10,-0.5) {\includegraphics[height=3.5cm]{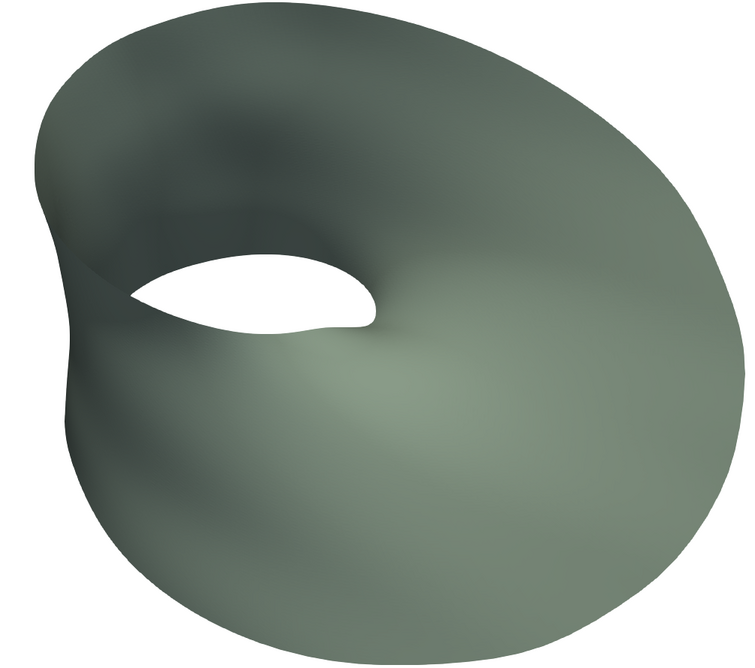}};
			\node at (9.4,4.6){$\bm_0(\omega)$};
			\node at (10,1.5){$\bm(\omega)$};

			\node at (5,-1.2){$\bm$};
			\draw[-stealth] (3.5,0.2) arc (220:290:4cm);
			\node at (4.8,3.5){$\bm_0$};
			\draw[-stealth] (3.4,3.4) arc (120:90:4.5cm);

			\draw[-stealth] (6.0,-2.8) -- ++(1,0);
			\draw[-stealth] (6.0,-2.8) -- ++(0,1);
			\draw[-stealth] (6.0,-2.8) -- ++(0.6,0.6);
			\node at (5.7,-2.5){$\R^3$};
		\end{tikzpicture}
	\end{center}
	\caption{Kinematics of the shell surface: The abstract parameter manifold~$\omega$
	 is immersed into $\R^3$ to yield configurations $\bm_0$ and $\bm$ of the shell.
	 Points $\eta$ on $\omega$ are described by local coordinate systems
	 such as $\tau : U \to \R^2$ and $\widetilde{\tau} : \widetilde{U} \to \R^2$.}
	 \label{fig:shell_surface_kinematics}
\end{figure}

The shell surface is parame\-trized by an abstract two-dimensional manifold $\omega$, possibly with boundary.
If $(U, \tau)$ is a coordinate chart, i.e., a homeomorphism from an open set $U \subset \omega$ to a  subset of $\R^2$, then we write $\eta$ for a point in $\omega$ and $x = \tau(\eta)$ for its local coordinates
(Figure~\ref{fig:shell_surface_kinematics}).
The situation of~\cite{neff_derivation:2020} is recovered when $\omega$ can be described by a single coordinate patch $U = \omega$, in which case $\omega$ can be chosen as a subset of $\R^2$, and $\tau$ can be the identity map.

Configurations of the shell surface are realized as immersions of $\omega$ into $\R^3$.
Physics require them to be even embeddings, i.e., injective, but we stick to the slightly more general case.
The stress-free reference configuration is given as
an immersion $\bm_0: \omega \to \R^3$. We assume $\bm_0$ to be a map in $H^1(\omega,\R^3)$,
with the definition of a Sobolev space on a manifold from \cite{wloka:1987}.
In addition, we require that $\bm_0$ is at least piecewise in $H^2$,
for the second fundamental tensor $\bb$ to exist almost everywhere.
Deformations of the shell surface under load are described by a second immersion $\bm:\omega \to \R^3$,
which we discuss in Section~\ref{sec:kinematics-strain-measures}. For the rest of this section we focus on the initial configuration $\bm_0$.
Derivatives appearing below are to be interpreted in the weak sense if appropriate.

The Sobolev smoothness of $\bm_0: \omega \to \R^3$ allows to define metric and
curvature measures of the immersion $\bm_0$ in a weak sense.
\begin{definition}[Covariant basis vectors]
 Let $\eta$ be a point on $\omega$. The covariant basis vectors $\ba_1, \ba_2 \in \R^3$ at
 $\eta \in \omega$ are
 \begin{equation*}
  \ba_1(\eta) \colonequals \frac{\partial \bm_0(\eta)}{\partial x_1},
  \qquad
  \ba_2(\eta) \colonequals \frac{\partial \bm_0(\eta)}{\partial x_2}.
 \end{equation*}
\end{definition}
Expressions like these are to be interpreted in local coordinates: If $(U, \tau)$ is a
coordinate chart with $\eta \in U$, then write $x = \tau(\eta)$ and define
\begin{equation*}
  \ba_1(\eta) \colonequals \frac{\partial \bm_0(\tau^{-1}(x))}{\partial x_1},
  \qquad
  \ba_2(\eta) \colonequals \frac{\partial \bm_0(\tau^{-1}(x))}{\partial x_2}.
\end{equation*}
The map $\bm_0 \circ \tau^{-1}$ that appears in these expressions corresponds to
what is called~$y_0 = \Theta(0)$ in~\cite{neff_derivation:2020}.

The covariant basis vectors at $\eta$ span the tangent space of $\bm_0(\omega) \subset \R^3$ at $\bm_0(\eta)$.
We interpret them as column vectors. In the notation of \cite{neff_derivation:2020} they represent the first two columns of
\begin{equation*}
 \nabla_x \Theta(0)
 =
 \bigg( \frac{\partial}{\partial x_1} \Theta(0) \:\Big\vert \:\frac{\partial}{\partial x_2} \Theta(0)  \: \Big\vert \: \bn_0\bigg)
 =
 (\ba_1 \:\vert\: \ba_2 \: \vert \: \bn_0) \in \R^{3 \times 3}.
\end{equation*}
The third column is the unit normal vector field
 \begin{equation}\label{def:normal}
 \bn_0
 \colonequals
 \frac{\ba_1 \times \ba_2}{\norm{\ba_1 \times \ba_2}}.
 \end{equation}
The orientation of this field depends on the choice of local coordinates.

We also need the contravariant basis vectors of the tangent spaces of $\bm_0(\omega)$.
They should be interpreted as row vectors.

\begin{definition}[Contravariant basis vectors]
 At any $\eta \in \omega$, the contravariant basis vectors are the vectors $\ba^1(\eta)$, $\ba^2(\eta) \in \R^3$ that are orthogonal
 to $\bn_0(\eta)$ and such that $\ba_\alpha(\eta) \cdot \ba^\beta(\eta) = \delta_{\alpha \beta}$,
 where $\delta_{\alpha \beta}$ is the Kronecker delta.
\end{definition}

Next, we define the first and second fundamental tensors.

\begin{definition}[Fundamental tensors]\label{def:ab}
 The first and second fundamental tensors of the immersion $\bm_0 : \omega \to \R^3$ are
 \begin{alignat*}{2}
 \ba & : \omega \to \R^{3 \times 3}
 & \qquad &
 \ba(\eta)
 \colonequals
 \frac{\partial \bm_0(\eta)}{\partial x_\alpha} \otimes \ba^\alpha(\eta)  = \ba_\alpha(\eta)  \otimes \ba^\alpha(\eta)  \\
 \shortintertext{and}
 \bb & : \omega \to \R^{3 \times 3}
 & &
 \bb(\eta)
 \colonequals
 - \frac{\partial \bn_0(\eta) }{\partial x_\alpha} \otimes \ba^\alpha(\eta),
 \end{alignat*}
 respectively.
\end{definition}
Again, these expressions should be interpreted in local coordinates $x = \tau(\eta)$.
Direct computations show that the tensors $\ba$ and $\bb$ are symmetric.
They are related to the better-known first and second fundamental forms
\begin{equation*}
\I  \colonequals (\nabla \bm_0)^T \nabla \bm_0
\qquad \text{and} \qquad
\II \colonequals -(\nabla \bm_0)^T \nabla \bn_0
\end{equation*}
via
\begin{align*}
 \ba
 & =
 (\ba^\alpha \otimes \be_\alpha) \I (\be_\beta \otimes \ba^\beta) \\
 \shortintertext{and}
 \bb
 & =
 (\ba^\alpha \otimes \be_\alpha) \II (\be_\beta \otimes \ba^\beta),
\end{align*}
where $\nabla \bm_0 \colonequals (\ba_1 | \ba_2)$ is the $3 \times 2$ matrix with columns
$\ba_1$ and $\ba_2$, and $\be_1$, $\be_2$ are the canonical basis vectors of $\R^2$.

\begin{remark}
\label{rem:coordinate_independence_I}
As the quantities $\ba$ and $\bb$ are defined with respect to particular coordinate systems,
it is important to verify that the definitions are independent of the specific choice.
This is indeed the case~\cite{anicic:2001}, but as it turns out, $\bb$ changes its sign under
orientation-reversing coordinate changes. To see both, let $\eta$ be a point on $\omega$,
and let $(U,\tau)$, $(\widetilde{U}, \widetilde{\tau})$ be two coordinate charts with $\eta \in U, \widetilde{U}$.
Call $x$ and $\widetilde{x}$ the coordinates of $\eta$ in $\tau(U) \subset \R^2$
and $\widetilde{\tau}(\widetilde{U}) \subset \R^2$, respectively.
Then we have the coordinate transformation map $\widetilde{x} \mapsto x$, i.e.,
$x = \big(\tau \circ \widetilde{\tau}^{-1}\big)(\widetilde{x})$, defined in an open neighborhood
of $\widetilde{\tau}(\eta) \in \R^2$. The Jacobian of this map is
{
\everymath{\displaystyle}
\begin{equation*}
 G
 \colonequals
 \begin{pmatrix}
  \frac{\partial (\tau \circ \widetilde{\tau}^{-1})_1(\widetilde{x})}{\partial \widetilde{x}_1} &
  \frac{\partial (\tau \circ \widetilde{\tau}^{-1})_1(\widetilde{x})}{\partial \widetilde{x}_2} \\
  \frac{\partial (\tau \circ \widetilde{\tau}^{-1})_2(\widetilde{x})}{\partial \widetilde{x}_1} &
  \frac{\partial (\tau \circ \widetilde{\tau}^{-1})_2(\widetilde{x})}{\partial \widetilde{x}_2}
 \end{pmatrix}.
\end{equation*}
}

Suppose that $\bv_1$, $\bv_2$ are two vector fields that transform like vectors.
By this we mean that when we interpret $\bv_1$, $\bv_2$ to be given locally
in coordinates $(U,\tau)$, and $\widetilde{\bv}_1$, $\widetilde{\bv}_2$
to be the same vector fields in coordinates $(\widetilde{U}, \widetilde{\tau})$, then
\begin{equation} \label{eq:transformation-vectors}
 \widetilde{\bv}_\alpha = \bv_\alpha G,
 \qquad \text{which is} \qquad
 \big( \widetilde{\bv}_1  | \widetilde{\bv}_2 \big)
 =
 \big( \bv_1 | \bv_2 \big)
 G.
\end{equation}
Similarly, let $\bw^2$, $\bw^2$ be fields that transform like covectors:
\begin{equation*}
 \begin{pmatrix} \widetilde{\bw}^1 \\[-2.2mm] \textrm{---} \\[-1.3mm] \widetilde{\bw}^2 \end{pmatrix}
 =
 G^{-1}
 \begin{pmatrix} \bw^1 \\[-2.2mm] \textrm{---} \\[-1.3mm] \bw^2 \end{pmatrix}.
\end{equation*}
Then the matrix field $\bv_\alpha \otimes \bw^\alpha = \sum_{\alpha = 1}^2 v_\alpha w^\alpha$ is independent of the coordinates on~$\omega$, because
\begin{equation*}
 \widetilde{\bv}_\alpha \otimes \widetilde{\bw}^\alpha
 =
 \big( \widetilde{\bv}_1  | \widetilde{\bv}_2 \big) \begin{pmatrix} \widetilde{\bw}^1 \\[-2.2mm] \textrm{---} \\[-1.3mm] \widetilde{\bw}^2 \end{pmatrix}
 =
 \big( \bv_1 | \bv_2 \big) G G^{-1} \begin{pmatrix} \bw^1 \\[-2.2mm] \textrm{---} \\[-1.3mm] \bw^2 \end{pmatrix}
 =
 \big( \bv_1 | \bv_2 \big) \begin{pmatrix} \bw^1 \\[-2.2mm] \textrm{---} \\[-1.3mm] \bw^2 \end{pmatrix}
 =
 \bv_\alpha \otimes \bw^\alpha.
\end{equation*}
The independence of $\ba \colonequals \ba_\alpha \otimes \ba^\alpha$ then follows because $\ba_\alpha = \partial \bm_0 / \partial x_\alpha$ transforms
like a vector, and $\ba^\alpha$ transforms like a covector.  Similarly, $\bn_0$ is invariant
under orientation-preserving changes of coordinates, and therefore $\partial \bn_0 / \partial x_\alpha$
transforms like a vector, too. This implies that $\bb \colonequals - \frac{\partial \bn_0}{\partial x_\alpha} \otimes \ba^\alpha$
is invariant under orientation-preserving coordinate changes.
However, if the orientation is reversed, then $\bn_0$ changes its sign, and so does $\bb$.
\end{remark}

The hyperelastic energy of the Cosserat shell model depends on the classical extrinsic curvature measures
(Gauß curvature and mean curvature) of the immersed shell surface $\bm_0(\omega)$.
These can be computed conveniently from the second fundamental tensor:
\begin{lemma}[{Curvature \cite[Section~3.1]{neff_dim_reduction:2019}}]\label{lem:curvature}
 Let $\kappa_1$ and $\kappa_2$ be the principal curvatures of $\bm_0(\omega)$, wherever defined.
 Then the Gauss curvature $K \colonequals \kappa_1 \cdot \kappa_2$ and the mean curvature
 $H \colonequals \frac{1}{2}(\kappa_1 + \kappa_2)$ can be expressed as
\begin{equation*}
 K = \det(\bb),
 \qquad
 H = \frac{1}{2}\tr(\bb).
\end{equation*}
\end{lemma}
Note that $K$ is independent of the orientation, but $H$ is not.

Finally, the hyperelastic energy involves the alternating pseudo-tensor of the immersed surface.
\begin{definition}[Surface alternating pseudo-tensor]\label{def:alternating_symbol}
 The alternating pseudo-tensor of the immersed~surface~$\bm_0(\omega)$ is
 \begin{align*}
  \bc : \omega \to \R^{3 \times 3},
  \qquad
  \bc
  & \colonequals
  \frac{1}{\sqrt{\det\big((\nabla \bm_0)^T \nabla \bm_0 \big)}} \big( \ba_1 \otimes \ba_2 - \ba_2 \otimes \ba_1 \big) \\
  & \;=
  \sqrt{\det\big((\nabla \bm_0)^T \nabla \bm_0 \big)} \big( \ba^1 \otimes \ba^2 - \ba^2 \otimes \ba^1 \big).
 \end{align*}
\end{definition}

Direct computation shows that the matrix $\bc$ is skew-symmetric, and that it is
indeed a pseudo-tensor, because, if $\widetilde{\bc}$ is a representation with respect to a second coordinate chart,
\begin{align*}
 \widetilde{\bc}
 & =
 \frac{1}{\sqrt{\det\big((\nabla \widetilde{\bm}_0)^T \nabla \widetilde{\bm}_0 \big)}}
    \big( \widetilde{\ba}_1 \otimes \widetilde{\ba}_2 - \widetilde{\ba}_2 \otimes \widetilde{\ba}_1 \big) \\
 & =
 \frac{\det G}{\sqrt{(\det G)^2 \det\big((\nabla \bm_0)^T \nabla \bm_0 \big)}} \big( \ba_1 \otimes \ba_2 - \ba_2 \otimes \ba_1 \big) \\
 & =
 \operatorname{sgn}({\det G}) \cdot \bc.
\end{align*}
Also, it is a linear complex structure on each tangent space, because $\bc^2 = -\ba$
and $\ba$ is the identity on the tangent space.

\begin{remark} [Relationship to \cite{neff_derivation:2020}]
 In the language of~\cite[Section\,3.2]{neff_derivation:2020}, the tensors $\ba$, $\bb$, and $\bc$
 are defined in terms of the map
 \begin{equation*}
  \nabla_x \Theta(0)
  =
  (\nabla y_0 \: | \: n_0),
 \end{equation*}
 and read
 \begin{align*}
  A_{y_0}
  & \colonequals
  (\nabla y_0 | 0) [\nabla_x \Theta(0)]^{-1} \\
  B_{y_0}
  & \colonequals
  -(\nabla n_0 | 0) [\nabla_x \Theta(0)]^{-1} \\
  C_{y_0}
  & \colonequals
  \det (\nabla_x \Theta(0)) [\nabla_x \Theta(0)]^{-T} \textstyle
  \begin{psmallmatrix} 0 & 1 & 0 \\ -1 & 0 & 0 \\ 0 & 0 & 0 \end{psmallmatrix}
  [\nabla_x \Theta(0)]^{-1}.
 \end{align*}
 Since
 \begin{equation*}
    \nabla_x \Theta(0) = (\nabla y_0 | n_0) = (\ba_1 \:\vert\: \ba_2 \: \vert \: \bn_0)
  \qquad \text{and} \qquad
  [\nabla_x \Theta(0)]^{-1}
  =
  \begin{psmallmatrix}
  \ba^1 \\[-0.2mm] \textrm{---} \\[-0.3mm] \ba^2 \\[-0.2mm] \textrm{---} \\[-0.3mm] \bn_0
  \end{psmallmatrix}
 \end{equation*}
 in our notation, we do get $\ba = A_{y_0}$, $\bb = B_{y_0}$, and $\bc = C_{y_0}$.
\end{remark}

\subsection{Kinematics and strain measures}
\label{sec:kinematics-strain-measures}
In Cosserat theory, the configuration of a shell is given by a deformation of the shell surface
together with an independent field of rotations, called the \emph{microrotation field}.
The shell surface deformation is described by a function $\bm : \omega \to \R^3$.
Physical models would require injectivity of~$\bm$, but
we explicitly allow $\bm$ to be
non-injective, and therewith accommodate objects like the Klein bottle of Chapter~\ref{sec:non-orientable_shell surfaces}.
The microrotation field $\Qe : \omega \to \SOdrei$ models rotations
of infinitesimal parts of the shell (see Chapter~\ref{sec:interpretation} below
for details on the interpretation).  In the reference configuration, $\Qe$ will be
the identity matrix field.

At any point $\eta \in \omega$ around which $\bm_0(\omega)$ is sufficiently smooth
we define strain measures for the shell surface deformation and the
microrotation field~\cite{pietraszkiewicz_eremeyev:2009}.
The geometry of the shell surface $\bm_0(\omega)$ enters these expressions
in form of the metric tensor~$\ba:\omega \to \R^{3 \times 3}$ of Chapter~\ref{sec:shell surface_geometry}.

\begin{definition}[Strain tensors]
\label{def:strain_tensors}
 The shell strain tensor is
 \begin{align*}
  \Ee & : \omega \to \R^{3 \times 3}
  \qquad
  \Ee (\eta) \colonequals \Qe^T(\eta) \frac{\partial \bm(\eta)}{\partial x_\alpha} \otimes \ba^\alpha(\eta) - \ba(\eta), \\
 \shortintertext{and the shell bending--curvature tensor is}
  \Ke & : \omega \to \R^{3 \times 3}
  \qquad
  \Ke(\eta)
  \colonequals
  \axl\Big(\Qe^T(\eta) \frac{\partial \Qe(\eta)}{\partial x_\alpha}\Big) \otimes \ba^\alpha(\eta).
 \end{align*}
\end{definition}
Here, the map $\axl(\cdot) : \mathfrak{so}(3) \to \R^3$ computes the axial vector of a given skew-symmetric $3 \times 3$ matrix
\begin{equation*}
 \axl(A) \colonequals (A_{23}, \; A_{31}, \; A_{12})^T \in \R^3
 \qquad
 \text{with $A = (A_{ij}) \in \R^{3 \times 3}$}.
\end{equation*}
The matrix $\Qe^T \frac{\partial \Qe}{\partial x_\alpha}$ really is skew symmetric,
because $\frac{\partial (\bQ^T \bQ)}{\partial x_\alpha}= 0$ for any field
of orthogonal matrices $\bQ$.
The partial derivatives in the expressions for $\Ee$ and $\Ke$ are to be interpreted
with respect to an arbitrary coordinate chart $(U, \tau)$ around $\eta$.
The image $\tau(U) \subset \R^2$ of this chart then corresponds to the flat ``fictitious'' domain~$\omega$
of~\cite{neff_derivation:2020}. Indeed, in view of equation (4.36) of \cite{neff_derivation:2020},
the quantities $\Ee$ and $\Ke$ correspond to the strains
\begin{align*}
 \mathcal{E}_{m,s}
 &\colonequals
 \overline{Q}_{e,s}^T \left(\nabla m \:\big\vert\: \overline{Q}_{e,s} \nabla_x \Theta(0) e_3\right)
        \left[\nabla_x \Theta(0)\right]^{-1} - \identity_3 \\
 \mathcal{K}_{e,s}
 &\colonequals
 \left(\axl\big(\overline{Q}_{e,s}^T \partial_{x_1}\overline{Q}_{e,s}\big)\:\big\vert\: \axl\big(\overline{Q}_{e,s}^T \partial_{x_2}\overline{Q}_{e,s}\big) \:\big\vert\:\:  0\:\right) \left[\nabla_x \Theta(0)\right]^{-1}
\end{align*}
defined there.  In these expressions, $\Theta(0)$ corresponds to $\bm_0 \circ \tau^{-1}$,
and its inverse Jacobian is $\left[\nabla_x \Theta(0)\right]^{-1} = \begin{psmallmatrix} \ba^1 \\[-0.2mm] \textrm{---} \\[-0.3mm] \ba^2 \\[-0.2mm] \textrm{---} \\[-0.3mm] \bn_0 \end{psmallmatrix}$.
The tensor $\overline{Q}_{e,s}$ (defined in equation~(4.2) of~\cite{neff_derivation:2020})
corresponds to our $\bQ_{e}$.

\begin{remark}
\label{rem:coordinate_independence_II}
Note that the quantities $\Ee$ and $\Ke$ are independent of the coordinates on $\omega$.
Continuing the reasoning of Remark~\ref{rem:coordinate_independence_I},
$\Qe^T \frac{\partial \bm}{\partial x_\alpha}$ transforms like a vector, and therefore
\begin{equation*}
 \Ee = \Big( \Qe^T \frac{\partial \bm}{\partial x_\alpha} - \frac{\partial \bm_0}{\partial x_\alpha} \Big)
       \otimes \ba^\alpha
\end{equation*}
is independent of the coordinates. To show that
$\Ke \colonequals \axl\Big(\Qe^T \frac{\partial \Qe}{\partial x_\alpha}\Big) \otimes \ba^\alpha$
is coordinate-independent, it is sufficient to show that $\axl\Big(\Qe^T \frac{\partial \Qe}{\partial x_\alpha}\Big)$
transforms like a vector, i.e., that
\begin{equation*}
 \Big( \axl\Big(\widetilde{\Qe}^T \frac{\partial \widetilde{\Qe}}{\partial x_1}\Big)
       \Big|
       \axl\Big(\widetilde{\Qe}^T \frac{\partial \widetilde{\Qe}}{\partial x_2}\Big) \Big)
 =
 \Big( \axl\Big(\Qe^T \frac{\partial \Qe}{\partial x_1}\Big) \Big| \axl\Big(\Qe^T \frac{\partial \Qe}{\partial x_2}\Big) \Big)
 G.
\end{equation*}
This follows by direct computation.

\end{remark}

\subsection{Interpretation and reconstruction}
\label{sec:interpretation}
The constructions of the two previous sections have been shown to be invariant under orientation-preserving coordinate changes. However, some quantities like the normal vector $\bn_0$ and the second fundamental tensor $\bb$ change their signs under changes of coordinates that switch the orientation. To show that the model is suitable nevertheless to represent actual thin elastic objects we have to show that the model is independent of the choice of orientation. By this we mean that the three-dimensional reconstruction that can be obtained from the shell model is independent of the choice of coordinates and orientation. As a by-product we obtain the result that the shell model is meaningful even for non-orientable parameter surfaces $\omega$.

We start with a discussion of the meaning of the microrotation field~$\Qe$. This field is typically interpreted as a transversal shear and local drilling
of the shell.  More formally, we equip the initial reference surface $\bm_0(\omega)$ with an initial microrotation field $\bQ_0$ that captures its local orientation.

\begin{definition}[Reference microrotation]\label{def:q0}
 The reference microrotation of $\bm_0(\omega)$ is given by
 the orthogonal part of the polar decomposition of the matrix $(\nabla \bm_0 \:\vert\: \bn_0) \in \R^{3 \times 3}$
 \begin{equation*}
  \bQ_0 \colonequals \operatorname{polar}(\nabla \bm_0 \:\vert\: \bn_0).
 \end{equation*}
\end{definition}
In a sense, the orientation $\bQ_0$ obtained by this is the closest orthogonal approximation to the deformation gradient $(\nabla \bm_0 \:\vert \: \bn_0)$ \cite{fischle_neff:2017}.
The columns $\bd^0_1$, $\bd^0_2$, $\bd^0_3$ of $\bQ_0$ are called directors, and they
form an orthonormal frame.
By the particular construction of $(\nabla \bm_0 \:\vert\: \bn_0)$ we
further know that $\det(\nabla \bm_0 \:\vert\: \bn_0) > 0$,
and hence $\det \bQ_0 = 1$, i.e., $\bQ_0 \in \SOdrei$ everywhere.
Similar local frames appear in many other shell models,
e.g., \cite{betsch_menzel_stein:1998,hughes_liu:1981,pimenta_campello:2009}.

Even though $\bQ_0$ is presented in Definition~\ref{def:q0} as a quantity defined on all of $\omega$,
it is actually local: Both terms $\nabla\bm_0$ and $\bn_0$ imply the choice of a local coordinate system,
and $\bQ_0 \colonequals \operatorname{polar}(\nabla \bm_0 \:\vert \: \bn_0)$ does depend on this choice.
This is not an issue in~\cite{neff_derivation:2020}, where the model is represented
with respect to one single fixed coordinate system only. However, from the more general viewpoint
considered here we have to conclude that $\bQ_0$ has no independent physical meaning.
It only serves to allow an interpretation of $\Qe$ as a change of the local orientation.
An exception is the transverse director~$\bd_3^0$, for which
it is shown in \cite[Chapter~3.2]{neff_dim_reduction:2019} that
\begin{equation*}
 \bd^0_3 = \bn_0.
\end{equation*}
Furthermore, since the reference surface $\bm_0:\omega \to \R^3$ is only required to be in $H^1$
and piecewise in $H^2$, it may, for example, have kinks, across which the deformation gradient $\nabla \bm_0$ and normal $\bn_0$ are discontinuous. Consequently, $\bQ_0$ also cannot be expected to be continuous even in a single coordinate chart.
As it turns out, though, $\bQ_0$ does not appear in the reconstruction formula
or in any of the proofs of existence of solutions.%
\footnote{It does appear in the existence proof in \cite{neff_existence:2020},
but on close inspection that proof turns out to be independent of $\bQ_0$. See Section~\ref{sec:existence}.}
Sign flips of $\bn_0$ on non-orientable surfaces $\omega$ will be accounted for
in the reconstruction formula below.

The microrotation under load $\Qe$ is instead interpreted as acting on the initial microrotation field $\bQ_0$,
to yield
the total microrotation $\ovr \colonequals \Qe \bQ_0$.
Unlike $\bm$, therefore, $\Qe$ is a relative quantity.
Of particular interest is the rotated transverse director
\begin{equation*}
 \bd_3 \colonequals \Qe \bd_3^0 = \Qe \bn_0,
\end{equation*}
which appears in the reconstruction formula below.
Note that $\bd_3$ is not necessarily orthogonal to the deformed shell surface $\bm(\omega)$ anymore.
Also, like $\bQ_0$, the total microrotation~$\ovr$ is not usually continuous.
This does not pose any problem, because only the three-dimensional reconstruction is of
physical relevance.

We now show how a three-dimensional thin shell can be reconstructed from a configuration
of the Cosserat shell model. The construction generalizes the approach of~\cite{neff_derivation:2020}
to general parameter domains $\omega$. Locally, for a fixed arbitrary coordinate chart,
the construction of~\cite{neff_derivation:2020} is recovered. This justifies our model
even though we have not directly derived it from a three-dimensional model.
Presumably such a derivation is possible if $\bm_0(\omega)$ is an embedding, and it would follow the steps in \cite{neff_dim_reduction:2019, neff_derivation:2020} locally in coordinate charts.

In \cite{neff_dim_reduction:2019, neff_derivation:2020} the two-dimensional flat parameter domain $\omega$
is extended by Cartesian multiplication to a thin three-dimensional parameter domain
$\Omega_h \colonequals \omega \times (- \tfrac{h}{2}, \tfrac{h}{2})$.
To generalize this to the case of the abstract two-dimensional parameter surfaces~$\omega$
considered here, the corresponding three-dimensional domain is a tubular neighborhood of $\omega$:
Let $(\mathcal{N},\pi, \omega)$ be the normal bundle
over $\omega$ with respect to the immersion $\bm_0$~\cite{tom_dieck:2008}.
It is a three-dimensional vector bundle, and it is orientable as a three-dimensional manifold.%
\footnote{Thanks to Andreas Thom (Technische Universität Dresden) for this result.}
Its fibres are isomorphic to the one-dimensional vector space~$\R$, and we construct it
such that the structure group is $O(1) = \{-1,1\}$.
Define a subbundle $(\mathcal{N}_h,\pi,\omega)$ such that the fiber $\pi^{-1}(\eta)$
over each point $\eta \in \omega$ is an open set that contains $0$.
Local trivializations can then be chosen to take the form
$\tau(U) \times \big(-\frac{h}{2}, \frac{h}{2}\big)$, where $(\tau, U)$ is a
coordinate chart of $\omega$ and hence $\tau(U)$ is an open set in $\R^2$.
These local trivializations correspond to the single set $\Omega_h$ employed in \cite{neff_dim_reduction:2019, neff_derivation:2020}.

The reconstruction of a three-dimensional deformed configuration from a two-dimensional
Cosserat shell is a map $\varphi_s : \mathcal{N}_h \to \R^3$, the immersion of a
three-dimensional object into $\R^3$.
In coordinates, the reconstruction has the form
\begin{align}
\label{eq:reconstruction}
\varphi_s &:  \:\:\ \tau(U) \times (-\tfrac{h}{2}, \tfrac{h}{2}) \to \R^3
\\
\nonumber
\varphi_s(x_1,x_2,x_3) & \colonequals \: \bm(x_1, x_2) + x_3 \rho_m(x_1,x_2)\bd_3(x_1, x_2) + \tfrac{1}{2} x_3^2 \rho_b (x_1,x_2)\bd_3(x_1, x_2),
\end{align}
which is the formula from \cite{neff_dim_reduction:2019, neff_derivation:2020}, but is now
interpreted as an expression of local coordinates $x = (x_1,x_2) \in \tau(U)$ of a point $\eta \in \omega$.
(In an abuse of notation we have omitted various occurrences of $\tau^{-1}$ here.)
Additionally, $x_3$ is the coordinate of the interval $(-\tfrac{h}{2},\tfrac{h}{2})$.
The coefficient functions are
\begin{align*}
\rho_m & \colonequals 1 - \frac{\lambda}{\lambda + 2 \mu} (\tr \Ee) \\
\shortintertext{and}
\rho_b & \colonequals - \frac{\lambda}{\lambda + 2 \mu} \big(\tr (\Ee \bb + \bc \Ke)\big),
\end{align*}
where $\lambda$ and $\mu$ are the Lamé parameters.

\begin{lemma}
 The reconstruction~\eqref{eq:reconstruction} is independent of the coordinates on $\omega$.
 It is well-defined even if $\omega$ is not orientable.
\end{lemma}
\begin{proof}
 By Remarks~\ref{rem:coordinate_independence_I} and~\ref{rem:coordinate_independence_II},
 the scalar-valued coefficient functions $\rho_m :\tau(U) \to \R$ and $\rho_b:\tau(U) \to \R$ are independent of the
 choice of coordinates on $\omega$; only $\rho_b$ flips its sign under orientation-reversing
 coordinate changes. To see the independence of $\varphi_s$,
 let $U$ and $\widetilde{U}$ be two overlapping coordinate charts of $\omega$.
 We call the corresponding local coordinates $x = (x_1, x_2)$ and $\widetilde x = (\widetilde{x}_1, \widetilde{x}_2)$, respectively,
 and extend them to coordinates $(x_1,x_2,x_3)$ and
 $(\widetilde{x}_1, \widetilde{x}_2, \widetilde{x}_3)$ of the corresponding locals trivializations
 $\tau(U) \times (-\tfrac{h}{2},\tfrac{h}{2})$ and
 $\widetilde{\tau}(\widetilde{U}) \times (-\tfrac{h}{2},\tfrac{h}{2})$, respectively,
 of the normal bundle $(\mathcal{N}_h, \pi, \omega)$.
 As the normal bundle is orientable as a manifold, we only need to consider orientation-preserving coordinate changes $(\widetilde{x}_1, \widetilde{x}_2,\widetilde{x}_3) \mapsto (x_1,x_2,x_3)$.
 Since these are coordinate changes of a (truncated) rank-1 vector bundle
 we can look at transformations $(x_1,x_2) \mapsto (\widetilde{x}_1,\widetilde{x}_2)$ and
 $x_3 \mapsto \widetilde{x}_3$ separately, and since we have used $O(1)$ as the structure group
 the only possible transformations for the latter are $\widetilde{x}_3 = x_3$ and $\widetilde{x}_3 = -x_3$.
 Regarding the overall orientation, we therefore have to distinguish only two cases:
 \begin{enumerate}
  \item $(\widetilde{x}_1, \widetilde{x}_2) \mapsto (x_1,x_2)$ preserves orientation.  Then $\widetilde{x}_3 = x_3$.
  \item $(\widetilde{x}_1, \widetilde{x}_2) \mapsto (x_1,x_2)$ inverts orientation. Then $\widetilde{x}_3$ must equal $-x_3$
   to make $(x_1, x_2, x_3) \mapsto (\widetilde{x}_1, \widetilde{x}_2, \widetilde{x}_3)$
   orientation-preserving again.
 \end{enumerate}
 In the first case, $\bd_3$ is invariant under the change of coordinates, and so is
 the entire reconstruction formula $\varphi_s$.
 In the second case, $\bd_3$ changes its sign. However, so does $x_3$, and the middle addend of~\eqref{eq:reconstruction}
 remains invariant. To see invariance of the last term note that $x_3^2$ does not change
 sign but $\rho_b$ does, and therefore the sign change of $\bd_3$ is compensated for again.
\end{proof}
The reconstruction is continuous only if $\bm:\omega\to \R^3$ is continuously differentiable, because otherwise $\bd_3$
may be discontinuous.

\begin{remark}
 The reconstruction formula~\eqref{eq:reconstruction} given here corresponds to the shell energy functional presented in
 the following chapter. For the variant of that energy discussed in Remark~\ref{rem:birsan_energy},
 a slightly different reconstruction formula has to be used. The details are given in~\cite{birsan:2021}.
\end{remark}

\subsection{Hyperelastic shell energy functional}
\label{sec:shell_energy_functional}

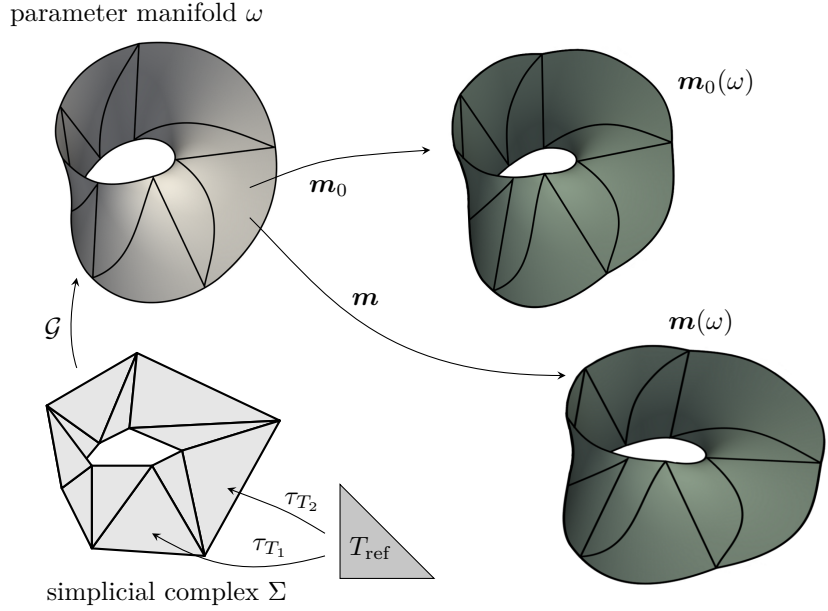
\begin{figure}
 \begin{center}
 	\begin{tikzpicture}
 		\node at (2.2,3.4) {\includegraphics[height=3.5cm]{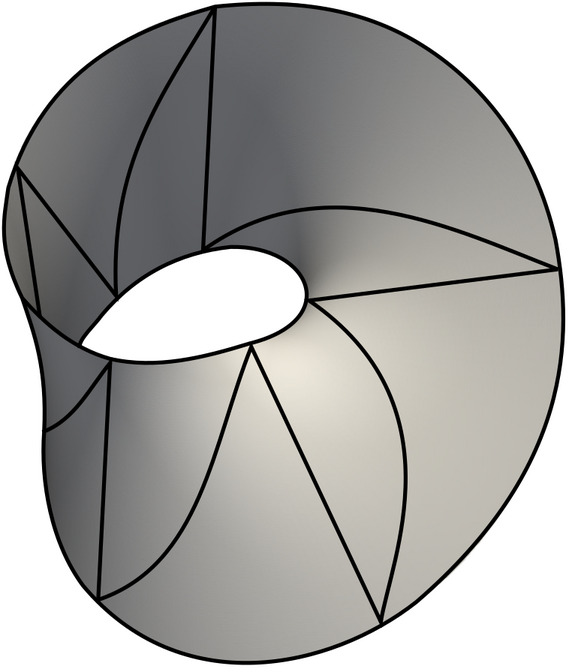}};
 		\node at (1.8,5.5){parameter manifold $\omega$};

 		\begin{scope}
 		\node at (2.2,-2.2){simplicial complex $\Sigma$};

		\coordinate (v0) at (0.6, 0.3);
		\coordinate (v1) at (1.8, 1.0);
		\coordinate (v2) at (3.7, 0.1);
		\coordinate (v3) at (2.7, -1.7);
		\coordinate (v4) at (1.2, -1.6);
		\coordinate (v5) at (0.8, -0.8);
		\coordinate (v6) at (1.3, -0.2);
		\coordinate (v7) at (1.7, -0.0);
		\coordinate (v8) at (2.4, -0.3);
		\coordinate (v9) at (2.0, -0.5);
		\coordinate (v10) at (1.2, -0.5);

		\draw[line join=round, fill=gray!20, thick] (v0) -- (v5) -- (v6) -- cycle;
		\draw[line join=round, fill=gray!20, thick] (v0) -- (v1) -- (v6) -- cycle;
		\draw[line join=round, fill=gray!20, thick] (v1) -- (v7) -- (v6) -- cycle;
		\draw[line join=round, fill=gray!20, thick] (v1) -- (v2) -- (v7) -- cycle;
		\draw[line join=round, fill=gray!20, thick] (v2) -- (v8) -- (v7) -- cycle;
		\draw[line join=round, fill=gray!20, thick] (v2) -- (v3) -- (v8) -- cycle;
		\draw[line join=round, fill=gray!20, thick] (v3) -- (v9) -- (v8) -- cycle;
		\draw[line join=round, fill=gray!20, thick] (v3) -- (v4) -- (v9) -- cycle;
		\draw[line join=round, fill=gray!20, thick] (v4) -- (v10) -- (v9) -- cycle;
		\draw[line join=round, fill=gray!20, thick] (v4) -- (v5) -- (v10) -- cycle;
		\draw[line join=round, fill=gray!20, thick] (v10) -- (v5) -- (v0) -- cycle;
		\end{scope}

		\filldraw[draw=black, fill=gray!45] (4.5,-2.0) -- ++(1.25,0.0) -- ++(-1.25,1.25) -- cycle;
		\node at (4.9,-1.6) {$\Tref$};

		\node at (7.5,3.4) {\includegraphics[height=3.7cm]{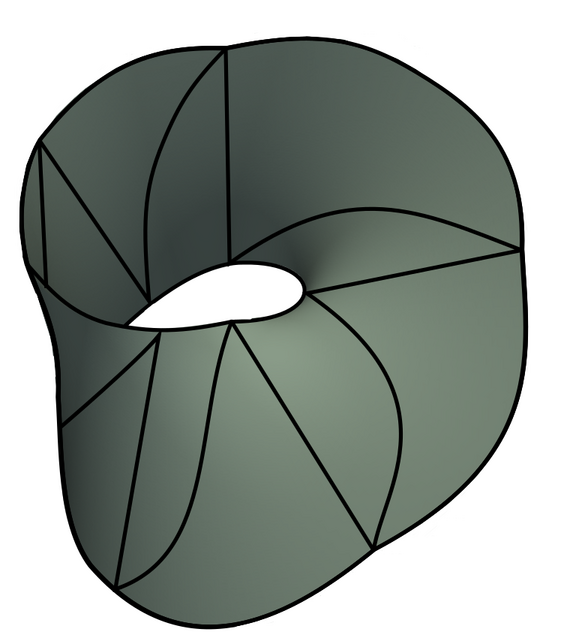}};
		\node at (9.5,4.6){$\bm_0(\omega)$};

		\node at (9.2,-0.5) {\includegraphics[height=3.3cm]{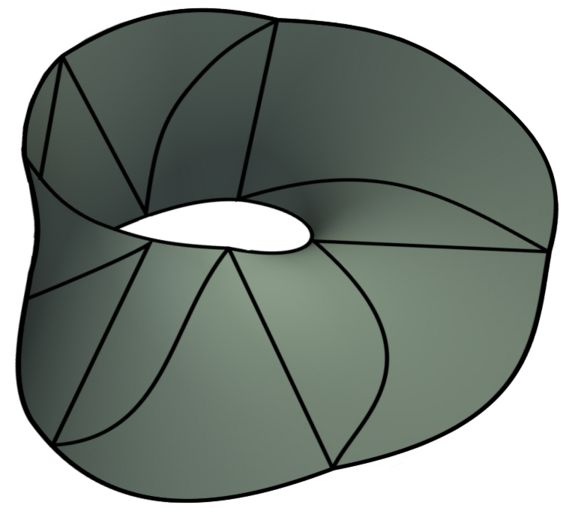}};
		\node at (9.3,1.4){$\bm(\omega)$};

		\draw[-stealth] (3.3,3.2) .. controls ++(1,0.4)  .. node[below=4pt] {$\bm_0$} (5.7,3.7);

		\draw[-stealth] (3.3,2.8) .. controls ++(1,-1) and (5,0.7) .. node[above=4pt] {$\bm$} (7.5,0.7);

		\draw[-stealth] (1.0,0.8) .. controls ++(-0.1,0.3) and (0.9, 1.5) .. node[left] {$\mathcal{G}$} (1.0,2.0);

		\draw[-stealth] (4.3, -1.7) .. controls ++(-1.1,-0.3) and (2.7, -1.8) .. node[near start, above] {$\tau_{T_1}$} (2.0,-1.4);
		\draw[-stealth] (4.3, -1.4) .. controls ++(-0.5,0.3)  .. node[near start, above] {$\tau_{T_2}$} (3.0,-0.8);

	\end{tikzpicture}
 \end{center}
 \caption{Triangulation of the abstract parameter manifold $\omega$, its triangulation by a simplicial
  complex $\Sigma$, and its immersions into $\R^3$}
 \label{fig:triangulation}
\end{figure}

We assume that the shell behavior can be described by a hyperelastic material.
For the case of a single coordinate patch, the authors of \cite{neff_dim_reduction:2019}
and~\cite{neff_derivation:2020} derived an
energy functional of the form
\begin{equation*}
 I(\bm, \Qe)
 \colonequals
 \int_{\omega} \Big[W_\text{memb}(\Ee, \Ke) + W_\text{bend}(\Ke) \Big] d\omega - \Pi_\text{ext}(\bm, \Qe)
\end{equation*}
by dimensional reduction of a three-dimensional Cosserat material.
While in these works the parameter domain $\omega$ was a flat domain in $\R^2$, it is now an abstract two-dimensional manifold.
To integrate the energy density over this manifold, we eschew the traditional approach using a partition of unity on $\omega$.
Rather, anticipating the numerical approximation of shell problems by the finite element method, we cover $\omega$ by a triangulation.  In the following definition, $\Sigma$ is a pure two-dimensional simplicial complex.
By $\abs{\Sigma}$ we denote the polyhedron of $\Sigma$, i.e., the union of all of its simplices.

\begin{definition}[Triangulation~\cite{thurston:1997}]\label{def:triangulation}
A triangulation of $\omega$ is a simplicial complex~$\Sigma$, together with
a homeomorphism $\mathcal G : \abs{\Sigma} \to \omega$.
We call $\mathcal{T}$ the set of triangles of~$\omega$, i.e., the set of the images
of the two-dimensional simplices of $\Sigma$ under the map $\abs{\Sigma} \to \omega$.
\end{definition}

For each two-dimensional simplex $S$ of $\Sigma$, there is an affine homeomorphism
from a fixed open triangle $\Tref \subset \R^2$ to $S$, unique up to permutations
of the triangle vertices. In finite element parlance, $\Tref$ is the reference triangle.
Concatenating this homeomorphism with the triangulation map~$\mathcal G$ we obtain maps $\tau_T : \Tref \to T \subset \omega$
for each triangle $T$ in $\mathcal{T}$ (Figure~\ref{fig:triangulation}).
In the following we will only use the maps $\tau_T$, $T \in \mathcal{T}$, and not the
simplicial complex $\Sigma$.  The reason we nevertheless have to introduce $\Sigma$ for the
definition of our triangulation is to make sure that the local coordinates induced
by the maps $\tau_T$ on the
triangles match at the triangle edges.
Later, the triangles will form the finite element grid.

To compute integrals over $\omega$, we further need an area element. We use the one induced by the reference immersion which, in local coordinates $\tau: \eta \mapsto x$, reads
\begin{equation}
\label{eq:area_element}
 d\omega
 \colonequals
 \sqrt{\det\big((\nabla \bm_0)^T \nabla \bm_0 \big)}\,dx_1\,dx_2.
\end{equation}
We then split the
integral along the triangulation $\mathcal{T}$, and we
rewrite the energy as a sum over the triangles
\begin{multline}\label{eq:finite_strain_energy}
 I(\bm, \Qe)
 =
 \sum_{T \in \mathcal{T}} \int_{T} \Big[W_\text{memb}(\Ee, \Ke) + W_\text{bend}(\Ke) \Big] \sqrt{\det((\nabla \bm_0)^T \nabla \bm_0)}\,dx_1\,dx_2 \\
 \hfill - \Pi_\text{ext}(\bm,\Qe).
\end{multline}

The energy density in~\eqref{eq:finite_strain_energy} depends on the pair $(\bm, \Qe)$
through the strain measures $\Ee$ and $\Ke$. It consists of a membrane part
\begin{multline}\label{eq:Wmemb}
 W_\text{memb}(\Ee, \Ke)
 \colonequals
 \Big(h - K \frac{h^3}{12}\Big)W_\text{m}(\Ee) + \Big(\frac{h^3}{12} - K \frac{h^5}{80}\Big) W_\text{m}(\Ee \bb + \bc \Ke) \\
 + \frac{h^3}{6} W_\text{mixt}(\Ee, \bc\Ke\bb- 2H \bc\Ke) + \frac{h^5}{80}W_\text{mp}\big((\Ee \bb + \bc\Ke)\bb\big),
\end{multline}
and a bending--curvature part
\begin{equation}\label{eq:Wbend}
 W_\text{bend}(\Ke)
 \colonequals
 \Big(h - K \frac{h^3}{12}\Big)W_\text{curv}(\Ke)
 + \Big(\frac{h^3}{12} - K \frac{h^5}{80}\Big)W_\text{curv}(\Ke \bb)
 + \frac{h^5}{80}W_\text{curv}(\Ke \bb^2).
\end{equation}
The parameter $h > 0$ represents the thickness of the shell.
The values $K$ and $H$ are the Gauss and mean curvatures
of $\bm_0(\omega)$, respectively, and $\ba, \bb$, and $\bc$ are the fundamental tensors and the
alternating pseudo-tensor of the reference shell surface $\bm_0(\omega)$ from Definitions~\ref{def:ab}
and~\ref{def:alternating_symbol}, respectively.
All these quantities are defined on the abstract parameter surface $\omega$, and can be expressed in flat coordinates by means of local coordinate charts.

\begin{remark}
 Note that some of the terms in~\eqref{eq:Wmemb}~and~\eqref{eq:Wbend} involve the quantities $\bb, \bc$, and $H$,
 which depend on the orientation of the surface.
 However, these quantities always either appear in pairs such that the orientation-dependence cancels,
 or they appear as arguments of quadratic functionals (see below), and therefore the sign does not matter.
 Consequently, the integral~\eqref{eq:finite_strain_energy} is independent of the choice of orientation,
 and it is well defined even for non-orientable surfaces.
\end{remark}

Of the functionals appearing in~\eqref{eq:Wmemb}~and~\eqref{eq:Wbend}, $W_\text{mixt}$ is a bilinear form on $\R^{3 \times 3} \times \R^{3 \times 3}$ and
$W_\text{m}$, $W_\text{mp}$,  and $W_\text{curv}$
are quadratic forms on $\R^{3 \times 3}$.  For their explicit representations define
\begin{equation*}
 \sym X \colonequals \frac{1}{2}\big(X + X^T),
 \qquad
 \skew X \colonequals \frac{1}{2}\big(X - X^T),
 \qquad
 \dev_3 X \colonequals X - \frac{1}{3}(\tr X) \identity_3.
\end{equation*}
Then
\begin{align}
\nonumber
W_\text{mixt}(X, Y) & \colonequals \mu \langle \sym X, \sym Y \rangle + \mu_c \langle \skew X, \skew Y \rangle  + \frac{\lambda \mu}{\lambda + 2 \mu} (\tr X)\cdot (\tr Y)
\\
\nonumber
W_\text{m}(X)
 & \colonequals \mu \norm{\sym X}^2 + \mu_c \norm{\skew X}^2  + \frac{\lambda \mu}{\lambda + 2 \mu} \left( \tr X\right)^2
\\
\nonumber
W_\text{mp}(X)
 & \colonequals
 \mu \norm{\sym X}^2 + \mu_c \norm{\skew X}^2  + \frac{\lambda}{2} \left( \tr X\right)^2
 \\
\label{eq:Wcurv}
 W_\text{curv}(X) &\colonequals \mu L_c^2 \Big(b_1 \norm{\dev_3 \sym X}^2 +
	b_2\norm{\skew X}^2 + b_3\left( \tr X\right)^2\Big).
\end{align}
The parameters $\mu$ and $\lambda$ are the Lamé constants of classical elasticity.
The coefficient $L_c > 0$ is an internal length scale, $b_1,b_2,b_3 > 0$ are curvature
coefficients, and $\mu_c \ge 0$ is the Cosserat couple modulus.  Keeping in mind
that the bending--curvature tensor~$\Ke$ has the dimension of inverse length $\mathsf{L}^{-1}$
and the curvatures $\bb$ and $K$ have dimension of inverse length squared $\mathsf{L}^{-2}$,
the dimension of $L_c$ is indeed a length $\mathsf{L}$, and the $b_1, b_2, b_3$ are dimensionless.
Note further that $W_\text{m}(X) = W_\text{mixt}(X,X)$ and $W_\text{mp}(X) = W_\text{m}(X) + \tfrac{\lambda^2}{2(\lambda + 2\mu)} (\tr X)^2$.

\begin{lemma}
The energy functional~$I$ defined in~\eqref{eq:finite_strain_energy}
(without the external load potential $\Pi_\textup{ext}$) is frame-indifferent in the sense that
\begin{equation}\label{eq:frame_indifference}
I(\bR\,\bm, \bR \,\Qe) = I (\bm, \Qe)
\end{equation}
for all constant rotations $\bR \in \SOdrei$.
\end{lemma}

\begin{remark}
If the reference configuration $\bm_0(\omega)$ is a subset of $(x_1, x_2, 0) \subset \R^3$, then the shell surface curvature measures $K, H, \bb$ vanish and
the model reduces to the Cosserat flat shell model of~\cite{sander_neff_birsan:2016,neff:2004},
but with a modified curvature tensor.
See~\cite[Section~5.2]{neff_derivation:2020} for a discussion.
\end{remark}

The term $\Pi_\text{ext}$ in~\eqref{eq:finite_strain_energy} represents external loads acting on the shell.
We focus here on loads acting on the deformation, and disregard possible orientation loads.
Let $\gamma_t \subset \partial \omega$ be a subset of the shell surface boundary.
We assume that the load on $\gamma_t$ can be expressed via a density function $\bt \in L^2(\gamma_t, \R^3)$.
Analogously, we assume that the load on the body $\omega$ can be expressed
via a function $\boldf \in L^2(\omega, \R^3)$.
Then we define the potential of the two loads resulting from $\boldf$ and $\bt$ as
\begin{equation}\label{eq:Pi}
\Pi_\text{ext}(\bm, \Qe)
\colonequals
\int_{\gamma_t} \langle \bt, \bm - \bm_0\rangle\,ds
+
\int_{\omega} \langle \boldf, \bm - \bm_0\rangle \,d\omega.
\end{equation}
This corresponds to the term used in~\cite{neff_existence:2020}, without the loads for the microrotations.

\begin{remark}
 \label{rem:birsan_energy}
 Following a slightly more general ansatz, \textcite{birsan:2021} derived the following
 closely related shell energy density
 \begin{equation*}
  W_\textup{shell, alt}(\Ee, \Ke) \colonequals
  W_\textup{memb, alt}(\Ee, \Ke) + W_\textup{curv}(\Ke),
  \end{equation*}
  with the same density for the bending--curvature part~\eqref{eq:Wbend}, but with an alternative membrane density
  \begin{align}
  \nonumber
  W_\textup{memb, alt}(\Ee, \Ke) \colonequals
  &
  \left( h - K \frac{h^3}{12} \right) W_\textup{Coss}(\Ee) +
  \left( \frac{h^3}{12} - K \frac{h^5}{80}\right) W_\textup{Coss}(\Ee\bb + \bc \Ke)\\
  \label{eq:Wmembrane_alt}
  & +
  \frac{h^3}{6} W_\textup{Coss}(\Ee, \bc \Ke \bb - 2H \bc \Ke) +
  \frac{h^5}{80} W_\textup{Coss}\big( ( \Ee \bb + \bc\Ke) \bb\big).
 \end{align}
 The density $W_\textup{memb, alt}(\Ee, \Ke)$ looks formally like $W_\textup{memb}(\Ee, \Ke)$
 from Equation~\eqref{eq:Wmemb}; however, the bilinear and quadratic forms $W_\textup{m}$,
 $W_\textup{mixt}$ and $W_\textup{mp}$ are all replaced by
 \begin{align*}
	W_\textup{Coss}(X,Y) & \colonequals W_\textup{mixt}(X,Y) - \frac{(\mu - \mu_c)^2}{2(\mu + \mu_c)}(\bn_0 X)\cdot (\bn_0 Y)
	\\
	& \: = \mu \langle \sym X, \sym Y \rangle + \mu_c \langle \skew X, \skew Y \rangle
	\\
	& \quad + \frac{\lambda \mu}{\lambda + 2 \mu} (\tr X)\cdot (\tr Y) - \frac{(\mu - \mu_c)^2}{2(\mu + \mu_c)}(\bn_0 X)\cdot (\bn_0 Y)
  \shortintertext{or}
  W_\textup{Coss}(X) & \colonequals W_\textup{Coss}(X,X).
 \end{align*}
 To see that the difference to~\eqref{eq:Wmemb} is smaller than it seems, note that
 the tensors $\Ee$, $\Ke$, and $\bb$ are all of the form
 \begin{equation}
 \label{eq:normal_kernel}
  X = \sum_{\substack{i=1,2,3 \\ \alpha=1,2}} X^i{}_\alpha \ba_i \otimes \ba^\alpha
  \qquad
  \text{(with $\ba_3 \colonequals \bn_0$)},
 \end{equation}
 i.e., their kernels always contain $\bn_0$ (Note that $\alpha$ ranges form $1$ to $2$,
 but that $i$ ranges from $1$ to $3$).
 As a consequence, all
 arguments of $W_\textup{Coss}$ in~\eqref{eq:Wmembrane_alt} and of $W_\textup{memb}$ in~\eqref{eq:Wmemb}
 are also of this form.
 For such tensors $X,Y$ we get, with the split
 \begin{equation*}
   X = \identity_3 X = (\ba + \bn_0 \otimes \bn_0)X = \ba X + \bn_0 \otimes (\bn_0 X),
 \end{equation*}
 the formulas
 \begin{align*}
  \langle \sym X, \sym Y \rangle
  & =
  \langle \sym(\ba X), \sym(\ba Y) \rangle + \tfrac{1}{2} \left(\bn_0 X \right)\cdot \left(\bn_0 Y \right) \\
  \langle \skew X, \skew Y \rangle
  & =
  \langle \skew(\ba X), \skew(\ba Y) \rangle + \tfrac{1}{2} \left(\bn_0 X \right)\cdot \left(\bn_0 Y \right) \\
  \tr X
  & =
  \tr(\ba X) \qquad \text{and} \qquad \tr Y= \tr(\ba Y).
 \end{align*}
 Using these, the densities $W_\textup{mixt}(X,Y)$
 and $W_\textup{Coss}(X,Y)$ can be rewritten as \cite[Eq.\,(111)]{birsan:2020}
 \begin{align*}
  W_\textup{mixt}(X,Y) & = W_\textup{mixt}(\ba X, \ba Y) + \frac{\mu + \mu_c}{2}(\bn_0X)\cdot (\bn_0Y),
 \shortintertext{and}
  W_\textup{Coss}(X,Y) & = W_\textup{mixt}(\ba X, \ba Y) + \frac{2\mu \mu_c}{\mu + \mu_c}\left(\bn_0X \right)\cdot (\bn_0Y),
 \end{align*}
 respectively, if $X$ and $Y$ are of the form~\eqref{eq:normal_kernel}.

 Hence, for all relevant arguments the density $W_\textup{Coss}$ differs from
 $W_\textup{mixt}$ only in the transverse shear coefficient:
 The geometric mean $\frac{\mu + \mu_c}{2}$ is replaced by the harmonic mean $\frac{2\mu \mu_c}{\mu + \mu_c}$.
 The same modified membrane energy has been justified by $\Gamma$-convergence arguments
 in~\cite{saem_ghiba_neff:2022}, and the harmonic mean has already appeared in~\cite{neff_chelminski:2007}.
 We will compare this alternative energy functional numerically to the functional from equation~\eqref{eq:finite_strain_energy}
 in Chapter~\ref{sec:comparison_with_3d_model}.
\end{remark}

\subsection{Existence of minimizers}
\label{sec:existence_minimizers}
In \cite{neff_existence:2020}, \citeauthor*{neff_existence:2020} showed existence
of minimizers~$(\bm,\Qe)$ for the functional~\eqref{eq:finite_strain_energy}
in the space $H^1(\omega,\R^3) \times H^1(\omega, \SOdrei)$ for the case that
$\omega$ can be parametrized with a single coordinate chart. However, the
result can be easily generalized to shells with a general two-dimensional parameter manifold~$\omega$.
We briefly state the result here as a preparation for the existence result for geometric finite element solutions in Chapter~\ref{sec:existence}.
We do not give a detailed proof, because that would largely be a copy of the proof in \cite{neff_existence:2020},
but Chapter~\ref{sec:existence} on the existence proof for finite element solutions has more details.
Note that the result here only covers the case that the Cosserat couple modulus $\mu_c$ is strictly positive.
Showing existence of solutions for the important case $\mu_c =0$ requires a different proof,
and more regularity of the microrotation field $\Qe$.
Such a proof has appeared in the literature only for the case of a reference surface
without curvature~\cite{neff:2007}, but it is clear that that proof could be
easily extended to the more general case with curvature as well.

For stating the result in Theorem~\ref{thm:neff_existence} below we had to change
various details, beyond the modifications needed to adapt the statement to our notation.
The original result for surfaces with a single coordinate system (Theorem~3.3 in~\cite{neff_existence:2020})
asked for smoothness of quantities like $\bQ_0$ and $\bn_0$, which are not
coordinate-independent.
As it turns out, though, these smoothness assumptions are not actually needed. We comment
on our modified assumptions at the end of this section.

For a rigorous existence result, we need a formal definition of the space of $\SOdrei$-valued Sobolev functions on $\omega$.
From the different, non-equivalent definitions in the literature we select the one
based on the canonical injection of $\SOdrei$ into $\R^{3\times 3}$ used also, e.g.,
in \cite{gawlik_leok:2017,neff_existence:2020} or \cite{sander:2013}.
With $H^1(\omega, \R^{3\times 3})$ the space of $\R^{3\times 3}$-valued first-order Sobolev functions
as defined in~\cite{wloka:1987}, the corresponding $\SOdrei$-valued Sobolev space is
\begin{equation}\label{eq:sobolev_space_manifold}
 H^1(\omega, \SOdrei)
 \colonequals
 \Big\{R \in H^1(\omega, \R^{3 \times 3})\; : \;
        \text{$R^T R = \identity_3$ and $\det R = 1$ a.e.}\Big\}.
\end{equation}

We then need appropriate Dirichlet conditions
for the shell surface deformation.  For this, let $\gamma_d$ be a subset of $\omega$
such that the restriction of $H^1$-functions to $\gamma_d$ is well-defined.
This can be a part of the boundary of $\omega$ if $\omega$ has a boundary,
but $\gamma_d$ can also be an open subset of $\omega$.
We prescribe Dirichlet conditions on $\gamma_d$ by means of a given deformation $\bm^* \in H^1(\omega, \R^3)$, and require
\begin{equation} \label{eq:dirichlet}
\bm = \bm^*\quad \text{on $\gamma_d$} \quad \text{a.e.}
\end{equation}

There are several reasonable corresponding conditions for the microrotation field $\Qe$.
The proof in \cite{neff_existence:2020} assumes that $\Qe$ is clamped in $\gamma_d$, i.e.,
\begin{equation}\label{eq:dirichlet_rotation}
\Qe =  \Qe^* \quad\text{on } \gamma_d \quad \text{a.e.},
\end{equation}
for a given microrotation field $\Qe^* \in H^1(\omega, \SOdrei).$
However, the proof can be easily adapted to the case of Dirichlet conditions for $\Qe$
on a different set, and it even works with no Dirichlet conditions for $\Qe$
at all \cite[Corollary~3.4]{neff_existence:2020}.

In the following statement of the existence result, $\kappa_1, \kappa_2 : \omega \to \R$ are the principal
curvatures of the reference immersion $\bm_0(\omega)$.
\begin{theorem}[Existence of minimizers]\label{thm:neff_existence}
 Assume that the external loads defined in~\eqref{eq:Pi} satisfy
 \begin{equation*}
  \boldf \in L^2 (\omega,\R^3),
  \qquad
  \bt \in L^2 (\gamma_t, \R^3),
 \end{equation*}
 and that the Dirichlet data satisfies
 \begin{equation*}
  \bm^* \in H^1 (\omega,\R^3),
  \qquad
  \Qe^* \in H^1(\omega, \SOdrei).
 \end{equation*}
 Suppose that the reference shell surface immersion $\bm_0 \in H^1(\omega, \R^3)$ is
 $H^2$ on each triangle $T$ of $\mathcal{T}$, and such that
 \begin{equation*}
  \nabla \bm_0 \in L^\infty (\omega,\R^{3 \times 2})
  \qquad \text{and} \qquad
  \sqrt{\det\big((\nabla \bm_0)^T \nabla \bm_0\big)} \ge a_0 > 0 \; \text{a.e.}
 \end{equation*}
 in any coordinate system,
where $a_0$ is a constant. Then, for values of the thickness~$h$ such that $h\abs{\kappa_1} < \frac{1}{2}$
and $h\abs{\kappa_2} < \frac{1}{2}$,
and for constitutive coefficients $\mu > 0$, $\mu_c > 0$, $2\lambda + \mu > 0$, and $b_1, b_2,b_3 > 0$, the functional $I$ defined in \eqref{eq:finite_strain_energy} has at least one minimizer in the set $H^1(\omega, \R^3) \times H^1(\omega, \SOdrei)$ subject to the
 boundary conditions \eqref{eq:dirichlet} and (possibly)~\eqref{eq:dirichlet_rotation}.
\end{theorem}

In \cite{birsan_neff:2023}, \citeauthor{birsan_neff:2023} showed existence of solutions
under the weaker condition $3\lambda + 2\mu > 0$.

As mentioned above, several technical changes have been made to the statement
of Theorem~\ref{thm:neff_existence} in comparison to Theorem~3.3 of~\cite{neff_existence:2020}.
When reading the following list keep in mind that the map $y_0$
that is used in~\cite{neff_existence:2020} corresponds to a representation of $\bm_0$
in local coordinates $\bm_0 \circ \tau^{-1} : \R^2 \to \R^3$ in our notation
(see also Section~\ref{sec:continuous_model}).

\begin{enumerate}[1),leftmargin=*]
 \item The original result required
 \begin{equation*}
  \det(\nabla y_0 \:\vert \:\bn_0) \ge a_0 > 0,
 \end{equation*}
 instead of
 \begin{equation*}
  \sqrt{\det\big((\nabla \bm_0)^T \nabla \bm_0 \big)} \ge a_0 > 0 ,
 \end{equation*}
 but direct calculations show that the two conditions are the same.
 The latter form makes it clear that the normal vector $\bn_0$ (and hence
 the local surface orientation) is not involved.

 \item For the same reason we write
  \begin{equation*}
   \nabla \bm_0 \in L^\infty (\omega, \R^{3 \times 2})
  \end{equation*}
  instead of
  \begin{equation*}
   (\nabla y_0 \: \vert \: \bn_0) \in L^\infty (\omega, \R^{3 \times 3}).
  \end{equation*}
  These two lines are equivalent, as $\bn_0$ is a unit vector field on $\omega$ by construction,
  and therefore in $L^\infty$ automatically.

 \item Theorem~3.3 of~\cite{neff_existence:2020} also demanded that
   $\bQ_0 \in H^1 (\omega,\SOdrei)$, but $\bQ_0$ is not independent of the choice
   of local coordinates. A thorough inspection of the proof, though, reveals that
   this condition is not actually needed.

 \item Finally, the original proof assumed that the reference shell surface deformation
   $y_0:\R^2 \to \R^3$ is continuous and injective.  However, this is not actually used in the proof,
   and we have therefore omitted in from Theorem~\ref{thm:neff_existence}.
\end{enumerate}

\section{Discretization by geometric finite elements}
\label{sec:discretization}

{The general geometrically nonlinear Cosserat shell model of the previous section
is independent of any particular finite element discretization.  However,}
the discretization of \replaced{such a}{the general geometrically nonlinear Cosserat shell} model is special for two reasons.
First, the model is based on an abstract parameter surface~$\omega$, which cannot
be represented directly by an algorithm or data structure. The second, more severe problem
is the field of microrotations $\Qe$.
Such a field lives in a non-Euclidean space, and hence it cannot be discretized by
standard finite element methods. We instead use geometric finite elements (GFE),
which generalize standard finite elements to non-Euclidean spaces~\cite{hardering_sander:2020}.
Sections~\ref{sec:geometric_interpolation_rules}~and~\ref{sec:geometric_fe_functions}
provide a brief introduction,
focusing on the relevant case of the target space $\text{SO(3)}$.

\subsection{Discretizing the shell surface}
\label{sec:discretizing_shell_surface}

In the shell model of Chapter~\ref{sec:continuous_model}, the shell surface is represented
as an immersion~$\bm$ of an abstract two-dimensional parameter manifold $\omega$ into $\R^3$.
We have already equipped $\omega$ with a conforming triangulation~$\mathcal{T}$ (Definition~\ref{def:triangulation}),
which induces a finite collection of homeomorphisms $\{ \tau_T \}_{T \in \mathcal{T}}$
from a particular triangle $\Tref \subset \R^2$ (called the reference triangle) to the
triangles $T \subset \omega$ of the triangulation.  To construct a Lagrange finite element space
in this abstract situation let $p \ge 1$ be an approximation order, and choose a global set
of corresponding Lagrange points.

\begin{definition}[Lagrange points]
\label{def:lagrange_points}
 Let $a_1, \dots, a_N$ be a finite set of points on $\omega$ such that:
 \begin{enumerate}
  \item For each triangle $T \in \mathcal{T}$, the pulled-back points
   \begin{equation*}
    A_{\textup{ref},T}
    \colonequals
    \Big\{ x \in \Tref \; : \; \tau_T(x) \in \{a_1,\dots,a_N\} \Big\}
   \end{equation*}
   form a well-posed $p$-th-order polynomial interpolation problem on $\Tref$,

  \item restricting $A_{\textup{ref},T}$ to any edge of $\Tref$ results in a
   well-posed $p$-th-order polynomial interpolation problem on that edge.
 \end{enumerate}
\end{definition}
Such a set of Lagrange points allows to define Lagrange finite elements on $\omega$ mapping
into $\R^3$. Such finite elements shall be globally continuous on~$\omega$, and polynomial
on each triangle when expressed in terms of coordinates on $\Tref$.
We call the space of such finite elements $S_h(\omega,\R^3)$,
and we use it to discretize the shell surface configuration $\bm$,
and frequently also $\bm_0$.

Conceptually, the two configurations $\bm_0$ and $\bm$ are from the same set of objects: they are both immersions from $\omega$ into $\R^3$.
However, while $\bm$ is an unknown of the problem, the reference configuration $\bm_0$ forms part of the problem specification.
It is used to define the metric and curvature terms $\ba, \bb, \bc, H$ and $K$,
and the area element~\eqref{eq:area_element} for the integration of the hyperelastic shell energy density.
Therefore, in the actual implementation it is convenient to treat $\bm_0$ as the geometric realization of the finite element grid (cf. the separation of a finite element grid into its topological and geometric aspects in Chapter~5.3 of~\cite{sander_dune:2020}).  That is, from a more algorithmic
viewpoint the elements of the finite element grid will be the sets $\bm_0(\omega|_T) \subset \R^3$,
parametrized via the maps $\bm_0 \circ \tau_T : \Tref \to \bm_0(\omega|_T)$, $T \in \mathcal{T}$.

To reap the benefits of such an approach, one needs a grid implementation that is able
to represent non-affine element geometries, because otherwise all curvature terms
appearing in the energy~\eqref{eq:finite_strain_energy} will evaluate to zero.
While vanishing curvatures are not a problem per se (the true reference surface to be simulated
may be piecewise affine after all), the true power of the shell model comes from the
proper handling of the curvature terms.

As both $\bm_0$ and $\bm$ are objects of the same type, it is natural to discretize them
in the same way, i.e., to approximate them using finite elements of the same type and order.
It is unclear whether this is really necessary, but the experiments in Section~\ref{sec:locking_test}
show a clear influence of the approximation order of $\bm_0$ on the simulation result.
With a powerful grid data structure such as the one described in~\cite{praetorius_stenger:2022}, it is also possible to approximate $\bm_0$ by functions that are more general than piecewise polynomials.

\subsection{Geometric interpolation rules}
\label{sec:geometric_interpolation_rules}

The second major challenge for finite element methods for the Cosserat shell model
is how to discretize the field of microrotations $\bQ_e$.
Such fields cannot be represented by piecewise polynomials, because nontrivial matrix-valued
polynomials cannot map into $\SOdrei$ everywhere. As a remedy, the Geometric Finite Element
method introduces generalized polynomials that do map into $\SOdrei$ everywhere, and
discretizes manifold-valued functions by piecewise such generalized polynomials.
Just as for the regular finite element method, the geometric finite element functions are first defined
on the individual grid
elements, and are then pieced together via global continuity requirements.

As the function space is nonlinear, the space of finite element functions
is not described as the span of a set of basis functions, but rather as the range of
an interpolation rule that maps a finite set of values to a function.
Several such interpolation rules have been proposed in the literature~\cite{hardering_sander:2020}.
We review the two most prominent ones.
For what follows let $T \subset \omega$ be a triangle from the triangulation $\mathcal{T}$
of the shell parameter surface~$\omega$.
This triangle is given a local coordinate system by the map $\tau_T : T_\text{ref} \to T$.
The interpolation rule will be defined with respect to these coordinates.
In a slight abuse of notation, let $a_1, \dots, a_m$ be the subset of those Lagrange points
of Definition~\ref{def:lagrange_points} contained in $T$.  By construction, there is then
a set of scalar $p$-th-order Lagrangian interpolation functions
$\lambda_i : T \to \R$, i.e., polynomials of order $p$ in the local coordinates such that
\begin{equation*}
 \lambda_i(a_j) = \delta_{ij}\quad \text{for $i,j=1,\dots,m$},
 \qquad \text{and} \qquad
 \sum_{i=1}^m \lambda_i(\eta) = 1 \quad \forall \eta \in T.
\end{equation*}
The following constructions now both generalize Lagrangian interpolation
of values in $\SOdrei$ given at the Lagrange points.

\subsubsection{Projection-based interpolation}
\label{sec:projection_bases_interpolation}

The first approach uses an embedding space of $\SOdrei$.  It was first used
for flat, one-dimensional domains in \cite{gawlik_leok:2017,romero:2004},
and later generalized to general Riemannian manifolds
$\mathcal{M}$ and higher domain dimensions in~\cite{grohs_hardering_sander_sprecher:2019}.
\begin{figure}
\begin{tikzpicture}
 \node at (0.0,2.8) {\includegraphics[width=2.8cm]{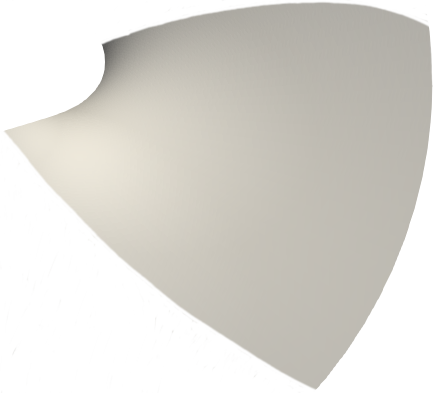}};
 \node at (-0.5,2.1) {$\omega$};

 \coordinate (a1) at (-0.5, 3.4);
 \coordinate (a2) at (1.1, 3.6);
 \coordinate (a3) at (0.6, 2.1);

 \draw[thick,fill=gray!45,fill opacity=0.3] (a1)
              .. controls ++(0.8,0.0) and (0.5,3.3) .. (a2)
              .. controls ++(-0.1,-0.8) and (1.3,2.7) .. (a3)
              .. controls ++(-1.0,0.5) and (0.0,2.7) .. (a1);

 \draw[fill=white, thick] (a1) circle (0.05cm);
 \draw[fill=white, thick] (a2) circle (0.05cm);
 \draw[fill=white, thick] (a3) circle (0.05cm);

 \node[label={[label distance=-0.1cm]left:$a_1$}] at (a1) {};
 \node[label={[label distance=-0.14cm]above:$a_2$}] at (a2) {};
 \node[label={[label distance=-0.1cm]below:$a_3$}] at (a3) {};

 \node at (0.3,2.95) {$T$};

 \draw[-stealth, very thick] (0.8,3.0) .. controls ++(1.2,0.4) and (3.1,3.3) .. (4.2,3.1);
 \node at (3.0,3.6) {$\eta \mapsto I^\text{proj}(R, \eta)$};

 \node at (4.25,2.55) {$R_1$};
 \node at (6.85,3.05) {$R_2$};
 \node at (6.05,3.8) {$R_3$};

 \coordinate (v1) at (4.513,2.7);
 \coordinate (v2) at (6.8,3.3);
 \coordinate (v3) at (5.8,3.7);
 \coordinate (p1) at (5.5,3.15);
 \coordinate (p2) at (5.25,3.475);
 \coordinate (center) at (6,2.5);
 \draw[dashed] (center) -- (p1);
 \filldraw[draw=black, fill=gray!45] (v2) -- (v1) -- (v3) -- (v2);
 \draw[fill=black] (p1) circle (0.04cm);
 \draw[fill=black] (p2) circle (0.04cm);
 \draw[fill=black] (center) circle (0.04cm);
 \draw[dashed] (p1) -- (p2);
 \draw[thick] (center) circle (1.5cm);
 \draw[fill=white, thick] (v1) circle (0.05cm);
 \draw[fill=white, thick] (v2) circle (0.05cm);
 \draw[fill=white, thick] (v3) circle (0.05cm);
 \coordinate (left) at (4.5,2.5);
 \draw[dashed, thick] (left) arc(180:0:1.5cm and 0.3cm);
 \draw[thick] (left) arc(-180:0:1.5cm and 0.3cm);
 \node at (7.5,1) {$\SOdrei$};

 \filldraw[draw=black, fill=gray!45] (2.1,0.4) -- ++(1.25,0.0) -- ++(-1.25,1.25) -- cycle;
 \draw[-stealth] (2.1,0.4) -- ++(1.5,0);
 \draw[-stealth] (2.1,0.4) -- ++(0,1.5);
 \node at (2.5,0.8) {$\Tref$};

 \draw[-stealth, very thick] (2.0,1.5) .. controls (1.8,1.8) and (1.4,2.2) .. (0.7,2.6);
 \node at (2.0,2.3) {$\mathcal G \circ \tau_T$};

\end{tikzpicture}
\caption{First-order projection-based interpolation}
\label{fig:projection_based_interpolation}
\end{figure}
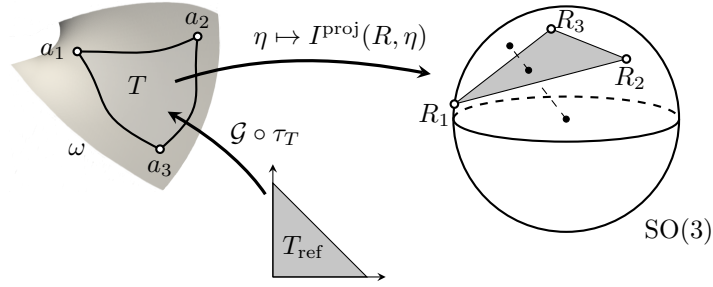
Let $\iota : \SOdrei \to \R^{3 \times 3}$ be the canonical injection, and let
$R \colonequals (R_1, \dots, R_m) \in \SOdrei^m$ be a set of values associated to the Lagrange points on $T$.
First we consider the canonical Lagrange interpolation operator $I_{\R^{3 \times 3}}$
of values embedded into $\R^{3 \times 3}$
\begin{align*}
I_{\R^{3 \times 3}}  & \; : \SOdrei^m \times \,T \to \R^{3 \times 3}, \\
I_{\R^{3 \times 3}}(R,\eta) & \coloneqq \sum_{i = 1}^m \iota(R_i)\, \lambda_i(\eta).
\end{align*}
Even though the $R_i$ are elements of $\SOdrei$, the values of $I_{\R^{3 \times 3}}(R,\cdot)$ will
in general not be in $\SOdrei$ away from the Lagrange points $a_1,\dots,a_m$.
To get $\SOdrei$-valued functions we compose $I_{\R^{3 \times 3}}$ pointwise
with the closest-point projection
\begin{equation*}
\mathcal{P} : \R^{3 \times 3} \to \SOdrei,
\qquad
\mathcal{P}(Q)\coloneqq \argmin_{P\in \SOdrei} \norm{\iota(P)-Q},
\end{equation*}
where $\norm{\cdot}$ is any unitarily invariant matrix norm.
As is well-known, this projection is simply $\operatorname{polar}(Q)$, the map onto the orthogonal factor
of the polar decomposition \cite{neff_lankeit_madeo:2014,fan_hoffman:1955}.
We therefore define $\SOdrei$-valued projection-based interpolation by composition of $I_{\R^{3 \times 3}}$
and the $\operatorname{polar}(\cdot)$ map.
\begin{definition}[Projection-based interpolation]
\label{def:projection_based_interpolation_HS}
 Let $T \subset \omega$ be a triangle.
 Let $\lambda_1,\dots,\lambda_m$ be a set of  $p$-th-order
 scalar Lagrangian shape functions on $T$, and let $R_1,\dots,R_m \in \SOdrei$ be values at the
 corresponding Lagrange points.  We call
\begin{align*}
 I^\textup{proj} & \; : \; \SOdrei^m \times \: T \to \SOdrei, \\
 I^\textup{proj}(R_1,\dots,R_m ;\eta) & \coloneqq \operatorname{polar}\Big( \sum_{i = 1}^m \iota(R_i) \:\lambda_i(\eta) \Big)
\end{align*}
$p$-th-order projection-based interpolation on $\SOdrei$.
\end{definition}
It is important to realize that there are values $R_1, \dots, R_m$ for which this construction fails.
The polar decomposition is defined for all matrices $X \in \R^{3 \times 3}$,
but it is unique only if $X$ is invertible. Furthermore, if $X$ is invertible then
$\operatorname{polar}(X)$ is in $\SOdrei$ if and only if $\det X > 0$.
However, $\det X > 0$ may not hold for all $\eta \in T$ if $X$ is constructed by Lagrange interpolation
$X = I_{\R^{3\times 3}}(R,\eta)$ in $\R^{3 \times 3}$.
It is argued in~\cite{hardering_sander:2020,grohs_hardering_sander_sprecher:2019} that
$\det I_{\R^{3\times 3}}(R,\eta) > 0$ holds for all $\eta \in T$
if the $R_1,\dots,R_m \in \SOdrei$ are close enough to each other.
Such a requirement of locality is common to all geometric finite element constructions
in spaces of positive curvature, and hardly ever poses a problem in practical computations.
Several efficient algorithms for computing the polar factor of $3 \times 3$ matrices are available in the literature \cite{higham_noferini:2016,gawlik_leok:2017b}.

The shell bending--curvature tensor $\Ke$ of Definition~\ref{def:strain_tensors}
requires first derivatives of the microrotation field $\Qe$.
Existence of the derivative for projection-based interpolation functions is a consequence of the smoothness
of Lagrange interpolation, together with the fact that the closest-point
projection onto $\SOdrei$ is infinitely differentiable for all invertible
matrices~\cite{gawlik_leok:2017b,kenney_laub:1991}.

\begin{theorem}[\textcite{grohs_hardering_sander_sprecher:2019}]
\label{thm:projection_based_interpolation_is_differentiable}
Let $R_1,\dots,R_m$ be coefficients on $\SOdrei$ with respect to
a $p$-th-order Lagrange basis $\lambda_1, \dots, \lambda_m$ on a triangle $T$.
If the function
\begin{equation*}
I^\textup{proj} \; : \; \SOdrei^m \times \: T \to \SOdrei
\end{equation*}
is defined for particular $R_1,\dots,R_m$ and $\eta$, then it is infinitely differentiable
with respect to the $R_1,\dots,R_m$ and $\eta$ there.
\end{theorem}

Computing the derivative of the projection onto $\SOdrei$ is discussed,
e.g., in~\cite{gawlik_leok:2017b}.

\subsubsection{Geodesic interpolation}
\label{sec:geodesic_interpolation}

\begin{figure}
\begin{center}
 \begin{tikzpicture}
 \node at (0.0,2.8) {\includegraphics[width=2.8cm]{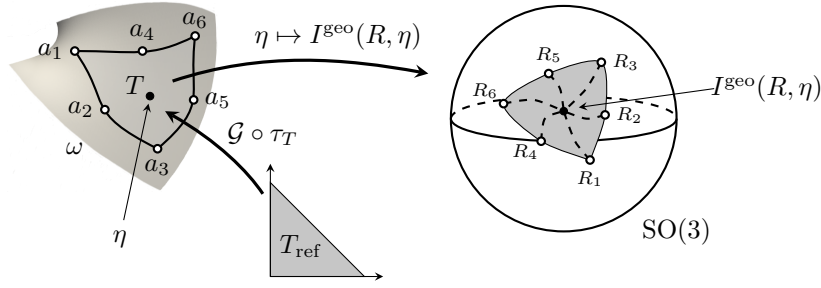}};
 \node at (-0.5,2.1) {$\omega$};

 \coordinate (a1) at (-0.5, 3.4);
 \coordinate (a2) at ( -0.1, 2.62);
 \coordinate (a3) at (0.6, 2.1);
 \coordinate (a4) at (0.4, 3.4);
 \coordinate (a5) at (1.08, 2.75);
 \coordinate (a6) at (1.1, 3.6);

 \draw[thick,fill=gray!45,fill opacity=0.3] (a1)
              .. controls ++(0.8,0.0) and (0.5,3.3) .. (a6)
              .. controls ++(-0.1,-0.8) and (1.3,2.7) .. (a3)
              .. controls ++(-1.0,0.5) and (0.0,2.7) .. (a1);

 \draw[fill=white, thick] (a1) circle (0.05cm);
 \draw[fill=white, thick] (a2) circle (0.05cm);
 \draw[fill=white, thick] (a3) circle (0.05cm);
 \draw[fill=white, thick] (a4) circle (0.05cm);
 \draw[fill=white, thick] (a5) circle (0.05cm);
 \draw[fill=white, thick] (a6) circle (0.05cm);

 \node[label={[label distance=-0.1cm]left:$a_1$}] at (a1) {};
 \node[label={[label distance=-0.1cm]left:$a_2$}] at (a2) {};
 \node[label={[label distance=-0.1cm]below:$a_3$}] at (a3) {};
 \node[label={[label distance=-0.1cm]above:$a_4$}] at (a4) {};
 \node[label={[label distance=-0.1cm]right:$a_5$}] at (a5) {};
 \node[label={[label distance=-0.14cm]above:$a_6$}] at (a6) {};

 \node at (0.3,2.95) {$T$};

 \draw[-stealth, very thick] (0.8,3.0) .. controls ++(1.2,0.4) and (3.1,3.3) .. (4.2,3.1);
 \node at (3.0,3.6) {$\eta \mapsto I^\text{geo}(R, \eta)$};

 \filldraw[draw=black, fill=gray!45] (2.1,0.4) -- ++(1.25,0.0) -- ++(-1.25,1.25) -- cycle;
 \node at (2.5,0.8) {$\Tref$};
 \draw[-stealth] (2.1,0.4) -- ++(1.5,0);
 \draw[-stealth] (2.1,0.4) -- ++(0,1.5);

 \draw[-stealth, very thick] (2.0,1.5) .. controls (1.8,1.8) and (1.4,2.2) .. (0.7,2.6);
 \node at (2.0,2.3) {$\mathcal G \circ \tau_T$};

 \draw[-stealth]  (0.15,1.15) -- (0.48,2.7);
 \draw[fill=black] (0.5,2.8) circle (0.05cm);
 \node at (0.1,0.9) {$\eta$};

 \coordinate (center) at (6,2.5);
 \draw[thick] (center) circle (1.5cm);
 \coordinate (left) at (4.5,2.5);
 \draw[dashed, thick] (left) arc(180:0:1.5cm and 0.3cm);
 \draw[thick] (left) arc(-180:0:1.5cm and 0.3cm);
 \node at (7.5,1) {$\SOdrei$};

 \coordinate (R1) at (6.35,1.95);
 \node[label={[label distance=-0.1cm]below:{\scriptsize $R_1$}}] at (R1) {};
 \coordinate (R1c) at (6.15,2.2);

 \coordinate (R2) at (6.55,2.55);
 \node[label={[label distance=-0.1cm]right:{\scriptsize $R_2$}}] at (R2) {};
 \coordinate (R2c) at (6.3,2.5);

 \coordinate (R3) at (6.5,3.25);
 \node[label={[label distance=-0.1cm]right:{\scriptsize $R_3$}}] at (R3) {};
 \coordinate (R3c) at (6.3,2.9);

 \coordinate (R4) at (5.7,2.2);
 \node[label={[label distance=-0.3cm]below left:{\scriptsize $R_4$}}] at (R4) {};
 \coordinate (R4c) at (5.75,2.5);

 \coordinate (R5) at (5.8,3.1);
 \node[label={[label distance=-0.1cm]above:{\scriptsize $R_5$}}] at (R5) {};
 \coordinate (R5c) at (5.95,2.9);

 \coordinate (R6) at (5.2,2.7);
 \node[label={[label distance=-0.25cm]above left:{\scriptsize $R_6$}}] at (R6) {};
 \coordinate (R6c) at (5.6,2.75);

 \coordinate (center) at (6,2.6);
 \coordinate (centerSO3) at (6,2.5);
 \draw[fill=black] (centerSO3) circle (0.03cm);
 \draw[fill=gray!45] plot [smooth] coordinates {(R1)(R2)(R3)(R5)(R6)(R4)(R1)};
 \draw [dashed, thick] plot [smooth] coordinates {(R1)(R1c)(center)};
 \draw [dashed, thick] plot [smooth] coordinates {(R2)(R2c)(center)};
 \draw [dashed, thick] plot [smooth] coordinates {(R3)(R3c)(center)};
 \draw [dashed, thick] plot [smooth] coordinates {(R4)(R4c)(center)};
 \draw [dashed, thick] plot [smooth] coordinates {(R5)(R5c)(center)};
 \draw [dashed, thick] plot [smooth] coordinates {(R6)(R6c)(center)};
 \draw[fill=white, thick] (R1) circle (0.05cm);
 \draw[fill=white, thick] (R2) circle (0.05cm);
 \draw[fill=white, thick] (R3) circle (0.05cm);
 \draw[fill=white, thick] (R4) circle (0.05cm);
 \draw[fill=white, thick] (R5) circle (0.05cm);
 \draw[fill=white, thick] (R6) circle (0.05cm);
 \draw[fill=black] (center) circle (0.05cm);
 \draw[-stealth] (8,2.9) -- (6.2,2.65);
 \node at (8.7,2.9) {$I^\text{geo}(R, \eta)$};
 \end{tikzpicture}
\end{center}
\caption{Second-order geodesic interpolation}
 \label{fig:geodesic_interpolation}
\end{figure}

The second construction works without an embedding space.
Recall that the usual Lagrange interpolation of values $v_1,\dots,v_m$ in $\R$ can be written as a minimization
problem
\begin{equation*}
 \eta \mapsto \argmin_{w \in \R} \sum_{i=1}^m \lambda_i(\eta) \abs{v_i -w}^2
\end{equation*}
for each $\eta \in T$.
This formulation can be generalized to values in $\SOdrei$
by replacing the absolute value by the canonical geodesic distance on $\SOdrei$~\cite{huynh:2009}
\begin{equation*}
 \dist(Q_1,Q_2)
 \coloneqq
 \norm{\log (Q_1 Q_2^T)}.
\end{equation*}

\begin{definition}[Geodesic interpolation \cite{sander:2013}]
\label{def:geodesic_interpolation}
 Let $\lambda_1, \dots \lambda_m$ be a set of  $p$-th-order
 scalar Lagrange functions on $T$, and let $R_i \in \SOdrei$,
 $i=1,\dots,m$, be values at the corresponding Lagrange points.  We call
\begin{align}
 \nonumber
 I^\textup{geo} & \; : \; \SOdrei^m \times \: T \to \SOdrei \\
 \label{eq:geodesic_interpolation} I^\textup{geo}(R_1,\dots,R_m;\eta) & = \argmin_{R \in \SOdrei}  \sum_{i=1}^m \lambda_i(\eta) \dist(R_i,R)^2
\end{align}
$p$-th-order geodesic interpolation on $\SOdrei$.
\end{definition}
Again, this definition is only well-posed if the coefficients $R_1, \dots, R_m$  are close enough to each
other on $\SOdrei$.  Different proofs for this can be
found in~\cite{karcher:1977,kendall:1990,groisser:2004,sander:2013}.
A simple example for first-order geodesic finite elements on triangles is given in~\cite{sander_neff_birsan:2016}.
\bigskip

As in the case of projection-based interpolation one can show that the functions created by geodesic interpolation are smooth.
This fact follows directly from the implicit function theorem.
\begin{theorem}[\textcite{sander:2012,sander:2013}]
\label{thm:geodesic_interpolation_is_differentiable}
Let $R_1,\dots,R_m$ be coefficients on $\SOdrei$ with respect to
a $p$-th-order Lagrange basis $\lambda_1,\dots, \lambda_m$ on a triangle $T \subset \omega$.
If the function
\begin{equation*}
I^\textup{geo} \; : \; \SOdrei^m \times \:T \to \SOdrei
\end{equation*}
is well-posed in the sense that the minimization Problem \eqref{eq:geodesic_interpolation} defining $I^\textup{geo}$ has a unique solution at $R_1, \dots, R_m$ and $\eta$, then it is infinitely differentiable with respect to the $R_i$ and $\eta$.
\end{theorem}
The papers \cite{sander:2012,sander:2013} explain how the derivatives can be computed in practice.

\subsubsection{Relationship between geodesic and projection-based interpolation}
There is a surprising connection between the geodesic and the projection-based interpolation rule.
As observed independently by \cite{gawlik_leok:2017} and \cite{sprecher:2016}, we recover
the projection-based interpolation if we replace the geodesic distance in \eqref{eq:geodesic_interpolation} by the Euclidean distance of $\R^{3 \times 3}$ defined by the Frobenius norm $\norm{\cdot}_F$:
\begin{align}\label{eq:connection_projection_based_geodesic_interpolation}
 \argmin_{R \in \SOdrei} \sum_{i=1}^m \lambda_i(\eta) \norm{R_i - R}^2_{F}
 & =
 \argmin_{R \in \SOdrei} \bigg(\norm{R}^2_{F}-2\bigg\langle R,\sum_{i=1}^m \lambda_i(\eta)R_i\bigg\rangle\bigg) \\\nonumber
 & =
 \argmin_{R \in \SOdrei} \bigg\lVert R-\sum_{i=1}^m \lambda_i(\eta)R_i\bigg\rVert^2_F\\
 & =
 \operatorname{polar}\bigg(\sum_{i=1}^m \lambda_i(\eta)R_i\bigg)\nonumber \\
 & =
 I^\textup{proj} (R_1,\dots, R_m;\eta).\nonumber
\end{align}
This does not mean, however, that projection-based finite elements are the same as geodesic finite elements.
We will use their relationship, however, to prove completeness of projection-based geodesic finite element spaces
in Section~\ref{sec:existence}.

\subsection{Geometric finite element functions}
\label{sec:geometric_fe_functions}

The interpolation functions of the previous section are used to construct
generalizations of Lagrange finite element functions with values in $\SOdrei$,
defined on the triangulated $2$-manifold $\omega$.

Just like standard Lagrange finite element spaces, global geometric finite element spaces
are constructed by connecting the functions on the individual elements by
continuity conditions.  By the definition of the global
Lagrange point set (Definition~\ref{def:lagrange_points}), such continuity conditions can be fulfilled by the
generalized polynomials of the previous section.  Let $S_h^\textup{proj}(T)$ and
$S_h^\textup{geo}(T)$ be the sets of functions $T \to \SOdrei$ created by projection-based
and geodesic interpolation of a particular polynomial order $p$, respectively.
Then, we define the global spaces
\begin{equation}\label{eq:S_geo}
S_h^{\textrm{proj}}(\omega,\SOdrei)
\coloneqq
\Big\{R_h \in C(\omega,\SOdrei) \; : \; R_h|_{T} \in S_h^{\textrm{proj}}(T,\SOdrei) \quad \forall T \in \mathcal{T} \Big\},
\end{equation}
and likewise for $S_h^{\textrm{geo}}(\omega,\SOdrei)$.
Note that the spaces $S_h^\textup{proj}$ are nested in the sense that low-order spaces
are contained in higher-order ones, but the spaces $S_h^\textup{geo}$ are not.
Neither of the two families of spaces is nested under uniform grid refinement.

A central feature of the construction of geometric finite elements is their first-order Sobolev conformity.
Indeed, recall the definition of the nonlinear function space $H^1(\omega, \SOdrei)$ from Section~\ref{sec:existence_minimizers}.
Then we have the following inclusions:
\begin{lemma}\label{lem:gfe_smoothness}
 Let $S_h^{\textup{proj}}(\omega,\SOdrei)$ and $S_h^\textup{geo} (\omega, \SOdrei)$ be
 geometric finite element spaces of order $p \ge 1$.
 Let $P : A \mapsto \operatorname{polar}(A)$, i.e., the projection from $\R^{3 \times 3}$ onto~$\text{O}(3)$.
 \begin{enumerate}[i)]
  \item \label{item:conformity_projection_based_fe}
    If the pointwise operator norm of the differential $dP(A):\R^{3 \times 3}\to T_{P(A)}\SOdrei$ is globally bounded, then
		\begin{equation*}
		S_h^{\textup{proj}}(\omega,\SOdrei) \subset H^1(\omega,\SOdrei).
		\end{equation*}
		\item \label{item:conformity_geodesic_fe}
		 $S_h^\textup{geo} (\omega, \SOdrei)\subset H^1(\omega,\SOdrei)$.
	\end{enumerate}
\end{lemma}
Lemma~\ref{lem:gfe_smoothness}\,\ref{item:conformity_projection_based_fe} is proved in \cite[Sec.\,1.1]{grohs_hardering_sander_sprecher:2019}. Part\,\ref{item:conformity_geodesic_fe} corresponds to Theorem~3.1 in~\cite{sander:2012}.
The operator norm of the differential $dP(A):\R^{3 \times 3}\to T_{P(A)}\SOdrei$ is bounded in each compact subset
of the set of invertible real matrices, because, as \textcite{kenney_laub:1991} show,
we have $\norm{dP(A)} = \frac{2}{\sigma_{\text{min}} + \sigma_{\text{min}-1}}$,
where $\sigma_\text{min}$ and $\sigma_{\text{min}-1}$ are the two smallest singular values of $A$.

As a consequence of Lemma~\ref{lem:gfe_smoothness}, discrete approximations
$(\bm_h, \Qeh) \in S_h(\omega,\R^3) \times S_h(\omega,\SOdrei)$ are elements
of the space $H^1(\omega,\R^3) \times H^1(\omega,\SOdrei)$,
in which the Cosserat shell problem is well-posed (Theorem~\ref{thm:neff_existence}).
(Here we have used $S_h(\omega,\SOdrei)$ to denote either $S_h^\text{proj}(\omega,\SOdrei)$ or $S_h^\text{geo}(\omega,\SOdrei)$.)
This means that the hyperelastic shell energy~\eqref{eq:finite_strain_energy}
can be directly evaluated for geometric finite element functions,
which facilitates the mathematical understanding of the discrete shell model considerably.

\begin{remark}
\label{rem:global_space_definitions}
 For later reference in Lemma~\ref{lem:pointwise_limits} we note that there is also
 a direct definition of the global geodesic finite element spaces,
 which incorporates continuity on the level of the scalar Lagrange basis
 used in Definition~\ref{def:geodesic_interpolation}.
 Indeed, let $\lambda_1, \dots, \lambda_N$ now be a global scalar Lagrange finite element basis
 defined with respect to  the global set of Lagrange points $a_1,\dots, a_N$
 of Definition~\ref{def:lagrange_points}.
 Then it is easily seen that
 \begin{multline*}
  S_h^\textup{geo}(\omega,\SOdrei)
  =
  \Big\{ \bR_h \in L^2(\omega, \R^{3 \times 3}) \; : \; \exists R_1,\dots, R_N \in \SOdrei \\
  \text{s.t.} \quad \bR_h(\eta) = \argmin \sum_{i=1}^N \lambda_i(\eta) \dist(R_i,R)^2
  \quad \forall \eta \in \omega \Big\}.
 \end{multline*}
 The minimization problem in
 this definition looks formally like the one from the original definition~\eqref{eq:geodesic_interpolation},
 but the coefficients and Lagrange basis are now global.  Equivalence follows
 from the fact that for each point $\eta$ in a triangle $T \subset \omega$ of $\mathcal{T}$,
 the set of nonzero global Lagrange functions at~$\eta$ is exactly the set of local
 Lagrange functions used in~\eqref{eq:geodesic_interpolation}.
\end{remark}

Finally, we point out the following equivariance result,
which is of central importance for applications in mechanics.
It implies that discretizations of frame-indifferent shell models are frame-indifferent as well.
The result for $S_h^\textup{geo}$ follows directly from the fact that geodesic interpolation
is defined using metric quantities alone, and is hence invariant under isometries~\cite{sander:2012}.
The proof for projection-based finite elements is given in~\cite{grohs_hardering_sander_sprecher:2019}.
\begin{lemma}[\textcite{grohs_hardering_sander_sprecher:2019}]
\label{lem:equivariance_of_gfe}
 Let $\textup{O}(3)$ be the orthogonal group on $\R^3$, which acts isometrically on $\SOdrei$ by left multiplication.
 Pick any element $Q \in O(3)$.  For any geodesic or projection-based finite element
 function $\bR_h \in S_h(\omega,\SOdrei)$ we define
 $Q\bR_h : \omega \to \SOdrei$ by $(Q\bR_h)(\eta) = Q(\bR_h(\eta))$ for all $\eta \in \omega$.
 Then $Q\bR_h \in S_h(\omega,\SOdrei)$.
\end{lemma}

Optimal approximation error bounds for geometric finite element function spaces
with values in general Riemannian manifolds~$\mathcal{M}$ have been proven
in~\cite{grohs_hardering_sander:2015,grohs_hardering_sander_sprecher:2019,hardering:2018,hardering:2018Arxiv}.
The same works show optimal discretization error bounds for certain elliptic problems.
The application of those abstract results to the energy functionals considered in this paper
will be left for future work.

\section{Discrete and algebraic Cosserat shell problems}
\label{sec:discete_and_algebraic_problems}

We now apply the geometric finite element method to the Cosserat shell model of Section~\ref{sec:continuous_model}.
As the finite element grid we use the triangulation $\mathcal{T}$ of $\omega$ (which describes the
grid topology) together with the geometry given by the reference immersion~$\bm_0$.
For each triangle $T$ of $\mathcal{T}$ we then have a mapping from $\Tref$ to $\bm_0(\omega|_T)$
given by $\bm_0 \circ \tau_T^{-1}$. This is the standard mapping from the reference element
to the grid elements used in the finite element method.

Reaping the full power of the shell model requires a finite element grid implementation that supports
nonplanar grid elements because otherwise all terms in \eqref{eq:finite_strain_energy} involving the curvatures $K$, $H$, and $\bb$
of $\bm_0(\omega)$ would vanish.

\subsection{The discrete problem}
\label{sec:discrete_problem}

The hyperelastic shell energy functional $I$ given in~\eqref{eq:finite_strain_energy} is
defined on the product of the spaces $H^1(\omega,\R^3)$ and $ H^1(\omega,\SOdrei)$.
The first factor is a standard Sobolev space of vector-valued functions on a manifold~\cite{wloka:1987}.
For its discretization we use
the space $S_h(\omega,\R^3)$ of conforming Lagrange finite elements with values in $\R^3$
introduced in Section~\ref{sec:discretizing_shell_surface}.
We consider Lagrange spaces of any order $p_1 \ge 1$, but omit the order from the notation
for simplicity.
For the microrotation field $\Qe : \omega \to \text{SO(3)}$ we use either type of
geometric finite elements described in the previous chapter.
Denote by $S_h(\omega,\SOdrei)$
the $p_2$-th-order geometric finite element space of functions on $\omega$ with respect to the grid $\mathcal{T}$ and with values in $\SOdrei$.
We omit both the order $p_2 \ge 1$ and whether it is a projection-based or geodesic
finite element space from the notation.

In the following we write $\bm_h$ for discrete displacement functions from $S_h(\omega,\R^3)$
and $\Qeh$ for discrete microrotations from $S_h(\omega,\SOdrei)$.
As is widespread in the finite element literature we use a subscript~$h$ to denote
finite element quantities. This~$h$ should not be confused with the thickness parameter
that appears in the shell model.

The function spaces may be restricted by Dirichlet conditions.
For simplicity we assume that the grid resolves the Dirichlet set~$\gamma_d$ of $\omega$.
Let $\bm^*_h \in S_h(\omega,\R^3)$ be a finite element approximation
of the Dirichlet value function $\bm^* : \omega \to \R^3$ of \eqref{eq:dirichlet}.
Then we demand that the discrete displacement $\bm_h$ fulfills the condition
\begin{equation}
\label{eq:discrete_dirichlet_condition}
 \bm_h = \bm^*_h
 \qquad
 \text{on $ \gamma_d$}.
\end{equation}
For the microrotations $\Qe$ we can define discrete approximations of the Dirichlet conditions from Section~\ref{sec:existence_minimizers}.
If the continuous model requires to match a Dirichlet value function $\Qe^* : \omega \to \SOdrei$
on the set $\gamma_d$, then one can select a geometric finite element
approximation $\Qeh^* \in S_h(\omega,\SOdrei)$ of $\Qe^*$ and require
\begin{equation}
\label{eq:discrete_rotation}
\Qeh = \Qeh^*
\qquad
\text{on $\gamma_d$}.
\end{equation}
The case with Dirichlet values on a different part of $\omega$ is handled analogously.
Also, the existence proof below works if there are no prescribed microrotations at all.

For the formulation of the discrete shell problem we exploit that both projection-based
and geodesic approximation spaces $S_h(\omega, \SOdrei)$ are contained in $H^1$ (Lemma~\ref{lem:gfe_smoothness}).
Together with the corresponding well-known result for Euclidean finite elements
(see, e.g., \cite[Satz~5.2]{braess:2013}) we can conclude that
the Cosserat shell energy functional $I$ is well-defined on the product space
$S_h(\omega,\R^3) \times S_h(\omega,\SOdrei)$ for all orders $p_1,p_2 \in \mathbb{N}$.
A suitable discrete approximation of the Cosserat shell model therefore consists
of the unmodified energy functional $I$ restricted to such a space.

\begin{problem}[Discrete Cosserat shell problem]\label{prob:discrete_shell_problem}
Find a pair of functions $(\bm_h, \Qeh)$ with
$\bm_h \in S_h(\omega,\R^3)$ and $\Qeh \in S_h(\omega,\SOdrei)$ that minimizes
the shell energy functional~$I$ given in~\eqref{eq:finite_strain_energy},
subject to the constraints~\eqref{eq:discrete_dirichlet_condition} (and possibly~\eqref{eq:discrete_rotation}) on $\gamma_d$.
\end{problem}

Note that Problem~\ref{prob:discrete_shell_problem} (without the external loads)
is frame-indifferent in the sense that
\begin{equation}\label{eq:discrete_frame_indifference}
I(R\bm_h, R \Qeh) = I (\bm_h, \Qeh)
\end{equation}
for any $\bm_h \in S_h(\omega,\R^3)$, $\Qeh \in S_h(\omega,\SOdrei)$, and any orthogonal matrix $R$.
Indeed, we have $R \bm_h \in S_h(\omega,\R^3)$, and by Lemma~\ref{lem:equivariance_of_gfe}
we also have $R\Qeh \in S_h(\omega,\SOdrei)$.
With this, \eqref{eq:discrete_frame_indifference} is simply the frame indifference~\eqref{eq:frame_indifference}
of the original model.

\subsection{Existence of solutions of the discrete problem}
\label{sec:existence}
In the following, we prove the existence of solutions of the discrete Cosserat shell Problem \ref{prob:discrete_shell_problem}. For this, we show that the energy functional \eqref{eq:finite_strain_energy} has global minimizers in the discrete product space $S_{h}(\omega,\R^3) \times S_{h}(\omega,\SOdrei)$
when subject to suitable Dirichlet conditions.
Such an existence result is obvious for standard finite element methods, where the finite element space
is a finite-dimensional vector subspace of $H^1$, and hence closed.  The closedness of the nonlinear
geometric finite element space $S_{h}(\omega,\SOdrei)$ is less obvious, however, and we show it
explicitly in Lemma~\ref{lem:pointwise_limits}.
As in Section~\ref{sec:existence_minimizers} we only show existence for the
Dirichlet condition \eqref{eq:discrete_rotation}, i.e., the case of $\Qeh$ being prescribed on $\gamma_d$.
Solutions also exist without any Dirichlet conditions for the microrotations.
The necessary modifications of the proof are discussed in \cite{neff_existence:2020}.

The existence proof uses the direct method in the calculus of variations.  It is
a modification of the proof in \cite{neff_existence:2020} for existence of minimizers
of the continuous model. The crucial fact is that geometric finite element functions are elements
of $H^1(\omega, \SOdrei)$. Therefore, all properties of the energy functional~$I$ proved
in \cite{neff_existence:2020} for the application of the direct method in that space also hold
when considering geometric finite element functions. The important new step is to show that
weakly $H^1$-convergent sequences of such functions have a subsequence that converges to a
limit in the geometric finite element space $S_h(\omega,\SOdrei)$. We will do this in Lemma~\ref{lem:pointwise_limits}.
The following existence theorem is then very similar to its continuous conterpart Theorem~\ref{thm:neff_existence}.
Note that the case $\mu_c = 0$ is again not covered---it requires a different proof,
and a slight modification of the model.

\begin{theorem}\label{thm:neff_existence_discrete}
 Let the external loads satisfy
 \begin{equation}\label{eq:external-loads}
  \boldf \in L^2(\omega, \R^3)
  \qquad \text{and} \qquad
  \bt \in L^2 (\gamma_t , \R^3 ).
 \end{equation}
 Assume that the reference configuration $\bm_0 \in H^1(\omega, \R^3)$ is an immersion
 that is in~$H^2$ on each triangle, and
 \begin{equation*}
  \nabla \bm_0 \in L^\infty (\omega,\R^{3 \times 2}),
  \qquad
  \sqrt{\det\big((\nabla \bm_0)^T \nabla \bm_0\big)} \ge a_0 > 0, \; a.e.,
 \end{equation*}
 where $a_0$ is a constant. Then, for values of the thickness $h$ such that $h\abs{\kappa_1} < \frac{1}{2}$ and $h\abs{\kappa_2} < \frac{1}{2}$,
 and for constitutive coefficients such that $\mu > 0$, $\mu_c > 0$, $2\lambda + \mu > 0$,
 and $b_1, b_2,b_3 > 0$, the minimization
 problem~\ref{prob:discrete_shell_problem} admits at least one solution in the set
 $S_h(\omega,\R^3) \times S_h(\omega,\SOdrei)$, subject to
 \begin{equation}
 \label{eq:discrete_dirichlet_condition_theorem}
  \bm_h = \bm^*_h \quad \text{and} \quad \Qeh = \Qeh^* \quad \text{on $\gamma_d$}.
 \end{equation}
\end{theorem}

\begin{proof}
As the proof of Theorem~\ref{thm:neff_existence_discrete} is so close to the one of Theorem~3.3
in~\cite{neff_existence:2020}, most of the major steps will only be sketched.
\begin{enumerate}[1),leftmargin=*]
		\item
		In the first step one shows that the functional $I$ of \eqref{eq:finite_strain_energy} is bounded from below. More formally, one shows that there are constants $c_1 > 0$ and $c_2 \in \R$ such that
		\begin{align}\label{eq:bounded_below}
			I(\bm_h, \Qeh) & \ge c_1 \norm{\bm_h - \bm_h^*}^2_{H^1(\omega)} + c_2
		\end{align}
                for all $\bm_h \in S_h(\omega,\R^3)$ and $\Qeh \in S_h(\omega, \SOdrei)$.
		The trick is to see that the internal energy can compensate the unbounded external
		load functional~$\Pi_\text{ext}$ of~\eqref{eq:Pi}.%
\footnote{This part of the original proof uses the field $\bQ_0$, but can be made independent
of it by using the invariance of the matrix norm under orthogonal transformations.}

		\item It then follows that there exists an infimizing sequence $(\bm_h^k, \Qeh^k)$
		in the set $\mathcal{A}_h$ that consists of all functions in
		$S_h(\omega,\R^3) \times S_h(\omega,\SOdrei)$ that comply with the Dirichlet
		conditions~\eqref{eq:discrete_dirichlet_condition_theorem}:
		\begin{equation*}
		\lim_{k \to \infty} I(\bm_h^k, \Qeh^k) = \inf_{(\bm_h, \Qeh)  \in \mathcal{A}_h} I(\bm_h, \Qeh).
		\end{equation*}
		Since the boundary value functions $\bm_h^*$, $\Qeh^*$ are assumed to be in $\mathcal{A}_h$
		we can choose the sequence such that
		\begin{equation}\label{eq:bounded_above}
			I(\bm_h^k, \Qeh^k) \le I (\bm^*_h, \Qeh^*) < \infty, \quad \forall \, k \ge 1.
		\end{equation}
		\item From \eqref{eq:bounded_below} and \eqref{eq:bounded_above} the sequence $(\bm_h^k)$ is bounded in $H^1(\omega, \R^3)$. Therefore, one can extract a subsequence that converges weakly in $H^1(\omega, \R^3)$, and by Rellich's selection principle this subsequence also converges strongly in $L^2(\omega, \R^3)$.
		
		Let $\widehat{\bm}_h \in H^1(\omega, \R^3)$ be the limit function. As $S_h(\omega, \R^3)$ is a finite-dimensional vector subspace of $H^1(\omega, \R^3)$ we get $\widehat{\bm}_h \in S_h(\omega, \R^3)$.
		Also, $\widehat{\bm}_h$ complies with the Dirichlet conditions.

		\item Likewise, one then shows that there is a $\widehat{\bQ}_{e,h} \in H^1(\omega, \SOdrei)$
		such that for a subsequence
		$$
		\Qeh^k \rightharpoonup \widehat{\boldsymbol{Q}}_{e,h} \quad\text{in }H^1(\omega, \R^{3 \times 3})
		$$
		and
		$$
		\Qeh^k \to \widehat{\boldsymbol{Q}}_{e,h} \quad\text{in }L^2(\omega, \R^{3 \times 3}).
		$$

 \item \label{item:pointwise_convergent_subsequence}
  By~\cite[Korollar~VI.2.7]{elstrodt:2011}, the $L^2$ convergence of the subsequence implies that there is a subsubsequence
  (not relabeled) $(\Qeh^k)$ that converges pointwise a.e.\ to the same
  $\widehat{\boldsymbol{Q}}_{e,h} \in H^1(\omega, \SOdrei)$.
   Lemma~\ref{lem:pointwise_limits} below then shows that the limit is again in $S_h(\omega,\SOdrei)$,
   as required.

 \item To show lower semi-continuity of~$I$, the energy density $W \colonequals W_\text{memb} + W_\text{bend}$ of~\eqref{eq:Wmemb}
  and~\eqref{eq:Wbend}
  is considered as a function of the
  strains $\Ee$ and $\Ke$.  One shows that the sequences of strains $(\Eek)$
  and $(\Kek)$ corresponding to $(\bm_h^k, \Qeh^k)$ converge weakly
  in $L^2(\omega,\R^{3 \times 3})$ to limit strains $\hatEe$, $\hatKe$,
  and that these limits correspond to the limit $(\widehat \bm_h, \widehat{\bQ}_{e})$
  of $(\bm_h^k, \Qeh^k)$
  via the strain formulas of Definition~\ref{def:strain_tensors}.%
  \footnote{Here again, the original proof in~\cite{neff_existence:2020} uses the orientation field $\bQ_0$,
  but the argument can be reformulated without it.}

 \item The energy density is convex in the strains (albeit nonconvex in $\nabla \bm$ and $\Qe$), which implies the
  lower semi-continuity of the internal energy functional
  \begin{equation*}
   \int_\omega W(\hatEe, \hatKe)\,d\omega
   \le
   \liminf_{k \to \infty} \int_\omega W(\Eek, \Kek)\,d\omega.
  \end{equation*}

 \item With~\eqref{eq:external-loads}, the $L^2$-convergence of $\bm_h^k$ and $\Qeh^k$, and the
  continuity of the external load potential we therefore get
  \begin{equation*}
   I(\widehat{\bm}_h, \widehat{\bQ}_{e,h}) \le \liminf_{k \to \infty} I(\bm_h^k, \bQ^k_{e,h}).
  \end{equation*}
 With the standard argument of the direct method one can then conclude
 that $(\widehat{\bm}_h, \widehat{\bQ}_{e,h}) \in S_h(\omega, \R^3) \times S_h(\omega, \SOdrei)$
 is a minimizing pair. \qedhere
\end{enumerate}
\end{proof}

This existence proof requires in Step~\ref{item:pointwise_convergent_subsequence} that pointwise limits of
geometric finite element functions are again geometric finite element functions.
This is what we prove next. We prove it for geodesic finite elements only,
but by relationship~\eqref{eq:connection_projection_based_geodesic_interpolation}
we also obtain the same result for projection-based finite elements.
For future reference we temporarily drop the focus on $\SOdrei$.  Instead,
we prove the more general result that the geodesic finite element space
$S^\text{geo}_{h}(\omega, (\mathcal{M},d))$ is complete under pointwise convergence for any complete
metric space $\mathcal{M}$ with distance $d(\cdot,\cdot)$. This trivially includes the case
$\mathcal{M} = \SOdrei$.  To streamline the notation we write geodesic finite elements functions
without a subscript~$h$ in this result.

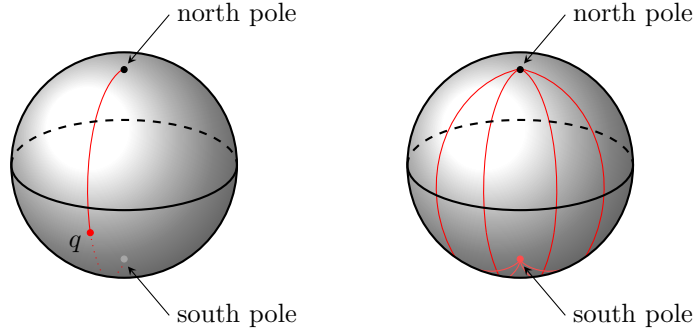
\begin{figure}
 \begin{center}
 		\begin{tikzpicture}
 		\begin{scope}
 			\coordinate (center) at (6,2.5);
 			\shade [ball color=white] (center) circle [radius=1.5cm];
 			\coordinate (nP) at (6,3.77);
 			\coordinate (redP) at (5.55,1.6);
 			\node at (5.35,1.45) {$q$};
 			\draw[draw=red, fill=red] (redP) circle (0.04cm);
 			\draw[draw=red] (redP) arc(200:100:0.6cm and 1.65cm);
 			\draw[dotted, draw=red] (redP) arc(200:230:0.6cm and 1.4cm);
 			\draw[dotted, draw=red] (6,1.23) arc(355:330:1cm and 0.5cm);
 			\draw[-stealth] (6.6,4.5) -- (6.05,3.85);
 			\node at (7.5,4.5) {north pole};
 			\draw[fill=black] (nP) circle (0.04cm);
 			\draw[-stealth] (6.6,0.5) -- (6.05,1.15);
 			\node at (7.5,0.5) {south pole};
 			\draw[draw=gray!70, fill=gray!70] (6,1.25) circle (0.04cm);
 			\draw[thick] (center) circle (1.5cm);
 			\coordinate (left) at (4.5,2.5);
 			\draw[dashed, thick] (left) arc(180:0:1.5cm and 0.6cm);
 			\draw[thick] (left) arc(-180:0:1.5cm and 0.6cm);
		\end{scope}

		\begin{scope}[xshift=0.1\textwidth]
 			\coordinate (center) at (10,2.5);
 			\shade [ball color=white] (center) circle [radius=1.5cm];
 			\coordinate (nP) at (10,3.77);
 			\coordinate (sP) at (10,1.23);
 			\coordinate (redP) at (9.55,1.6);
 			\coordinate (redP2) at (10.45,1.6);

 			\draw[draw=red] (redP) arc(200:100:0.6cm and 1.65cm); 
 			\draw[draw=red] (redP) arc(200:228:0.6cm and 1.4cm); 
 			\draw[draw=red!70] (sP) arc(355:329:1cm and 0.5cm);

 			\draw[draw=red] (redP2) arc(-20:80:0.6cm and 1.65cm); 
 			\draw[draw=red] (redP2) arc(340:312:0.6cm and 1.4cm); 
 			\draw[draw=red!70] (sP) arc(185:211:1cm and 0.5cm); 

 			\draw[draw=red] (10,3.78) arc(80:-36:1.35cm and 1.6cm); 
 			\draw[draw=red!70] (sP) arc(220:290:0.5cm and 0.4cm); 

 			\draw[draw=red] (10,3.78) arc(100:216:1.35cm and 1.6cm); 
 			\draw[draw=red!70] (sP) arc(-40:-110:0.5cm and 0.4cm); 

 			\draw[-stealth] (10.6,4.5) -- (10.05,3.85);
 			\node at (11.5,4.5) {north pole};
 			\draw[fill=black] (nP) circle (0.04cm);
 			\draw[-stealth] (10.6,0.5) -- (10.05,1.15);
 			\node at (11.5,0.5) {south pole};
 			\draw[draw=red!70, fill=red!70] (10,1.25) circle (0.04cm);
 			\draw[thick] (center) circle (1.5cm);
 			\coordinate (left) at (8.5,2.5);
 			\draw[dashed, thick] (left) arc(180:0:1.5cm and 0.6cm);
 			\draw[thick] (left) arc(-180:0:1.5cm and 0.6cm);
		\end{scope}
 		\end{tikzpicture}
 \end{center}
 \caption{The possible non-uniqueness of geodesic interpolation, shown for a first-order geodesic
   interpolation function from an interval to the sphere.
   Left: There is a unique such interpolation function
   from the north pole to any other point $q$, no matter how close to the south pole.
   Right: There are infinitely many interpolation functions from the north pole
   to the south pole.}
 \label{fig:limit_nonuniqueness}
\end{figure}

One apparent difficulty is that limits of uniquely defined geodesic interpolation functions
may not be uniquely defined anymore. As an example consider interpolation by a
shortest path on the sphere~$S^2$ from the north pole to a point $q \in S^2$.
This corresponds to geodesic interpolation on a one-dimensional simplex with two
Lagrange points~\cite[Lemma~2.2]{sander:2012}.
The interpolation from the north pole by shortest path is uniquely defined for all $q$ not equal to
the south pole, and if $(q_k)$ is a sequence of points converging to the south pole,
interpolation yields a sequence of unique shortest paths that converge pointwise
to a shortest path from the north pole to the south pole (see Figure~\ref{fig:limit_nonuniqueness}).
That limit path is not the only shortest path between the two poles.
However, it still satisfies the defining minimization
problem~\eqref{eq:geodesic_interpolation}, and is therefore a geodesic interpolation
function in the sense of the definition.

\begin{lemma}
\label{lem:pointwise_limits}
Let $\mathcal{M}$ be a complete metric space with a metric $d(\cdot,\cdot)$, and let
$ (R_k)$ be a sequence of geodesic finite element functions with $R_k : \omega \to \mathcal{M}$
for all $k \in \mathbb{N}$,
converging pointwise to some limit function $R : \omega \to \mathcal{M}$. Then
the limit function is also a geodesic finite element function.
\end{lemma}

\begin{figure}
 \begin{center}
		\begin{tikzpicture}[scale=1.6] \small
		\coordinate (centerSO3) at (5.68,3.1);
		\shade [ball color=white] (centerSO3) circle [radius=1.5cm];
		\draw[thick] (centerSO3) circle (1.5cm);

		\coordinate (redR1) at (6.05,2.55);
		\node at (5.9,2.3) {\color{red} $R_k(a_1)$};
		\coordinate (redR12) at (6.25,3.15);
		\coordinate (redR2) at (6.2,3.85);
		\node at (6.2,4.05) {\color{red} $R_k(a_2)$};
		\coordinate (redR23) at (5.5,3.7);
		\coordinate (redR13) at (5.4,2.8);
		\coordinate (redR3) at (4.9,3.3);
		\node at (4.65,3.5) {\color{red} $R_k(a_3)$};
		\coordinate (redCenter) at (5.7,3.2);
		\draw[draw=red, fill=red, thick] (redR2) circle (0.04cm);
		\draw[red] plot [smooth] coordinates {(redR1)(redR12)(redR2)(redR23)(redR3)(redR13)(redR1)};
		\draw[fill=red, draw=red] (redR1) circle (0.04cm);
		\draw[fill=red, draw=red] (redR2) circle (0.04cm);
		\draw[fill=red, draw=red] (redR3) circle (0.04cm);
		\coordinate (redR1c) at (5.85,2.8);
		\coordinate (redR2c) at (6.0,3.5);
		\coordinate (redR3c) at (5.3,3.35);
		\draw [red, dashed] plot [smooth] coordinates {(redR1)(redR1c)(redCenter)};
		\draw [red, dashed] plot [smooth] coordinates {(redR2)(redR2c)(redCenter)};
		\draw [red, dashed] plot [smooth] coordinates {(redR3)(redR3c)(redCenter)};
		\draw[draw=red, fill=red] (redCenter) circle (0.05cm);
		
		\coordinate (R1) at (6.37,2.3);
		\node at (6.45,2.1) { $R(a_1)$};
		\coordinate (R12) at (6.55,2.95);
		\coordinate (R2) at (6.5,3.65);
		\node at (6.6,3.8) { $R(a_2)$};
		\coordinate (R23) at (5.8,3.5);
		\coordinate (R13) at (5.7,2.6);
		\coordinate (R3) at (5.2,3.1);
		\node at (4.9,2.95) { $R(a_3)$};
		\coordinate (center) at (6,3);
		\draw[fill=black, thick] (R2) circle (0.04cm);

		\draw[black] plot [smooth] coordinates {(R1)(R12)(R2)(R23)(R3)(R13)(R1)};
		\draw[fill=black] (R1) circle (0.04cm);
		\draw[fill=black] (R2) circle (0.04cm);
		\draw[fill=black] (R3) circle (0.04cm);
		\coordinate (R1c) at (6.15,2.6);
		\coordinate (R2c) at (6.3,3.3);
		\coordinate (R3c) at (5.6,3.15);
		\draw [dashed] plot [smooth] coordinates {(R1)(R1c)(center)};
		\draw [dashed] plot [smooth] coordinates {(R2)(R2c)(center)};
		\draw [dashed] plot [smooth] coordinates {(R3)(R3c)(center)};
		\draw[fill=black] (center) circle (0.05cm);
		
		\coordinate (R3b) at (5.15,3.14);
		\coordinate (redR3b) at (4.95,3.27);
		\draw[-stealth, blue] (redR3b) -- (R3b);
		
		\coordinate (R2b) at (6.45,3.7);
		\coordinate (redR2b) at (6.27,3.81);
		\draw[-stealth, blue] (redR2b) -- (R2b);
		
		\coordinate (R1b) at (6.28,2.36);
		\coordinate (redR1b) at (6.1,2.5);
		\draw[-stealth, blue] (redR1b) -- (R1b);
		
		\node at (5,2.5) {\color{blue} $k \to \infty$};
		\node at (5.2,3.9) {\color{red} $R_k(\eta)$};
		\draw[-stealth, red] (5.2,3.8) -- (5.6,3.3);
		\node at (6.9,3) {$R(\eta)$};
		\draw[-stealth, black] (6.65,3) -- (6.1,3);
	\end{tikzpicture}
 \end{center}
 \caption{The situation in the proof of Lemma~\ref{lem:pointwise_limits}
  for a single triangle, and with first-order interpolation}
\end{figure}
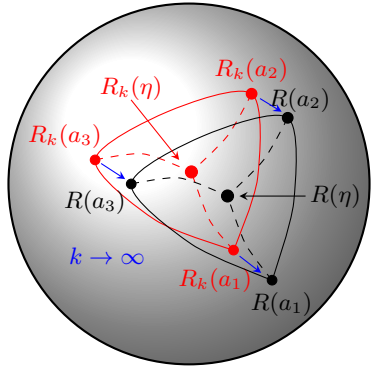

\begin{proof}
 Let $a_1,\dots, a_N \in \omega$ be the Lagrange points of Definition~\ref{def:lagrange_points},
 and let $\lambda_1,\dots,\lambda_N$ be the corresponding scalar-valued Lagrange basis.
 Recalling the global definition in Remark~\ref{rem:global_space_definitions}
 of a geodesic finite element space,
 we need to show that for any $\eta \in \omega$ we have
 \begin{equation}
 \label{eq:limit_is_minimizer}
  \sum_{i=1}^N \lambda_i(\eta)\, d(R(\eta), R(a_i))^2
  \le
  \sum_{i=1}^N \lambda_i(\eta)\, d(\hat{R}, R(a_i))^2
 \end{equation}
 for any $\hat{R} \in \mathcal{M}$.

 For simplicity we omit the $\eta$-dependence of the $\lambda_i$. Suppose the inequality~\eqref{eq:limit_is_minimizer}
 is false. Then there is an $R^* \in \mathcal{M}$ with
 \begin{equation*}
  \sum_{i=1}^N \lambda_i \,d(R^*, R(a_i))^2
  <
  \sum_{i=1}^N \lambda_i \,d(R(\eta), R(a_i))^2.
 \end{equation*}
 Let $\epsilon$ be a particular fraction of the difference:
 \begin{equation*}
  \epsilon
  \colonequals
  \frac{1}{3N} \bigg[ \sum_{i=1}^N \lambda_i \,d(R(\eta), R(a_i))^2 - \sum_{i=1}^N \lambda_i \,d(R^*,R(a_i))^2\bigg].
 \end{equation*}
 We then look at a single addend. If $\lambda_i \ge 0$ we have
 \begin{align*}
  \lambda_i \,d(R^*, R_k(a_i))^2
  & \le
  \lambda_i \Big[ \,d(R^*,R(a_i)) + \,d(R(a_i), R_k(a_i))\Big]^2 \\
  & =
  \lambda_i \Big[ \,d(R^*,R(a_i))^2
              + 2 \,d(R^*,R(a_i)) \,d(R(a_i),R_k(a_i)) \\
  & \qquad \qquad \qquad \qquad \qquad \qquad \qquad \qquad          + \,d(R(a_i), R_k(a_i))^2 \Big].
 \end{align*}
 As $R_k(a_i) \to R(a_i)$ we can pick $k$ large enough such that for all $i=1,\dots,N$ we have
 \begin{equation*}
  2 \lambda_i \,d(R^*,R(a_i)) \,d(R(a_i),R_k(a_i)) + \lambda_i \,d(R(a_i), R_k(a_i))^2
  <
  \epsilon.
 \end{equation*}
 If, on the other hand, $\lambda_i < 0$, we bound
 \begin{align*}
  \lambda_i \,d(R^*, R_k(a_i))^2
  & \le
  \lambda_i \Big[ d(R^*,R(a_i)) - d(R(a_i), R_k(a_i))\Big]^2 \\
  & =
  \lambda_i \Big[ d(R^*,R(a_i))^2
              - 2 d(R^*,R(a_i)) \,d(R(a_i),R_k(a_i)) \\
  & \qquad \qquad \qquad \qquad \qquad \qquad \qquad \qquad             + d(R(a_i), R_k(a_i))^2 \Big].
 \end{align*}
 Again, for all $k$ large enough we can assume that
 \begin{equation*}
  -2 \lambda_i \, d(R^*,R(a_i)) \,d(R(a_i),R_k(a_i)) + \lambda_i \,d(R(a_i), R_k(a_i))^2
  <
  \epsilon.
 \end{equation*}
 Summing up we get
 \begin{align*}
  \sum_{i=1}^N \lambda_i \,d(R^*, R_k(a_i))^2
  & \le
  \sum_{i=1}^N \lambda_i \,d(R^*, R(a_i))^2 + N\epsilon \\
  & <
  \sum_{i=1}^N \lambda_i \,d(R(\eta), R(a_i))^2 - N\epsilon.
 \end{align*}
 By continuity of the distance function, we can choose $k$ large enough such that
 \begin{equation*}
  \sum_{i=1}^N \lambda_i \,d(R_k(\eta), R_k(a_i))^2
  >
  \sum_{i=1}^N \lambda_i \,d(R(\eta), R(a_i))^2 - \epsilon.
 \end{equation*}
 Then
 \begin{equation*}
  \sum_{i=1}^N \lambda_i \,d(R^*, R_k(a_i))^2
  <
  \sum_{i=1}^N \lambda_i \,d(R_k(\eta), R_k(a_i))^2.
 \end{equation*}
 This is a contradiction to the definition of $R_k$: As it is a geodesic finite element
 function, $R_k(\eta)$ must be a minimizer of $\sum_{i=1}^N \lambda_i\,d(\cdot,R_k(a_i))^2$.
\end{proof}

\subsection{The algebraic problem}

For the numerical minimization of the Cosserat shell energy we introduce an algebraic formulation.
For standard finite elements there is a bijective correspondence between finite element functions and
coefficient vectors, via the representations of the functions with respect to a basis.
For geometric finite elements this bijection holds only locally.
The details are explained,
for example, in~\cite{sander:2013,sander_neff_birsan:2016}.

Let $a_1, \dots,a_{N_1}$ be the global set of Lagrange nodes
for the $p_1$-th order Lagrange finite element space on $\omega$ with respect to the triangulation $\mathcal{T}$.
We introduce the evaluation operator
\begin{equation*}
\mathcal{E}^{\R^3}: S_{h} (\omega, \R^3) \to (\R^3)^{N_1},
\qquad
\mathcal{E}^{\R^3}(\bm_h)_i \colonequals \bm_h(a_i), \quad i = 1, \dots, N_1.
\end{equation*}
Similarly, for the global set of nodes $a_1, \dots,a_{N_2}$ of a Lagrange space
of order $p_2$ we define the evaluation operator
\begin{equation*}
 \mathcal{E}^{\SOdrei} : S_h (\omega, \SOdrei) \to \SOdrei^{N_2},
 \qquad
 \mathcal{E}^{\SOdrei}(\Qeh)_i \colonequals \Qeh(a_i),
 \qquad
 i= 1,\dots,N_2.
\end{equation*}
To construct an algebraic formulation of the shell problem we need the inverse operators
$\big(\mathcal{E}^{\R^3}\big)^{-1}$ and $\big(\mathcal{E}^{\SOdrei}\big)^{-1}$, which associate
functions to sets of values at the Lagrange points.
The operator $\mathcal{E}^{\R^3}$ is a bijection and hence
$(\mathcal{E}^{\R^3})^{-1}$ is uniquely defined on all of $(\R^3)^{N_1}$.
In contrast, $\mathcal{E}^{\SOdrei}$
is not defined for all sets of coefficients
from $\SOdrei^{N_2}$, because the interpolation rules of Section~\ref{sec:geometric_interpolation_rules} used to construct
finite element functions from coefficients may fail to produce a
continuous finite element function, or they may produce more than one
from a single set of coefficients.
Nevertheless, as shown in~\cite{sander:2012} the evaluation operator
$\mathcal{E}^{\SOdrei}$ is invertible locally under reasonable circumstances, and one can use it
to define the algebraic Cosserat shell energy
\begin{equation}
\label{eq:algebraic_shell_energy}
 I_\textup{alg} \; : D_\textup{alg} \to \R,
 \qquad
 I_\textup{alg}(\overline{\bm},\overline{\boldsymbol{Q}}_e)
 \colonequals
 I\big((\mathcal{E}^{\R^3})^{-1}(\overline{\bm}),\: (\mathcal{E}^{\SOdrei})^{-1}(\overline{\bQ}_e)\big),
\end{equation}
where $D_\textup{alg}$ is a subset of $(\R^3)^{N_1} \times \SOdrei^{N_2}$, and $I$ is the functional~\eqref{eq:finite_strain_energy}.
The algebraic Cosserat shell problem then is:
\begin{problem}[Algebraic Cosserat shell problem]\label{prob:algebraic-cosserat-shell-problem}
Find a pair $(\overline{\bm}$, $\overline{\bQ}_e) \in D_\textup{alg}$
that minimizes $I_\textup{alg}$, subject to suitable Dirichlet conditions.
\end{problem}

Problem~\ref{prob:algebraic-cosserat-shell-problem} is a smooth optimization problem
in the set $D_\text{alg}$, which is an open subset of the manifold $\R^{3N_1} \times \SOdrei^{N_2}$.
Such problems can be solved conveniently and efficiently using algorithms for solving
optimization problems on manifolds~\cite{absil_mahony_sepulchre:2008}.
Details are given in~\cite{sander:2012,sander_neff_birsan:2016}.

The algebraic shell minimization problem naturally inherits the frame-indifference property
\begin{equation*}
 I(R \,\bm, R\,\Qe) = I(\bm,\Qe),
\end{equation*}
(where $R$ is any element of $\SOdrei$, acting on functions in $H^1(\omega,\R^3)$
and $H^1(\omega,\SOdrei)$ by pointwise multiplication) of the continuous and discrete
formulations.
\begin{theorem}
 When omitting the external loads,
 the algebraic shell energy functional~$I_\textup{alg}$ is frame-indifferent in the sense that
\begin{equation*}
  I_\textup{alg}(R\, \overline{\bm}, R \, \overline{\bQ}_e)
  =
  I_\textup{alg}(\overline{\bm},\overline{\bQ}_e),
\end{equation*}
for all $R \in \SOdrei$, which, by an abuse of notation, now act on the components
of $\overline{\bm} \in (\R^3)^{N_1}$ and $\overline{\bQ}_e \in \SOdrei^{N_2}$.
\end{theorem}
This important property sets geometric finite element discretizations apart from alternative approaches like~\cite{muench:2007, mueller:2009},
which do not preserve the frame indifference of continuous models.

\section{Numerical experiments}
\label{sec:numerical_experiments}
In this final chapter we demonstrate the capabilities of the shell model and the discretization with a
set of numerical experiments. The software for the experiments was implemented
in C++, based on the \textsc{Dune} libraries%
\footnote{\url{https://dune-project.org}}~\cite{bastian_et_al:2021,sander_dune:2020}.
The code relies heavily on the ADOL-C algorithmic differentiation library \cite{walther_griewank:2012} for the computation of first and second derivatives of the energy functional $I_\text{alg}$, and on the \texttt{dune-curvedgrid} extension module
to the \textsc{Dune} software%
\footnote{\url{https://dune-project.org/modules/dune-curvedgrid}}
\cite{praetorius_stenger:2022}, which provides finite element grids with nonplanar element geometries.
For two-dimensional grids in~$\R^3$ it also provides direct access
to the element normal vector $\bn_0$ and to its derivative.

\subsection{Deflection of a half sphere}
\label{sec:locking_test}
In our first numerical example we want to demonstrate that the discretized model does not suffer
from shear locking. For the planar Cosserat shell model in \cite{sander_neff_birsan:2016}
we showed experimentally that no shear locking occurs unless the deformation field $\bm$ is
discretized with Lagrange finite elements of first order. 
A similar result is obtained here, too, but now the approximation order of the reference shell surface
$\bm_0(\omega)$ comes into play as well.

As the problem setting we consider the upper
half of a sphere of radius $1\,\mathsf{L}$ as the reference configuration $\bm_0(\omega)$
(the $\mathsf{L}$ means length),
such that Gauß and mean curvature are $K = 1\,\mathsf{L}^{-2}$ and $H = 1\,\mathsf{L}^{-1}$
everywhere, respectively.
We clamp the deformation at the equator, but leave the microrotation free of any Dirichlet conditions.
We load the shell
with a constant vertical tensile volume load of $\mathbf{f} = h \cdot (0,0,10^4)\,\mathsf{M}/(\mathsf{L}\mathsf{T}^{2})$,
where $\mathsf{M}$ and $\mathsf{T}$ mean mass and time, respectively.
Note that we scale with the thickness~$h$ to represent a load obtained
by integrating a three-dimensional volume load over the shell thickness~$h$.

The material parameters are $\lambda = 4.4364\cdot10^4\,\mathsf{M}/(\mathsf{L}\mathsf{T}^{2})$
and $\mu = 2.7191\cdot10^4\,\mathsf{M}/(\mathsf{L}\mathsf{T}^{2})$ for the Lamé parameters,
$\mu_c = 0.1 \mu$ for the Cosserat couple modulus, $L_c = 5\cdot10^{-4}\,\mathsf{L}$
for the internal length, and $b_1 = b_2 = 1$, $b_3 = \tfrac{1}{3}$.
With this choice of $b_1$, $b_2$, and $b_3$, the curvature energy density defined in~\eqref{eq:Wcurv} becomes
simply $W_\text{curv}(X) = \mu L_c^2 \norm{X}^2$.

We discretize the deformation functions $\bm_0$ and $\bm$ by Lagrange finite elements, and the microrotation function by geometric finite elements based on geodesic interpolation.
For the discrete shell reference geometry $\bm_{0,h}$ we use elements of first and second order.
This means that the element geometries of the finite element grid are either flat triangles or second-order polynomials, constructed such that the triangle corners (and the edge midpoints in the second-order case) lie on the half sphere.
Recall that when using flat triangles, all curvature measures in the
energy~\eqref{eq:finite_strain_energy} vanish.

We run simulations for six combinations of approximation orders for the reference deformation $\bm_0$,
the deformation $\bm$, and the microrotation field $\Qe$:
\bigskip
\begin{center}
 \begin{tabular}{l|llllll}
  scenario & 1 & 2 & 3 & 4 & 5 & 6 \\
  \hline
  order: reference deformation $\bm_{0,h}$ & \color{lightgreen}1 & \color{mediumgreen}1 & \color{darkgreen}1 & \color{gray}2 & \color{lightblue}2 & \color{mediumblue}2 \\
  order: deformation $\bm_h$               & \color{lightgreen}1 & \color{mediumgreen}2 & \color{darkgreen}2 & \color{gray}1 & \color{lightblue}2 & \color{mediumblue}2 \\
  order:  microrotation $\Qeh$              & \color{lightgreen}1 & \color{mediumgreen}1 & \color{darkgreen}2 & \color{gray}1 & \color{lightblue}1 & \color{mediumblue}2.
 \end{tabular}
\end{center}
\medskip
Of the eight possible ways to combine first- and second-order approximations we omit only the cases
with first-order deformation and second-order microrotation. The reason is that in the continuous model
the third column of the microrotation can be interpreted as an approximation of the normal of $\bm(\omega)$, which is a
first-order derivative of $\bm$. Hence, when approximating the deformation by piecewise polynomials
of a given order, the exact normal is a piecewise polynomial of one order lower.
Consequently, the most natural choice of approximation order is for the deformation to have
one order more than the microrotation. In addition, we test the case of equal orders.

\begin{figure}
	\begin{tabular}{c c c}
		$24$ elements & $24 \cdot 4^2$ elements & $24 \cdot 4^4$ elements \\
		\hline
		\hline
		\\
		\includegraphics[width=0.3\textwidth]{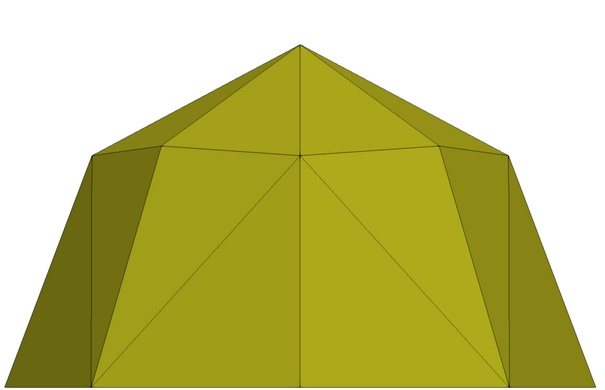}\quad &\includegraphics[width=0.3\textwidth]{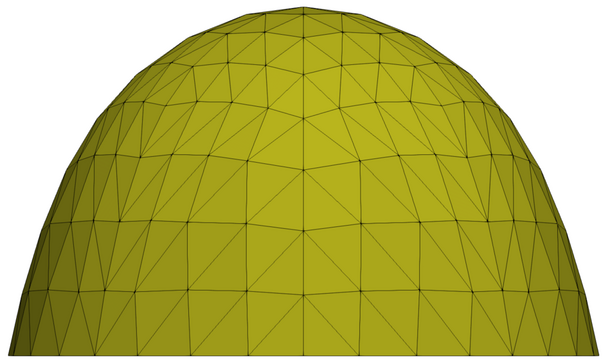}\quad &\includegraphics[width=0.3\textwidth]{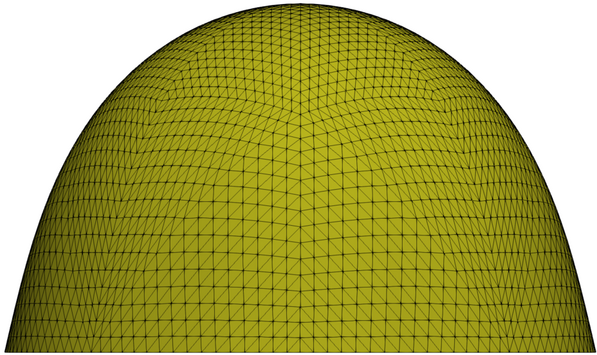}\\
		\multicolumn{3}{c}{\color{lightgreen} reference deformation order: 1, deformation order: 1, microrotation order: 1} \\
		\hline
		\\
		\includegraphics[width=0.3\textwidth]{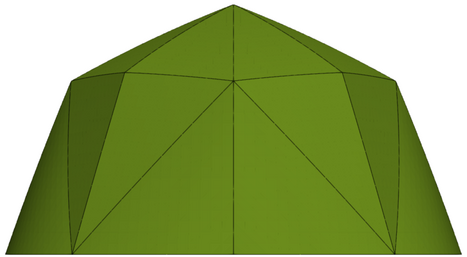}\quad &\includegraphics[width=0.3\textwidth]{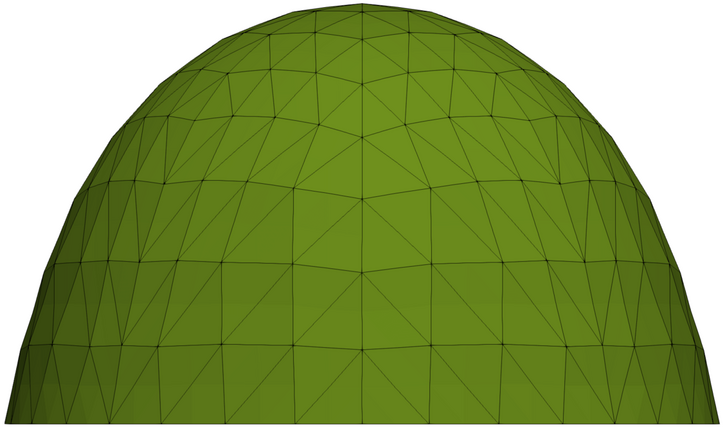}\quad &\includegraphics[width=0.3\textwidth]{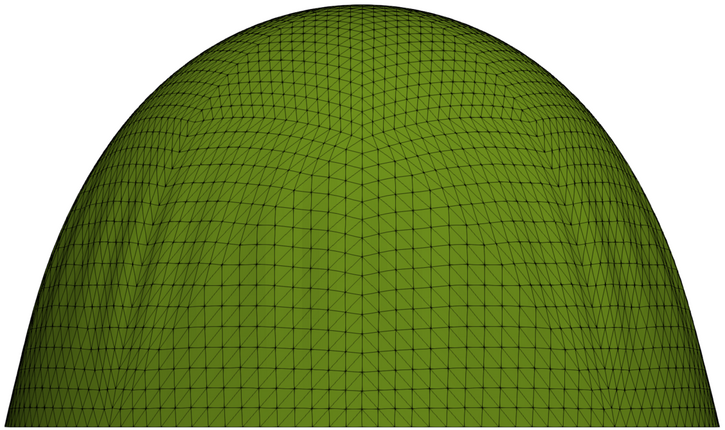}\\
		\multicolumn{3}{c}{\color{mediumgreen} reference deformation order: 1, deformation order: 2, microrotation order: 1} \\
		\hline
		\\ \includegraphics[width=0.3\textwidth]{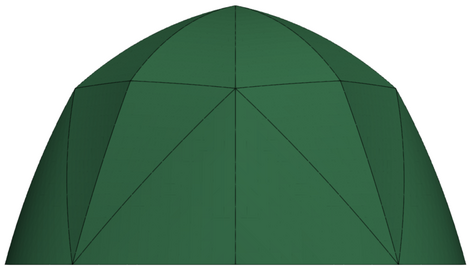}\quad &\includegraphics[width=0.3\textwidth]{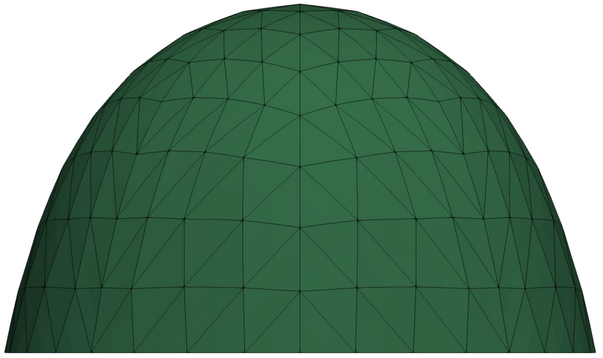}\quad &\includegraphics[width=0.3\textwidth]{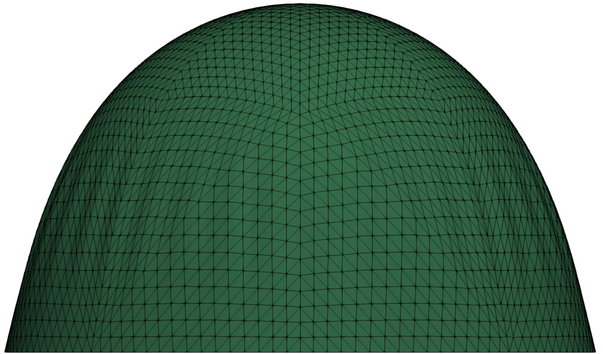}\\
		\multicolumn{3}{c}{\color{darkgreen} reference deformation order: 1, deformation order: 2, microrotation order: 2} \\
		\hline
		\\
		\includegraphics[width=0.3\textwidth]{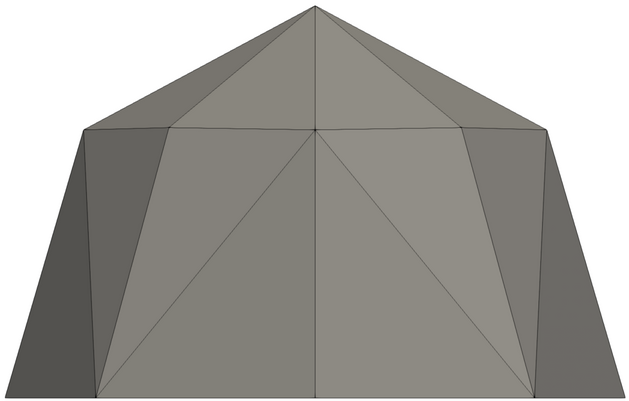}\quad & \includegraphics[width=0.3\textwidth]{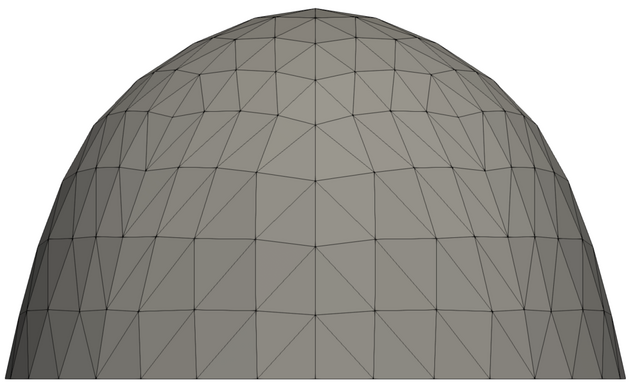}\quad& \includegraphics[width=0.3\textwidth]{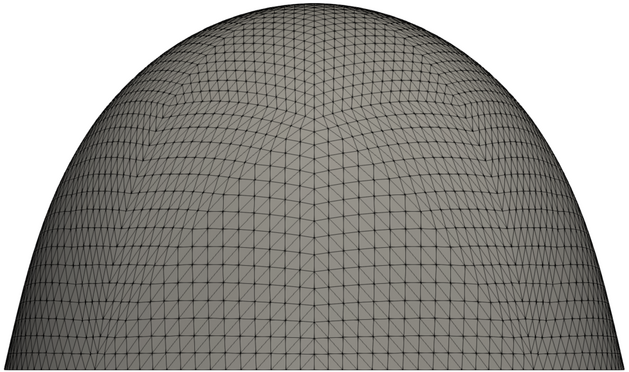}\\
		\multicolumn{3}{c}{\color{gray} reference deformation order: 2, deformation order: 1, microrotation order: 1} \\
		\hline
		\\
		\includegraphics[width=0.3\textwidth]{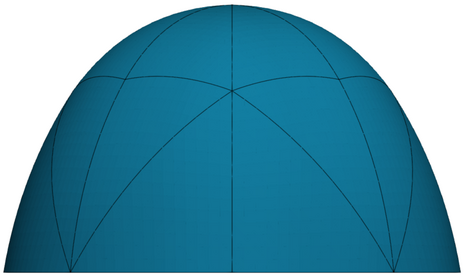}\quad & \includegraphics[width=0.3\textwidth]{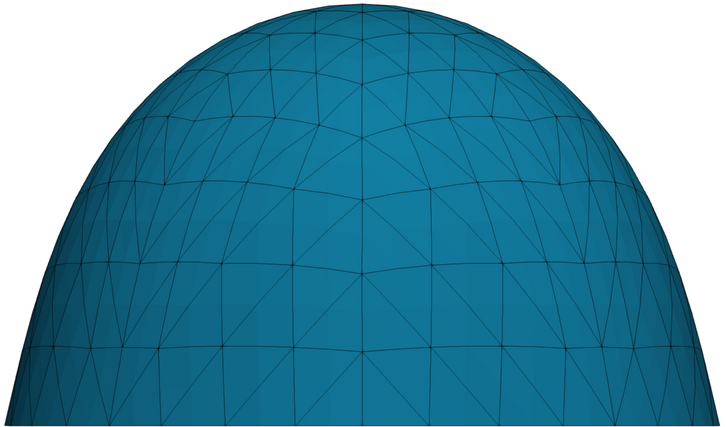}\quad& \includegraphics[width=0.3\textwidth]{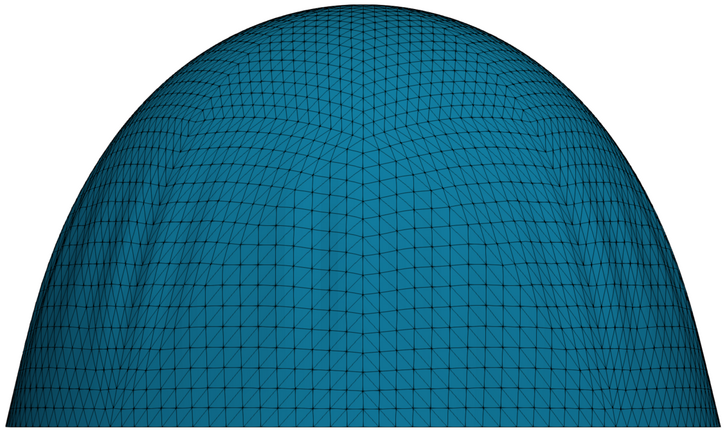}\\
		\multicolumn{3}{c}{\color{lightblue} reference deformation order: 2, deformation order: 2, microrotation order: 1} \\
		\hline
		\\
		\includegraphics[width=0.3\textwidth]{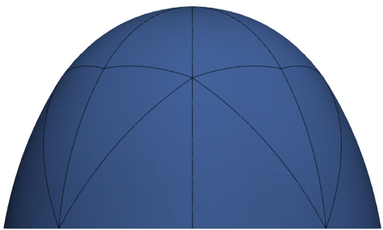}\quad & \includegraphics[width=0.3\textwidth]{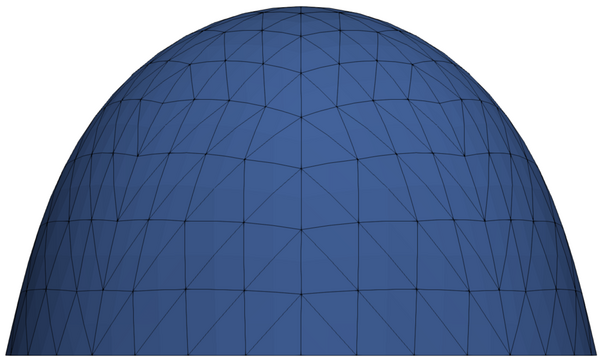}\quad& \includegraphics[width=0.3\textwidth]{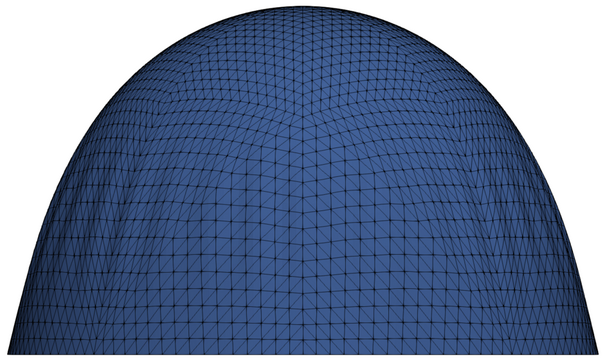} \\
		\multicolumn{3}{c}{\color{mediumblue} reference deformation order: 2, deformation order: 2, microrotation order: 2} \\
	\end{tabular}
	\caption{Deflection of a half sphere under vertical tensile load, modelled by a shell
	with thickness $h = 10^{-3}\,\mathsf{L}$ using different choices of approximation orders
	for the geometry $\bm_{0,h}$, for the deformation $\bm_h$ and for the microrotation $\Qeh$}
	\label{fig:deflection}
\end{figure}

Figure~\ref{fig:deflection}
shows qualitative results for a shell thickness of $h = 10^{-3}\,\mathsf{L}$ and three different grid resolutions,
namely grids with $24$ elements, $24 \cdot 4^2$ elements, and $24 \cdot 4^4$ elements.
All six combinations of finite element orders seem to agree on the vertical deflection
for the medium and high grid resolutions. With the highly resolved grid, the combinations where
the deformation~$\bm$ is approximated with finite element functions of second order
show small vertical wrinkles, which can be expected in this load scenario.
The simulations with first-order deformations do not show wrinkles.

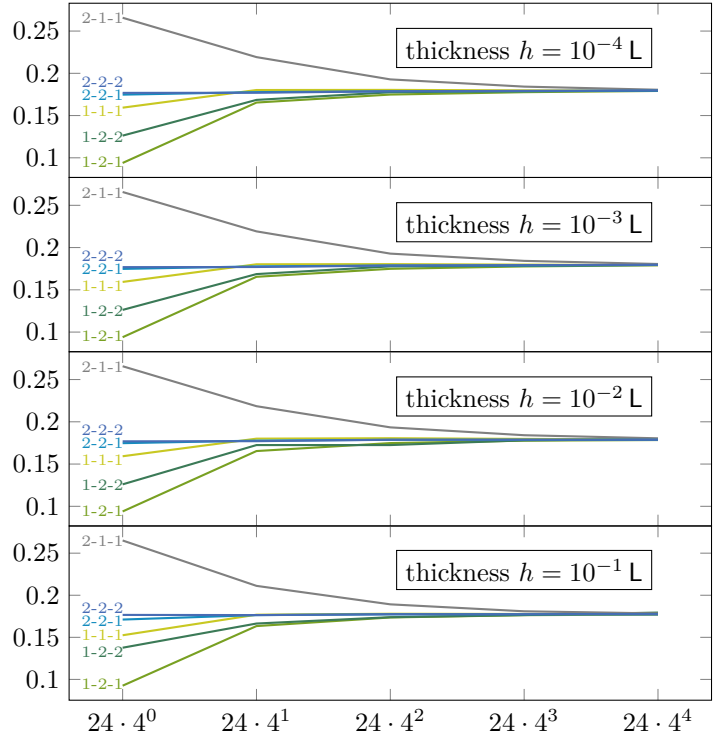
\begin{figure}
	\begin{tikzpicture}
		\begin{groupplot}[
			group style={
				group size=1 by 4,
				x descriptions at=edge bottom,
				vertical sep=0pt,
			},
			xtick = data,
			xticklabels={$24\cdot 4^0$, $24\cdot 4^1$, $24\cdot 4^2$, $24\cdot 4^3$, $24\cdot 4^4$},
			xmode=normal]

			\nextgroupplot[ymode=normal, height=0.19\textheight,width=0.8\textwidth]
			\node[draw=black,fill=white] at (axis cs:4,0.23) {{thickness $h=10^{-4}\,\mathsf{L}$}};
			\node at (axis cs:0.85,0.265){\tiny\color{gray}2-1-1};
			\node at (axis cs:0.85,0.19){\tiny\color{mediumblue}2-2-2};
			\node at (axis cs:0.85,0.175){\tiny\color{lightblue}2-2-1};
			\node at (axis cs:0.85,0.155){\tiny\color{lightgreen}1-1-1};
			\node at (axis cs:0.85,0.125){\tiny\color{darkgreen}1-2-2};
			\node at (axis cs:0.85,0.095){\tiny\color{mediumgreen}1-2-1};
			\addplot[color=lightgreen,thick] table[col sep=comma,header=true, x index=0,y index=1] {locking-muc-0.1mu-thickness-0.0001.csv};
			\addplot[color=mediumgreen,thick] table[col sep=comma,header=true, x index=0,y index=2] {locking-muc-0.1mu-thickness-0.0001.csv};
			\addplot[color=darkgreen,thick] table[col sep=comma,header=true, x index=0,y index=3] {locking-muc-0.1mu-thickness-0.0001.csv};
			\addplot[color=gray,thick] table[col sep=comma,header=true, x index=0,y index=4] {locking-muc-0.1mu-thickness-0.0001.csv};
			\addplot[color=lightblue,thick] table[col sep=comma,header=true, x index=0,y index=5] {locking-muc-0.1mu-thickness-0.0001.csv};
			\addplot[color=mediumblue,thick] table[col sep=comma,header=true, x index=0,y index=6] {locking-muc-0.1mu-thickness-0.0001.csv};
			
			\nextgroupplot[ymode=normal,height=0.19\textheight,width=0.8\textwidth]
			\node[draw=black,fill=white] at (axis cs:4,0.23) {{thickness $h=10^{-3}\,\mathsf{L}$}};
			\node at (axis cs:0.85,0.265){\tiny\color{gray}2-1-1};
			\node at (axis cs:0.85,0.19){\tiny\color{mediumblue}2-2-2};
			\node at (axis cs:0.85,0.175){\tiny\color{lightblue}2-2-1};
			\node at (axis cs:0.85,0.155){\tiny\color{lightgreen}1-1-1};
			\node at (axis cs:0.85,0.125){\tiny\color{darkgreen}1-2-2};
			\node at (axis cs:0.85,0.095){\tiny\color{mediumgreen}1-2-1};
			\addplot[color=lightgreen,thick] table[col sep=comma,header=true, x index=0,y index=1] {locking-muc-0.1mu-thickness-0.001.csv};
			\addplot[color=mediumgreen,thick] table[col sep=comma,header=true, x index=0,y index=2] {locking-muc-0.1mu-thickness-0.001.csv};
			\addplot[color=darkgreen,thick] table[col sep=comma,header=true, x index=0,y index=3] {locking-muc-0.1mu-thickness-0.001.csv};
			\addplot[color=gray,thick] table[col sep=comma,header=true, x index=0,y index=4] {locking-muc-0.1mu-thickness-0.001.csv};
			\addplot[color=lightblue,thick] table[col sep=comma,header=true, x index=0,y index=5] {locking-muc-0.1mu-thickness-0.001.csv};
			\addplot[color=mediumblue,thick] table[col sep=comma,header=true, x index=0,y index=6] {locking-muc-0.1mu-thickness-0.001.csv};
			
			\nextgroupplot[ymode=normal,height=0.19\textheight,width=0.8\textwidth]
			\node[draw=black,fill=white] at (axis cs:4,0.23) {{thickness $h=10^{-2}\,\mathsf{L}$}};
			\node at (axis cs:0.85,0.265){\tiny\color{gray}2-1-1};
			\node at (axis cs:0.85,0.19){\tiny\color{mediumblue}2-2-2};
			\node at (axis cs:0.85,0.175){\tiny\color{lightblue}2-2-1};
			\node at (axis cs:0.85,0.155){\tiny\color{lightgreen}1-1-1};
			\node at (axis cs:0.85,0.125){\tiny\color{darkgreen}1-2-2};
			\node at (axis cs:0.85,0.095){\tiny\color{mediumgreen}1-2-1};
			\addplot[color=lightgreen,thick] table[col sep=comma,header=true, x index=0,y index=1] {locking-muc-0.1mu-thickness-0.01.csv};
			\addplot[color=mediumgreen,thick] table[col sep=comma,header=true, x index=0,y index=2] {locking-muc-0.1mu-thickness-0.01.csv};
			\addplot[color=darkgreen,thick] table[col sep=comma,header=true, x index=0,y index=3] {locking-muc-0.1mu-thickness-0.01.csv};
			\addplot[color=gray,thick] table[col sep=comma,header=true, x index=0,y index=4] {locking-muc-0.1mu-thickness-0.01.csv};
			\addplot[color=lightblue,thick] table[col sep=comma,header=true, x index=0,y index=5] {locking-muc-0.1mu-thickness-0.01.csv};
			\addplot[color=mediumblue,thick] table[col sep=comma,header=true, x index=0,y index=6] {locking-muc-0.1mu-thickness-0.01.csv};
			
			\nextgroupplot[ymode=normal,height=0.19\textheight,width=0.8\textwidth]
			\node[draw=black,fill=white] at (axis cs:4,0.23) {{thickness $h=10^{-1}\,\mathsf{L}$}};
			\node at (axis cs:0.85,0.265){\tiny\color{gray}2-1-1};
			\node at (axis cs:0.85,0.185){\tiny\color{mediumblue}2-2-2};
			\node at (axis cs:0.85,0.17){\tiny\color{lightblue}2-2-1};
			\node at (axis cs:0.85,0.153){\tiny\color{lightgreen}1-1-1};
			\node at (axis cs:0.85,0.133){\tiny\color{darkgreen}1-2-2};
			\node at (axis cs:0.85,0.095){\tiny\color{mediumgreen}1-2-1};
			\addplot[color=lightgreen,thick] table[col sep=comma,header=true, x index=0,y index=1] {locking-muc-0.1mu-thickness-0.1.csv};
			\addplot[color=mediumgreen,thick] table[col sep=comma,header=true, x index=0,y index=2] {locking-muc-0.1mu-thickness-0.1.csv};
			\addplot[color=darkgreen,thick] table[col sep=comma,header=true, x index=0,y index=3] {locking-muc-0.1mu-thickness-0.1.csv};
			\addplot[color=gray,thick] table[col sep=comma,header=true, x index=0,y index=4] {locking-muc-0.1mu-thickness-0.1.csv};
			\addplot[color=lightblue,thick] table[col sep=comma,header=true, x index=0,y index=5] {locking-muc-0.1mu-thickness-0.1.csv};
			\addplot[color=mediumblue,thick] table[col sep=comma,header=true, x index=0,y index=6] {locking-muc-0.1mu-thickness-0.1.csv};
		\end{groupplot}
	\end{tikzpicture}
	
	\caption{Vertical displacement of the north pole as a function of the number of grid elements,
	for four different thicknesses $h$ and different approximation orders for the element geometry
	and the finite element functions.
	The labels of the graphs show: reference deformation order -- deformation order -- microrotation order.}
	\label{fig:deflection_result}
\end{figure}

When looking at the coarse grid results, one can see that the vertical deflection sometimes deviates
from the one of the finer grids. To investigate this phenomenon more closely, we measure the deflection
of the north pole of the half sphere resulting from the applied load.
Figure~\ref{fig:deflection_result} shows the deflection as a function of the grid resolution.
In addition to $h = 10^{-3}\,\mathsf{L}$, the figure shows the same measurements for the
thickness values $h = 10^{-4}\,\mathsf{L}$, $h = 10^{-2}\,\mathsf{L}$, and $h = 10^{-1}\,\mathsf{L}$.
The behavior is virtually identical for all four thicknesses: The vertical deflection
is correctly represented even on the coarser grids if the reference configuration~$\bm_0$ and
the deformation $\bm$ are both discretized with finite elements of second order.
Almost all other combinations are too stiff if the grid is too coarse, which is the
classical locking phenomenon to be expected for badly chosen finite elements.

The one exception is the case with a second-order approximation of the reference geometry
together with a first-order approximation of the deformation~$\bm$, shown in {\color{gray}gray}. Here, the vertical
deflection is too large, and convergence to the correct value for increasing grid resolution
is slower than for the locking cases. Note that in this particular situation, the discrete deformation
cannot assume the stress-free configuration---a residual pre-stress remains. The exact
mechanism of how this pre-stress leads to excessive softness is unclear.

\subsection{Comparison with alternative models}
\label{sec:comparison_with_3d_model}
Next, we compare the shell model with the related model of \textcite{birsan:2021} briefly described in Remark~\ref{rem:birsan_energy}.  For the comparison, we use the model problem from the previous section
with the same material parameters, i.e., $\lambda = 4.4364\cdot10^4\,\mathsf{M}/(\mathsf{L}\mathsf{T}^{2})$
and $\mu = 2.7191\cdot10^4\,\mathsf{M}/(\mathsf{L}\mathsf{T}^{2})$ for the Lamé parameters,
$\mu_c = 0.1 \mu$ for the Cosserat couple modulus, $L_c = 5\cdot10^{-4}\,\mathsf{L}$ and $b_1 = b_2 = 1$, $b_3 = \tfrac{1}{3}$.
This time we test the model with several different volume loads in the range between
$\mathbf{f} = 0.2 \cdot h \cdot (0,0,10^4)\,\mathsf{M}/(\mathsf{L}\mathsf{T}^{2})$
and $\mathbf{f} = 1.5 \cdot h \cdot (0,0,10^4)\,\mathsf{M}/(\mathsf{L}\mathsf{T}^{2})$.
We clamp the deformation~$\bm_h$ at the equator;
the microrotation function is not subject to Dirichlet conditions at all.
We compare the two models for six thickness values between $h = 10^{-4}\,\mathsf{L}$ and $h = 0.4\,\mathsf{L}$.
With the principal curvatures $\kappa_1 = \kappa_2 = 1$ for this example,
these thickness values are all below the curvature bounds $h \abs{\kappa_1}< \frac{1}{2}$
and $h \abs{\kappa_2} < \frac{1}{2}$ of Theorem~\ref{thm:neff_existence}.

The reference deformation $\bm_{0,h}$ and the deformation $\bm_h$ are represented
by Lagrange finite elements of second order and the microrotations are discretized
by geodesic finite elements of first order.  As shown in the previous section this choice
avoids shear locking.
We discretize the shell surface $\omega$ using a grid consisting of $24 \cdot 4^5$ second-order triangle
elements.

\begin{figure}
	 \begin{center}
	 	\begin{tabular}{c c c}
	 		$h= 10^{-4}\,\mathsf{L}$& $h=10^{-3}\,\mathsf{L}$ & $h=10^{-2}\,\mathsf{L}$
	 		\\\hline\hline \\
	 		\includegraphics[width=0.3\textwidth]{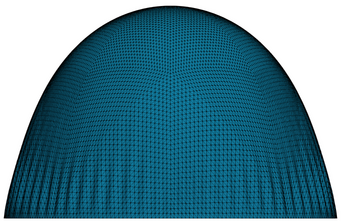}
	 		&\includegraphics[width=0.3\textwidth]{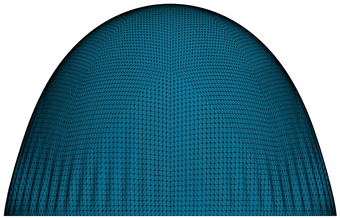}&\includegraphics[width=0.3\textwidth]{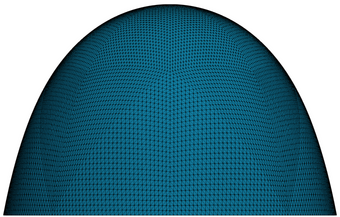}\\
	 		\multicolumn{3}{c}{model from Chapter~\ref{sec:continuous_model}} \\
	 		\hline \\
	 		\includegraphics[width=0.3\textwidth]{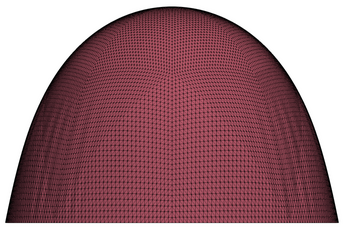}
	 		&\includegraphics[width=0.3\textwidth]{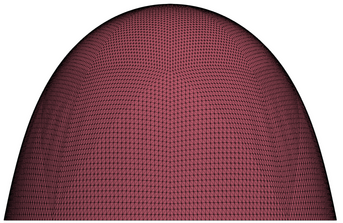}
	 		&\includegraphics[width=0.3\textwidth]{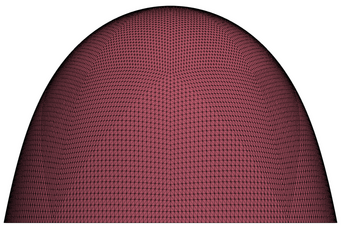}
	 		\\
	 		\multicolumn{3}{c}{alternative model described in Remark~\ref{rem:birsan_energy}} \\
	 		\hline
	 	\end{tabular}
	\caption{Deflection of a half sphere under a vertical tensile load of $\mathbf{f} = 1.5 \cdot h \cdot (0,0,10^4)\,\mathsf{M}/(\mathsf{L}\mathsf{T}^{2})$. The pictures show
	the two Cosserat shell models for different choices of the thickness $h$.}
	\label{fig:comparison}
\end{center}
\end{figure}

\begin{figure}
	\begin{tikzpicture}
		\begin{axis}[height=0.5\textheight,width=0.9\textwidth,xtick = data, xticklabels={{\small $10^{-4}\,\mathsf{L}$},{\small $10^{-3}\,\mathsf{L}$},{\small $10^{-2}\,\mathsf{L}$},{\small $0.1\,\mathsf{L}$},{\small $0.2\,\mathsf{L}$},{ \small $0.4\,\mathsf{L}$}}, xticklabel pos=lower]
	 		\addplot[color=mediumblue,very thick, dashed] table[col sep=comma,header=true,x index=1,y index=2] {model-comparison.csv};
	 		\addplot[color=birsanblue,very thick, dotted] table[col sep=comma,header=true,x index=1,y index=3] {model-comparison.csv};
	 		\addplot[color=3dcolor,very thick] table[col sep=comma,header=true,x index=1,y index=4]
	 		{model-comparison.csv};
			\addplot[color=mediumblue,very thick, dashed] table[col sep=comma,header=true,x index=1,y index=5] {model-comparison.csv};
			\addplot[color=birsanblue,very thick, dotted] table[col sep=comma,header=true,x index=1,y index=6] {model-comparison.csv};
			\addplot[color=3dcolor,very thick] table[col sep=comma,header=true,x index=1,y index=7] {model-comparison.csv};
			\addplot[color=mediumblue,very thick, dashed] table[col sep=comma,header=true,x index=1,y index=8] {model-comparison.csv};
			\addplot[color=birsanblue,very thick, dotted] table[col sep=comma,header=true,x index=1,y index=9] {model-comparison.csv};
			\addplot[color=3dcolor,very thick] table[col sep=comma,header=true,x index=1,y index=10] {model-comparison.csv};
			\addplot[color=mediumblue,very thick, dashed] table[col sep=comma,header=true,x index=1,y index=11] {model-comparison.csv};
			\addplot[color=birsanblue,very thick, dotted] table[col sep=comma,header=true,x index=1,y index=12] {model-comparison.csv};
			\addplot[color=3dcolor,very thick] table[col sep=comma,header=true,x index=1,y index=13] {model-comparison.csv};
			\node[draw=black,fill=white] at (axis cs:2.12,0.25) {\tiny$\mathbf{f} = 1.5 \cdot h \cdot (0,0,10^4)\,\mathsf{M}/(\mathsf{L}\mathsf{T}^{2})$};
			\node[draw=black,fill=white] at (axis cs:2.12,0.163) {\tiny$\mathbf{f} = 1.0 \cdot h \cdot (0,0,10^4)\,\mathsf{M}/(\mathsf{L}\mathsf{T}^{2})$};
			\node[draw=black,fill=white] at (axis cs:2.12,0.084) {\tiny$\mathbf{f} = 0.5 \cdot h \cdot (0,0,10^4)\,\mathsf{M}/(\mathsf{L}\mathsf{T}^{2})$};
			\node[draw=black,fill=white] at (axis cs:2.12,0.027) {\tiny$\mathbf{f} = 0.2 \cdot h \cdot (0,0,10^4)\,\mathsf{M}/(\mathsf{L}\mathsf{T}^{2})$};
			\node at (axis cs:3.41,0.27){\tiny\color{mediumblue}main model};
			\node at (axis cs:3.6,0.305){\tiny\color{birsanblue}alternative model};
			\node at (axis cs:3.33,0.284){\tiny\color{3dcolor}3D-model};
		\end{axis}
	\end{tikzpicture}
\caption{Vertical displacement of the half-sphere north pole as a function of the shell thickness
for four different volume loads. ``{\color{mediumblue}Main model}'' refers to the model
introduced in this manuscript (with the factor $\frac{\mu + \mu_c}{2}$ in the transverse
shear coefficient),
whereas ``{\color{birsanblue}alternative model}'' refers to the variant of~\textcite{birsan:2021}
discussed in Remark~\ref{rem:birsan_energy} (with the factor replaced by
$\frac{2\mu\mu_c}{\mu + \mu_c}$).
``{\color{3dcolor}3D-model}'' is the three-dimensional Cosserat bulk model of~\cite{neff_dim_reduction:2019}
that both shell models are derived from.
}
\label{fig:pole_displacement}
\end{figure}
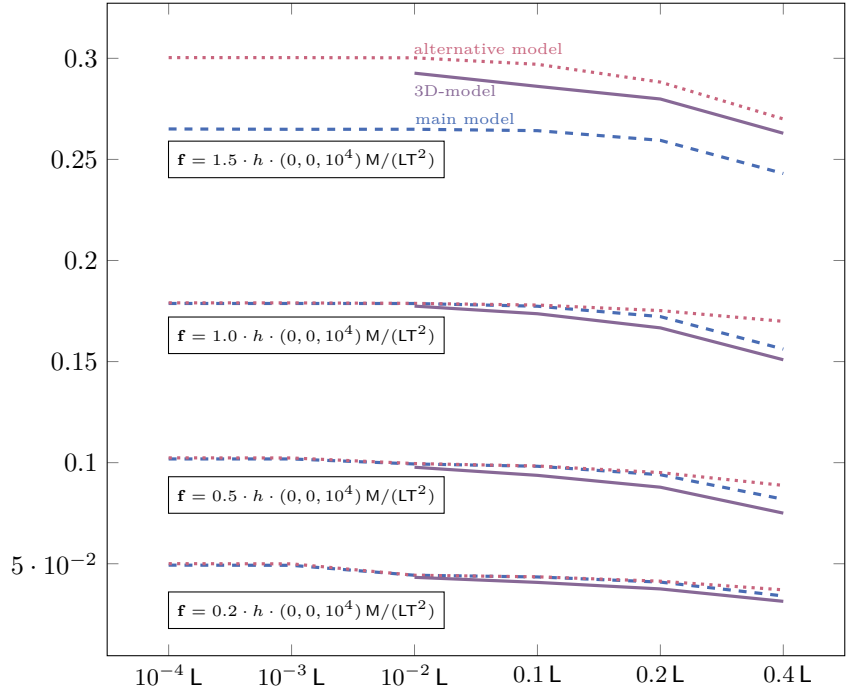

Results for $h = 10^{-4}\,\mathsf{L}$, $10^{-3}\,\mathsf{L}$, and $10^{-2}\,\mathsf{L}$,
and the load $\mathbf{f} = 1.5 \cdot h \cdot (0,0,10^4)\,\mathsf{M}/(\mathsf{L}\mathsf{T}^2)$
are shown in Figure~\ref{fig:comparison}.
We see that both models behave similarly.
For $h = 10^{-4}\,\mathsf{L}$ and $h = 10^{-3}\,\mathsf{L}$, both models show wrinkles,
but the wrinkles are less prominent in the model from~\cite{birsan:2021}.
For $h = 10^{-2}\,\mathsf{L}$ neither model shows wrinkles.

For a quantitative comparison we lower the resolution a little, and discretize the shell surface
using a grid of only $1580$ second-order triangles. Figure~\ref{fig:pole_displacement} shows
the vertical displacement of the north pole of the half sphere resulting from four different applied loads
as a function of the shell thickness.  The
force densities are $\mathbf{f} = \alpha \cdot h \cdot (0,0, 10^4)\,\mathsf{M}/(\mathsf{L}\mathsf{T}^2)$
with $\alpha \in \{0.2, 0.5, 1.0, 1.5\}$.
One can see that for the weaker loads and the smaller thickness values, the deflections
predicted by the two models match almost perfectly.
For thicker shells, the model from~\cite{birsan:2021} is slightly softer.
For a force density of $\mathbf{f} = 1.5 \cdot h \cdot (0,0,10^4)\,\mathsf{M}/(\mathsf{L}\mathsf{T}^{2})$
there is a noticeable difference between the two model responses. Indeed, while both models
show the same qualitative behavior as functions of $h$, the deflection for the model of
\textcite{birsan:2021} is consistently about $0.03\,\mathsf{L}$ larger than
the deflection of the main model for all values of the thickness.
Note however that the strains in this case are well beyond the limits of the linear
material law used in both shell models, and deviations can therefore be expected.

To assess the behavior of the two shell models even better we also
compare with a three-dimensional Cosserat model
with linear material response from \cite{neff_dim_reduction:2019},
the parent model that both shell models were derived from.
It considers a domain $\Omega$ in $\R^3$, and
deformation and microrotation fields
\begin{equation*}
 \bm^\textup{3D} : \Omega \to \R^3,
 \qquad
 \Qe^\textup{3D} : \Omega \to \SOdrei
\end{equation*}
that minimize
\begin{equation}\label{eq:three-dimensional-cosserat-energy}
 I_\text{3D} (\bm^\textup{3D}, \boldsymbol{Q}_e^\textup{3D})
 \colonequals
 \int_{\Omega} \Big[ W^\textup{3D}_\textup{mp}(\overline{\boldsymbol{E}}) +   W^\textup{3D}_\textup{curv}(\boldsymbol{\Gamma}) \Big]\,dV +\text{external loads},
\end{equation}
where
\begin{equation*}
 \overline{\boldsymbol{E}}
 \colonequals
 (\boldsymbol{Q}_e^\textup{3D})^T\left(\nabla \bm^\textup{3D}\right) - \identity_3
\end{equation*}
is the non-symmetric strain tensor, and
\begin{equation*}
  \boldsymbol{\Gamma} \colonequals \axl\Big((\boldsymbol{Q}_e^\textup{3D})^T \frac{\partial \boldsymbol{Q}_e^\textup{3D}}{\partial x_i}\Big) \otimes \ba^i
\end{equation*}
is the wryness tensor. The vectors $\ba^i$ are the contravariant basis vectors, of which there are three now.
The energy terms are \cite[Equation~(2.27)]{neff_dim_reduction:2019}
\begin{equation*}
 W^\textup{3D}_\textup{mp}(\overline{\boldsymbol{E}})
 \colonequals
 \mu\norm{\sym \overline{\boldsymbol{E}}}^2 + \mu_c\norm{\skew \overline{\boldsymbol{E}}}^2 + \tfrac{\lambda}{2} (\tr\overline{\boldsymbol{E}})^2
\end{equation*}
and \cite[Equation~(2.30)]{neff_dim_reduction:2019}
\begin{equation*}
 W^\textup{3D}_\textup{curv}(\boldsymbol{\Gamma})
 \colonequals
 \mu L_c^2\big(b_1\norm{\sym \boldsymbol{\Gamma}}^2 + b_2\norm{\skew \boldsymbol{\Gamma}}^2 + (b_3 - \tfrac{b_1}{3}) (\tr\boldsymbol{\Gamma})^2 \big).
\end{equation*}
With our choice of parameters $b_1$, $b_2$, $b_3$, the latter again reduces to
$W^\textup{3D}_\textup{curv}(\boldsymbol{\Gamma}) = \mu L_c^2 \norm{\boldsymbol{\Gamma}}^2$.

For the three-dimensional domain $\Omega$, we use a half sphere of thickness $h$,
with an outer radius of $1 + \frac{h}{2}$ and an inner radius of $1- \frac{h}{2}$.
We clamp the deformation~$\bm^\textup{3D}$ at the equator, and the microrotation function
$\boldsymbol{Q}_e^\textup{3D}$ is again not subject to Dirichlet conditions.
The domain is discretized using a three-dimensional unstructured simplex grid constructed such that
the element diameters match the mesh size of the two-dimensional grids used for the shells.
For the three-dimensional grid we use flat simplex elements, which is justified
because the three-dimensional Cosserat energy~\eqref{eq:three-dimensional-cosserat-energy}
does not contain terms involving the element geometry curvature.
As simulating very thin shells with a bulk model requires either very distorted elements or a very fine grid
we do not run the three-dimensional simulations for the thickness values $h = 10^{-4}\,\mathsf{L}$
and $h = 10^{-3}\,\mathsf{L}$.

In comparison to the shell models we see that for
the force densities $\mathbf{f} = \alpha \cdot h \cdot (0,0, 10^4)\,\mathsf{M}/(\mathsf{L}\mathsf{T}^2)$,
$\alpha \in \{ 0.2, 0.5, 1.0\}$, the main model approximates the three-dimensional model slightly better. For the thinner shells ($h = 0.01\,\mathsf{L}$ and $h = 0.1\,\mathsf{L}$), both two-dimensional models match the three-dimensional model almost perfectly.
For the force density $\mathbf{f} = 1.5\cdot h \cdot (0,0,10^4)\,\mathsf{M}/(\mathsf{L}\mathsf{T}^{2})$,
the response of the three-dimensional model lies between the shell two models, showing a similar qualitative behavior as a functions of $h$.
The question of whether one of the shell models is better in this regime
remains undecided.

\subsection{Shells with challenging topology}

\begin{figure}
 \begin{center}
\includegraphics[width=0.45\textwidth]{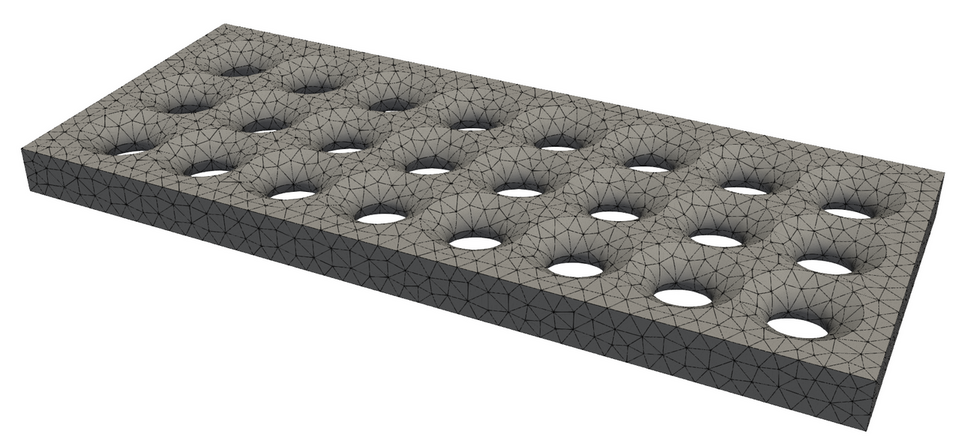}\quad \includegraphics[width=0.45\textwidth]{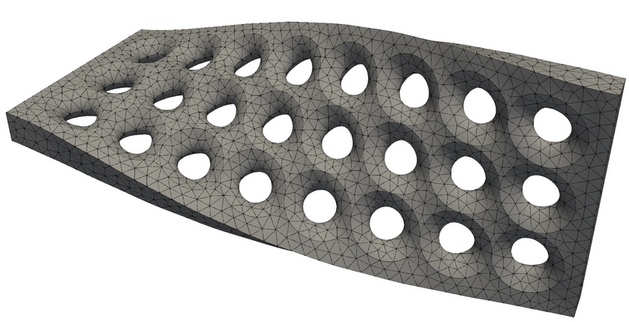}\\ \includegraphics[width=0.45\textwidth]{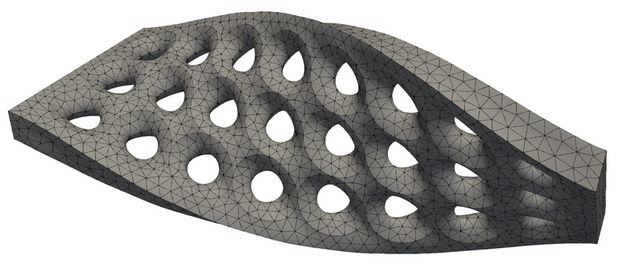}\quad \includegraphics[width=0.45\textwidth]{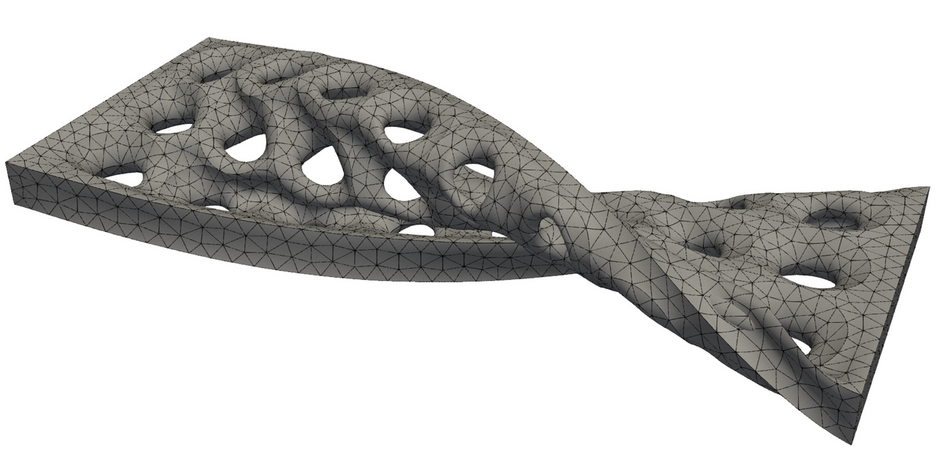}
 \end{center}

 \caption{Deforming the boundary of a rectangular block perforated by 24 holes.
   The boundary is a closed two-dimensional surface of genus~24.  It is
   displacement-loaded to twist around its long axis.
   Note that this object is hollow---we are only simulating the boundary.}
 \label{fig:perforated_plate}
\end{figure}

The next example demonstrates that the shell model and its discretization can handle
surfaces with a challenging topology. For this we consider a rectangular block of dimensions
$8.5\,\mathsf{L} \times 3.5\,\mathsf{L} \times 0.5\,\mathsf{L}$, from which material
has been removed to leave 24 rotation-symmetric holes.  The boundary of the block
is a closed two-dimensional surface of genus~24. We use this surface as the
reference configuration of a Cosserat shell.  The discrete shell model is shown
in the top left picture of Figure~\ref{fig:perforated_plate}---keep in mind that
what looks like a solid three-dimensional object is actually hollow.

For the material behavior we use the same parameters as in the
previous examples, i.e., $\lambda = 4.4364\cdot10^4\,\mathsf{M}/(\mathsf{L}\mathsf{T}^{2})$
and $\mu = 2.7191\cdot10^4\,\mathsf{M}/(\mathsf{L}\mathsf{T}^{2})$ for the Lamé parameters,
$\mu_c = 0.1 \mu$ for the Cosserat couple modulus,
$L_c = 5\cdot10^{-4}\,\mathsf{L}$ for the internal length,
and $b_1 = b_2 = 1$, $b_3 = \tfrac{1}{3}$. We choose a shell thickness of $h=0.05\,\mathsf{L}$.
We discretize the model and its reference configuration $\bm_0$ with a grid
consisting of 5518 second-order triangle elements.  For the deformation $\bm$ we use
second-order Lagrange finite elements, and for the microrotation we use first-order
geodesic finite elements.

To drive the model away from the reference configuration we do not apply a volume load.
Rather, we clamp the deformation and the microrotation at the left short end and apply a displacement load
on the right end.  More specifically, we twist the deformation and microrotation
of the right end to a rotation of $\pi$ around the long axis of the object.

To reach this state we additionally
solve the algebraic minimization Problem~\ref{prob:algebraic-cosserat-shell-problem}
at two intermediate load steps at $\frac{\pi}{2}$ and $\frac{3\pi}{4}$,
starting always from the solutions of the previous load steps.
Figure~\ref{fig:perforated_plate} shows the deformed configurations. One can observe
the uniform twisting expected for this type of load. The finite element model seems
to handle the large rotations without any difficulties.

\subsection{Buckling of a cylinder under a torsion load}

With the next example we demonstrate that the model and discretization can represent
extreme buckling situations.  For this we consider the deformation of a cylinder under a torsion load.
Let $\omega$ be a topological cylinder, with an immersion $\bm_0$
into $\R^3$ as a cylinder of radius~$10\,\mathsf{L}$ and height~$15\,\mathsf{L}$, centered around the $x_3$-axis.
We choose the same material parameters as in the previous examples,
i.e., $\lambda = 4.4364\cdot10^4\,\mathsf{M}/(\mathsf{L}\mathsf{T}^{2})$
and $\mu = 2.7191\cdot10^4\,\mathsf{M}/(\mathsf{L}\mathsf{T}^{2})$ for the Lamé parameters,
$\mu_c = 0.1 \mu$ for the Cosserat couple modulus, $b_1 = b_2 = 1$, $b_3 = \tfrac{1}{3}$, and the internal length
parameter $L_c = 5 \cdot 10^{-4}\,\mathsf{L}$. The thickness is $h= 0.05\,\mathsf{L}$.

We clamp the cylinder deformation and microrotation at a circular strip of width~$3\,\mathsf{L}$
at the bottom, and we apply a torsion displacement load at an identical strip at the top
of the cylinder.
More formally, the Dirichlet set is
\begin{equation*}
 \gamma_d
 =
 \Big\{ \eta \in \omega \; : \; (\bm_0(\eta))_3 \le 3 \quad \text{or} \quad (\bm_0(\eta))_3 \ge 12 \Big\},
\end{equation*}
and we apply the torsion load to all points $\eta \in \omega$ with $(\bm_0(\eta))_3 \ge 12$
by prescribing the microrotation and the $x_1$- and $x_2$-component of the deformation
(but not the $x_3$-component). The respective deformation and microrotation functions are
\begin{equation*}
 \bm_\alpha^* = \bm_\alpha^{3D} \circ \bm_0 :\omega \to \R^3
 \quad\text{and}\quad
 (\Qe)_\alpha^* = (\Qe)_\alpha^{3D} \circ \bm_0 :\omega \to \SOdrei
\end{equation*}
with
\begin{equation*}
 \bm_\alpha^{3D}(x)
  = \begin{pmatrix}x_1\\x_2\\x_3 \end{pmatrix},
  \qquad
 (\Qe)_\alpha^{3D}(x) = \identity_3
\end{equation*}
for all $x \in \R^3$ with $x_3 \le 3$ and
\begin{equation*}
 \bm_\alpha^{3D}(x)
  =
 \begin{pmatrix} x_1\cdot \cos \alpha - x_2 \cdot \sin \alpha \\
  x_1 \cdot \sin \alpha +  x_2 \cdot \cos \alpha \\
  \star
 \end{pmatrix},\quad
  (\Qe)^{3D}_\alpha(x) = \begin{pmatrix} \cos \alpha & - \sin\alpha & 0\\
 	\sin \alpha &\cos \alpha & 0\\
 	0 &0&1 \end{pmatrix}
\end{equation*}
for all $x \in \R^3$ with $x_3 \ge 12$.
The functions $\bm_\alpha^*$ and $(\Qe)_\alpha^*$ represent a counterclockwise rotation of angle~$\alpha$ about the $x_3$-axis.
The $\star$ denotes that the corresponding deformation component is not prescribed.
\begin{figure}
\begin{tabular}{c c c}
	\includegraphics[width=0.25\textwidth]{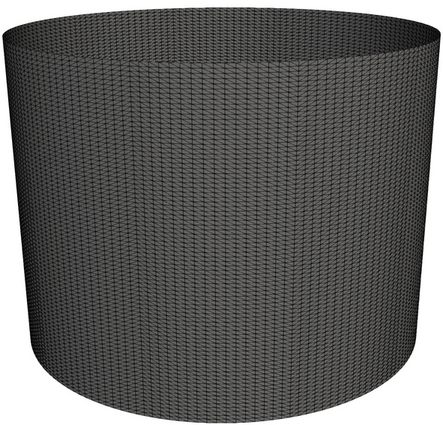} & \includegraphics[width=0.25\textwidth]{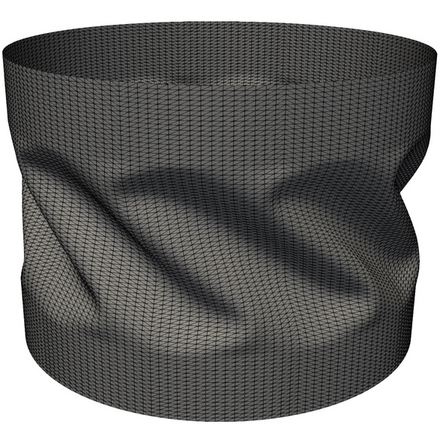} & \includegraphics[width=0.25\textwidth]{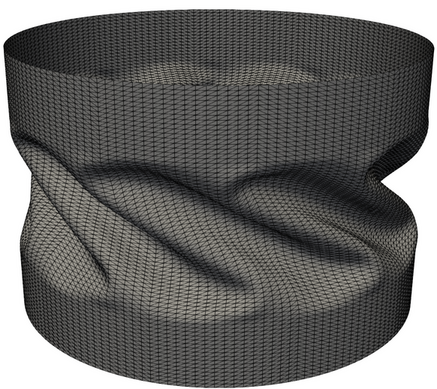} 
	\\\\
	$\alpha = 0\cdot 2\pi$ & $\alpha = -\tfrac{1}{64}\cdot 2\pi$ & $\alpha = -\tfrac{2}{64}\cdot2\pi $ 
	\\
	\includegraphics[width=0.25\textwidth]{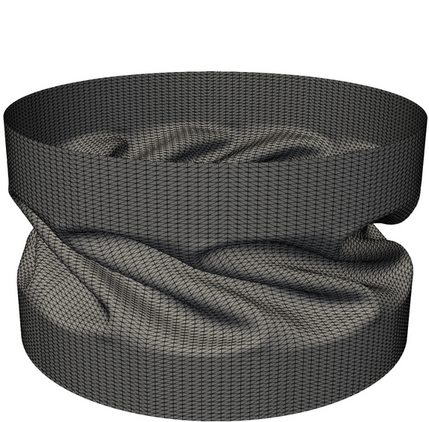} & \includegraphics[width=0.25\textwidth]{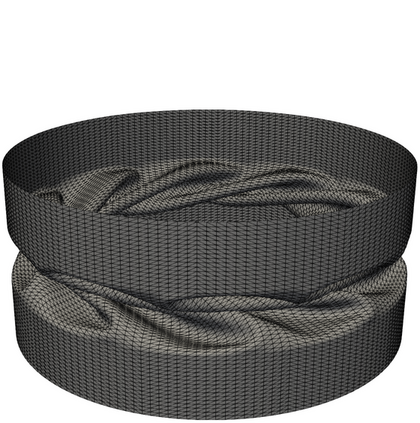} & \includegraphics[width=0.25\textwidth]{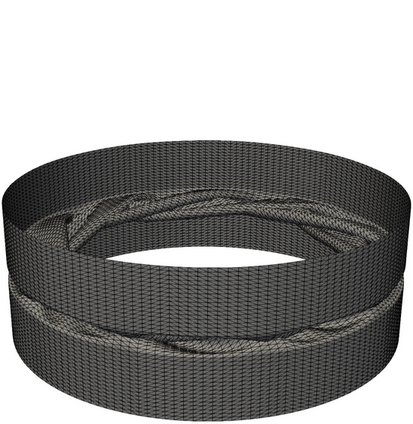}
	\\\\$\alpha = -\tfrac{3}{64}\cdot2\pi$ & $\alpha = -\tfrac{4}{64}\cdot2\pi$ & $\alpha = -\tfrac{5}{64}\cdot2\pi$ 
\end{tabular}
	\caption{Clamped cylinder with torsion displacement loads}
	\label{fig:cylinder}
\end{figure}

\begin{figure}
 \includegraphics[width=0.8\textwidth]{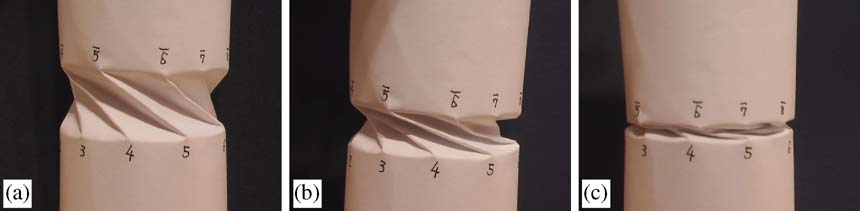}
 \caption{Paper-made cylinder under a torsion load, from~\cite{hunt_ario:2005}}
 \label{fig:twist_buckling_photo}
\end{figure}

\begin{figure}
 \begin{center}
  \begin{minipage}{0.5\textwidth}
   \begin{tikzpicture}
    \begin{axis}[height=0.25\textheight,
                 width=\textwidth,
                 xtick = data,
                 xticklabels={{\small $\frac{0\pi}{64}$},{\small $-\frac{2\pi}{64}$},{\small $-\frac{4\pi}{64}$},{\small $-\frac{6\pi}{64}$},{\small $-\frac{8\pi}{64}$},{ \small $-\frac{10\pi}{64}$}},
                 xticklabel pos=lower,
                 xlabel=torsion angle,
                 ytick = {7,9,11,13,15},
                 ylabel=height,
                 ylabel near ticks]

                 \addplot[color=mediumblue,very thick, mark=*] table[col sep=comma,header=true,x index=0,y index=1]
                 {cylinder-height.csv};
    \end{axis}
   \end{tikzpicture}
  \end{minipage}%
 \qquad
 \begin{minipage}{0.4\textwidth}
\begin{tabular}{c|c}
 torsion angle & height [$\mathsf{L}$] \\[0.007\textheight]
 \hline
 0 & 15.0000 \\[0.007\textheight]
 $-\frac{1}{64} \cdot 2\pi$ & 14.1329 \\[0.007\textheight]
 $-\frac{2}{64} \cdot 2\pi$ & 12.8815 \\[0.007\textheight]
 $-\frac{3}{64} \cdot 2\pi$ & 11.2033 \\[0.007\textheight]
 $-\frac{4}{64} \cdot 2\pi$ &  8.82082 \\[0.007\textheight]
 $-\frac{5}{64} \cdot 2\pi$ &  6.98778
\end{tabular}
 \end{minipage}
\end{center}
\caption{{Height of the loaded cylinder as a function of the torsion angle}}
\label{fig:height_under_load}
\end{figure}
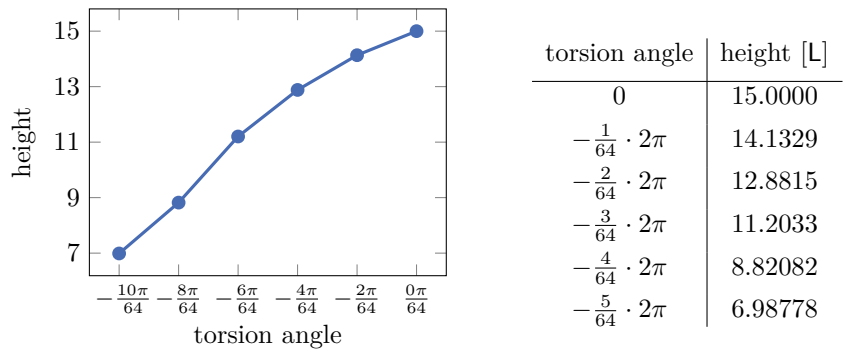

We discretize the domain using $80 \times 160 \times 2 = 25\,600$ triangular elements.
For the reference deformation $\bm_0$ and the loaded deformation $\bm$ we use second-order
Lagrange finite elements, and we use first-order geodesic finite elements for the
microrotation $\Qe$. Starting from $\alpha=0$ we load the structure in increments of $-\frac{1}{64} \cdot 2\pi$,
up to a load of $\alpha = -\tfrac{5}{64}\cdot2\pi$, to obtain a clockwise rotation.
The algebraic minimization problems are solved
again starting from the solutions of the previous load steps.

Figure~\ref{fig:cylinder} shows the results. The cylinder shows the expected
buckling behavior with evenly distributed folds. At the final load the deformable part of the
structure is reduced to almost zero height.
{Figure~\ref{fig:height_under_load} gives exact values for the cylinder height at the
different loading steps.  These were obtained by averaging all $x_3$-values of Lagrange points
on the upper boundary of the cylinder.}
To {further} highlight the quality of our simulations
we compare with a result of \textcite{hunt_ario:2005}, where a similar torsion problem
for a paper-made cylinder has been treated experimentally (Figure~\ref{fig:twist_buckling_photo}).
In comparison, our simulation shows features very similar to the sharp folds exhibited there.
Note, however, that paper deformations
are largely isometric, and membrane strains do not occur.  A shell model that,
unlike ours, does not support membrane deformations could therefore be more appropriate
for the modeling of such scenarios.

\subsection{Non-orientable shell surfaces}
\label{sec:non-orientable_shell surfaces}

In this final section we show simulations of non-orientable shells. It has been argued
in Chapter~\ref{sec:continuous_model} that our shell model is well-defined if $\omega$ is non-orientable
even though the original derivation in~\cite{neff_dim_reduction:2019,neff_derivation:2020}
contained terms that are orientation-dependent. We now reconfirm this assertion
of Chapter~\ref{sec:continuous_model} numerically, by testing with two classic examples
of non-orientable surfaces, namely the Möbius strip and the Klein bottle.

\subsubsection{The Möbius strip}

A Möbius strip with radius~$r$ and width~$t$ can be parametrized over the set
\begin{equation*}
 \omega = \Big[ -\frac{t}{2},\frac{t}{2}\Big] \times [0,2\pi],
\end{equation*}
with the following identification of points
\begin{equation*}
 (u,0) = (-u,2\pi)
 \qquad
 \forall u \in \Big[ -\frac{t}{2},\frac{t}{2}\Big].
\end{equation*}
A continuous reference deformation is then
\begin{equation*}
\bm_0 : \omega \to \R^3,
\qquad
\bm_0(u,v) \:=\:\begin{pmatrix}
 \big(r + u \cos \frac{v}{2}\big) \cos v \\
 \big(r + u \cos \frac{v}{2}\big) \sin v \\
u \sin \frac{v}{2}
\end{pmatrix}.
\end{equation*}
We use the values $r=3\,\mathsf{L}$ and $t = 4.5\,\mathsf{L}$.

We choose the same material parameters as for the previous examples, i.e.,
$\lambda = 4.4364\cdot10^4\,\mathsf{M}/(\mathsf{L}\mathsf{T}^{2})$
and $\mu = 2.7191\cdot10^4\,\mathsf{M}/(\mathsf{L}\mathsf{T}^{2})$ for the Lamé parameters,
$\mu_c = 0.1 \mu$,
$L_c = 5 \cdot 10^{-4}\,\mathsf{L}$, and $b_1 = b_2 = 1$, $b_3 = \tfrac{1}{3}$.
The thickness is $h= 0.05\,\mathsf{L}$.
We clamp the deformation and microrotation at all points $\eta \in \omega$
fulfilling $(\bm_0(\eta))_1 \ge 2.5$,
and we load all points fulfilling $(\bm_0(\eta))_1 \le -2.5$ with a constant volume force density of
$\mathbf{f} = h \cdot (6,0,0)\,\mathsf{M}/(\mathsf{L}\mathsf{T}^{2})$.

The Möbius strip reference deformation is discretized by a uniform grid of
$23 \times 120 \times 2 = 5520$ second-order triangles that interpolate $\bm_0$.
We use second-order Lagrange finite elements for the deformation $\bm_h$ and
second-order geodesic finite elements for the microrotation field $\Qeh$.

\begin{figure}
	\includegraphics[width=0.28\textwidth]{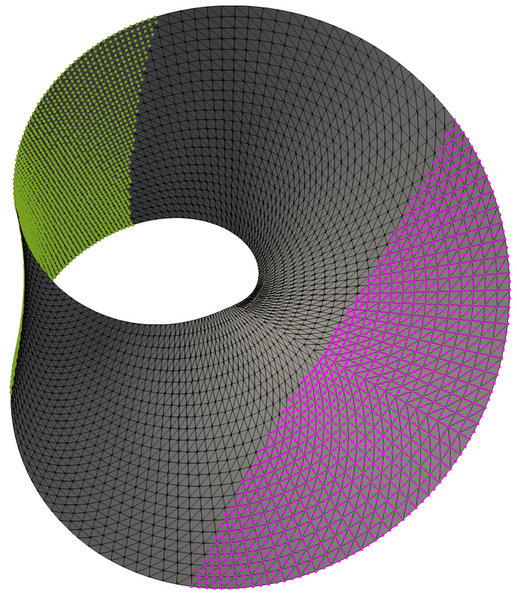}\quad\includegraphics[width=0.33\textwidth]{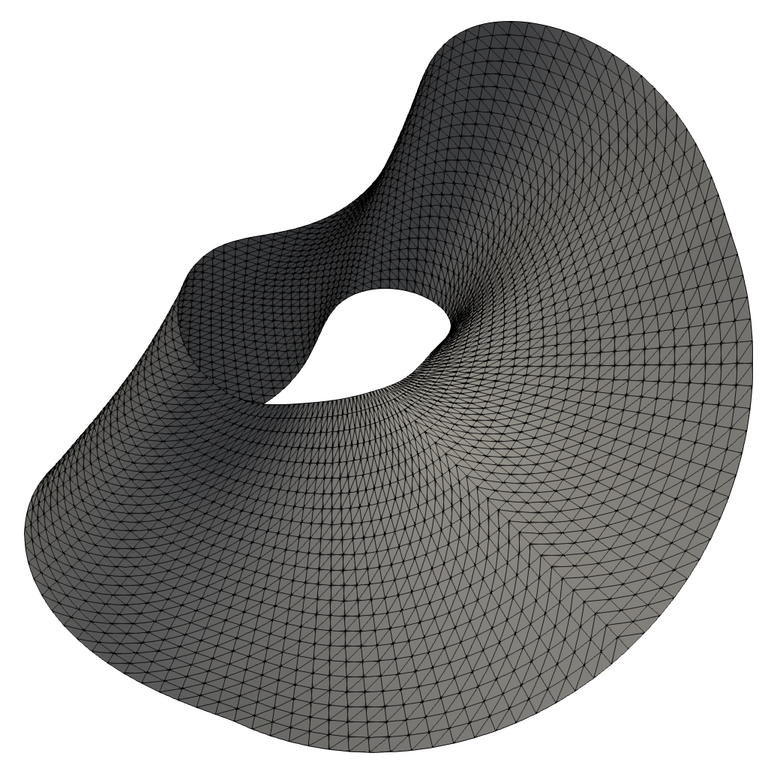}\quad\includegraphics[width=0.33\textwidth]{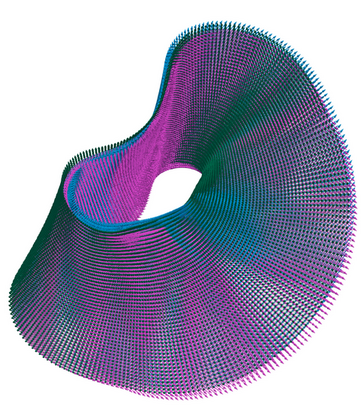}
	\caption{Stress-free configuration of the Möbius strip (left),
	with {\color{mediumgreen} nodes affected by the volume load (green)} and {\color{neonpink} Dirichlet nodes (pink)};
	the deformed configuration of the Möbius strip (center);
	and the discrete microrotation $\Qeh$ as an orthonormal frame of directors
	$({\color{neonpink}d_1}\,\vert\,{\color{lightblue}d_2} \,\vert\,{\color{mediumgreen}d_3})$ (right)}
	\label{fig:moebius}
\end{figure}

Figure~\ref{fig:moebius} shows the reference deformation with Dirichlet set and volume load
on the left, and the deformed configuration in the center.
To show how the non-orientability of the surface interacts with the microrotation field $\Qe$,
the right part of
Figure~\ref{fig:moebius} shows the discrete microrotation $\Qeh$ for the loaded Möbius strip
drawn as an orthonormal frame of director vectors. One sees that they are not in any obvious
relationship to the shell surface, in particular the third director is never normal
to the surface (and rarely ever close to it).
The microrotation $\Qeh$ forms a continuous field on~$\omega$ even though $\omega$
cannot be oriented.  This complements the
interpretation given in Chapter~\ref{sec:interpretation}: The microrotation $\Qe$ is not the
absolute local orientation of the shell, but rather the rotation that moves the reference
local orientation~$\bQ_0$ into the local orientation under load. As such a relative quantity,
it is independent of the orientability of the shell surface.

\subsubsection{The Klein bottle}

A Klein bottle can be constructed using the parameter domain
\begin{equation*}
 \omega = [0,2\pi] \times [0,2\pi]
\end{equation*}
with the identifications
\begin{equation*}
 (0,v) = (2\pi,v) \quad \forall v \in [0,2\pi]
 \qquad \text{and} \qquad
 (u,0) = (2\pi-u,2\pi)
 \quad
 \forall u \in [2,\pi].
\end{equation*}
A suitable reference deformation is then
\begin{align*}
\bm_0(u,v)
& =
\begin{pmatrix}
r \,(1-\sin u )\,\cos u +(2-\cos u)\,\cos v\,(2\,e^{-(u/2-\pi )^{2}}-1)
\\(2-\cos u)\,\sin v
\\t + t\,\sin u + \tfrac{2-\cos u}{2}\,\sin u \cos v \,e^{-(u-3\pi /2)^{2}}
\end{pmatrix},
\end{align*}
with $r = 1.5$ and $t = 5$.
As is well-known, this is not an injective map into $\R^3$.  However, recall that
the global injectivity of the stress-free deformation $\bm_0$ is not required for
the shell model (Chapter~\ref{sec:shell surface_geometry}).  From a mechanical point
of view there is no interaction between different intersecting sheets---they slide
right through each other.

As previously, we choose the material parameters
$\lambda = 4.4364\cdot10^4\,\mathsf{M}/(\mathsf{L}\mathsf{T}^2)$
and $\mu = 2.7191\cdot10^4\,\mathsf{M}/(\mathsf{L}\mathsf{T}^{2})$ for the Lamé parameters,
$\mu_c = 0.1\mu$, $L_c = 5 \cdot 10^{-4}\,\mathsf{L}$ for the internal length,
and $b_1 = b_2 = 1$, $b_3 = \tfrac{1}{3}$.
The thickness is $h= 0.05\,\mathsf{L}$.
To eliminate the rigid-body modes we clamp the displacement at all points
$\eta \in \omega$ with $(\bm_0(\eta))_1 \le -1.5$. The microrotation is not subject to Dirichlet conditions at all.
We then load the Klein bottle using a constant volume force density
$\mathbf{f} = h \cdot (10^3,0,0)\,\mathsf{M}/(\mathsf{L}\mathsf{T}^{2})$.

We discretize the Klein bottle reference deformation by a structured grid of
second-order triangles with $96 \times 128 \times 2 = 24\,576$ triangles.
For the deformation $\bm$ we use
second-order Lagrange finite elements, and for the microrotation we use first-order
geodesic finite elements.

\begin{figure}
	\includegraphics[height=0.3\textheight]{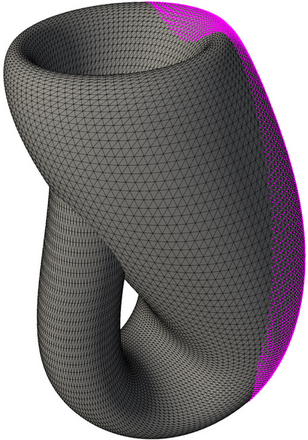}
	\hspace{0.1\textwidth}
	\includegraphics[height=0.3\textheight]{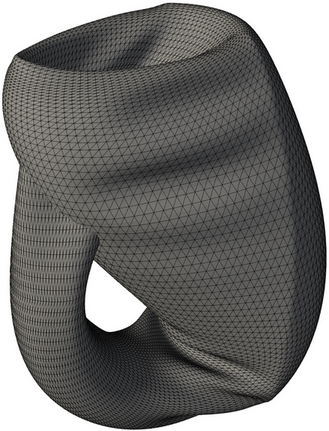}
	\caption{Stress-free configuration of the Klein bottle with Dirichlet nodes marked in purple (left),
	and deformed configuration under a volume load in the $x_1$-direction (right)}
	\label{fig:klein-bottle}
\end{figure}

Figure~\ref{fig:klein-bottle} shows the grid for the reference configuration, and the deformed configuration under load.
One can see the folds expected from such a loading scenario, and one can observe
that the thin handle does indeed move through the thicker right part without
any apparent resistance. Again, the non-orientability of the shell surface
does not pose any problem to the shell model and its discretization.

\printbibliography

\end{document}